\documentclass[12pt, reqno]{amsart}
\usepackage{amsmath, amsthm, amscd, amsfonts, amssymb, graphicx, color}

\newtheorem{df}{Definition}[section]
\newtheorem{thm}[df]{Theorem}
\newtheorem{pro}[df]{Proposition}

\newtheorem{rema}[df] {Remark}

\begin{document}
\setcounter{page}{1}

\title[Joint spectra of the tensor product representation]{ Joint spectra of the tensor product representation
           of the direct sum of two solvable Lie algebras}

\author{Enrico Boasso}


\begin{abstract}Given two complex Banach spaces $X_1$ and $X_2$,
a tensor product $X_1\tilde{\otimes} X_2$  of $X_1$ 
and $X_2$ in the
sense of [14], two
complex solvable finite dimensional Lie algebras $L_1$ and
$L_2$, and two representations $\rho_i\colon L_i\to
{\rm L}(X_i)$ of the algebras, 
$i=1$, $2$, we consider the Lie algebra
$L=L_1\times L_2$, and the tensor product representation of $L$, 
$\rho\colon L\to
{\rm L}(X_1\tilde{\otimes}X_2)$, $\rho=\rho_1\otimes I
+I\otimes \rho_2$. In this work we study the 
S\l odkowski and the split joint spectra of the representation 
$\rho$, and we describe them in terms of 
the corresponding joint spectra of $\rho_1$ and $\rho_2$. 
Moreover, we study the essential 
S\l odkowski and the essential split joint
spectra of the representation $\rho$, and we describe 
them by means of the corresponding
joint spectra and the corresponding essential joint spectra of $\rho_1$ and $\rho_2$. 
In addition, with similar arguments we describe
all the above-mentioned joint spectra  
for the multiplication representation in an operator ideal between 
Banach spaces in the sense of [14]. 
Finally, we consider nilpotent  
systems of operators, in particular commutative, and we apply 
our descriptions to them. 
\end{abstract}

\maketitle
\tableofcontents

\section{Introduction}

\indent In this work we deal with several joint spectra defined for
representations of complex solvable finite dimensional Lie algebras in
complex Banach spaces. Our main concern is to study 
the behavior of some joint spectra with respect to the procedure 
of passing from two given such representations, $\rho_1\colon L_1
\to L(X_1)$ and $\rho_2\colon L_2\to L(X_2)$, 
to the tensor product representation of the direct sum of the algebras,
$\rho\colon L_1\times L_2\to L(X_1\tilde{\otimes} X_2)$, $\rho=\rho_1
\otimes I+I\otimes \rho_2$, where $X_1\tilde{\otimes}X_2$ is a 
tensor product of the Banach spaces 
$X_1$ and $X_2$ in the sense of [14], and
$I$ denotes the identity operator of both $X_1$ and $X_2$.
In addition, we describe the spectral contributions of  
$\rho_1$ and $\rho_2$ to some joint spectra
of the multiplication 
representation $\tilde{\rho}\colon L_1\times L_2^{op}
\to L(J)$, $\tilde{\rho}(T)=\rho_1(l_1)T+T\rho_2(l_2)$,
where $J\subseteq L(X_2,X_1)$ is an operator ideal
between the Banach spaces $X_1$ and $X_2$ in the sense of [14],
and $L_2^{op}$ is the opposite algebra of $L_2$.
However, in order to accurately present
the problems we are concerned with, we review
how the theory of tensor product is placed within the general theory of 
joint spectra. We first recall some of the best known
joint spectra in the commutative and non-commutative setting and 
their relation with tensor products.\par

\indent Given a commutative complex Banach algebra ${\mathcal A}$ 
with unit element $I$, if $a=(a_1, \ldots ,a_n)\in{\mathcal A}^n $, 
$n\ge 1$, then the joint spectrum of $a$ is defined by 

\begin{align*}
\sigma_{\mathcal A} (a)=&\{(\lambda_1,\ldots ,\lambda_n)\in\Bbb C^n: 
\hbox{ the elements $a_i-\lambda_iI$, $i=1,\ldots ,n,$
generate}\\
 &\hbox{a proper ideal in ${\mathcal A}$}\}.\\\end{align*}

Another well-known formula giving the same set is 
$$
\sigma_{\mathcal A}(a)=\{(f(a_1),\ldots ,f(a_n))\in\Bbb C^n:\hbox{}f\in{\mathcal M}
({\mathcal A})\},
$$
where ${\mathcal M}({\mathcal A})$ is the set of all non-zero multiplicative 
linear functionals of ${\mathcal A}$. \par

\indent The joint spectrum 
$\sigma_{\mathcal A}(a)$ is always a non-void compact subset of 
$\Bbb C^n$. Moreover, the joint spectrum is a fundamental concept
in the theory of commutative Banach algebras, for it provides an
analytic functional calculus for several elements in such an algebra; see
[22], [3], [26] and [2]; for a general account
of the joint spectrum see
[11] and [17].\par     

\indent When ${\mathcal A}$ is a non-commutative unital Banach algebra,  
say ${\mathcal A}=L(X)$, where $X$ is a Banach space, 
one could define the joint spectrum of a commutative $n$-tuple $a=(a_1,\ldots ,a_n)$
in ${\mathcal A}$ 
as the joint spectrum of $a$
relative to a maximal abelian subalgebra ${\mathcal B}$ 
containing $a_1,\ldots , a_n$, $\sigma_{\mathcal B}(a)$.
Unfortunately, the joint spectrum so defined depends very strongly
on the choice of ${\mathcal B}$. Indeed, if we consider 
two maximal abelian subalgebras ${\mathcal B}_1$ and ${\mathcal B}_2$ 
containing $a_i$, $i=1,\ldots n$, unlike the case $n=1$ it is not
generally true that $\sigma_{{\mathcal B}_1}(a)=\sigma_{{\mathcal B}_2}(a)$,
see [1].\par

\indent So far we have considered a Banach algebra convention, i.e.,
all concepts are related to a Banach algebra ${\mathcal A}$. However,
there is another way to introduce joint spectra, the so-called spatial
convention, i.e., the joint spectra are defined for tuples of
commuting operators in the algebras $L(X)$, 
$X$ a Banach space, and in the definitions elements of $X$ are involved.
For a given Banach algebra ${\mathcal A}$, we put $X={\mathcal A}$
and interpret the elements of ${\mathcal A}$ as operators of left multiplication,
i.e., to $a\in{\mathcal A}$ we associate the map $L_a\in L(A)$,
where $L_a(b)=a.b$, $b\in {\mathcal A}$. 
Thus, a joint spectrum defined for commutative tuples of Banach
space operators, $\sigma(\cdot)$,
gives rise to a joint spectrum on ${\mathcal A}$, 
$\sigma(a, {\mathcal A})=
\sigma(L_a)$, where $L_a=(L_{a_1},\ldots ,L_{a_n})$
and $a=(a_1,\ldots ,a_n)$ is a commutative tuple in ${\mathcal A}$.\par

\indent Among the most important joint spectra defined in the
spatial convention, we have the Taylor joint spectrum; see [24] and [11]. 
This joint spectrum is
defined for commuting systems of Banach space operators $T=(T_1,\ldots ,T_n)$,
and it has the advantage that its definition 
depends on the action of the maps
$T_1,\ldots ,T_n$. The Taylor joint spectrum, $\sigma_T(T)$,
is a compact non-void subset of $\Bbb C^n$ and it has several 
additional important properties, such as an analytic functional calculus
and the so-called projection property.
When ${\mathcal A}$ is a commutative Banach 
algebra, if $a=(a_1,\ldots ,a_n)\in
{\mathcal A}^n$, then $\sigma_T(a,{\mathcal A})
=\sigma_{{\mathcal A}}(a)$; see [24] and [11]. 
Therefore, the joint spectrum $\sigma_{\mathcal A}(a)$ can
be thought of as the Taylor joint spectrum $\sigma_T(a,{\mathcal
 A})$.\par

\indent There are many other interesting joint spectra defined
in the spatial convention, for example, the S\l odkowski 
joint spectra, [23], the Fredholm or essential joint spectra,
[15] and [19], and the split and the essential split joint spectra, [13]. 
All these joint spectra are related to the Taylor spectrum and
have similar properties.\par

\indent On the other hand, over the last years some of the 
joint spectra originally introduced for commuting systems of operators
have been extended to the non-commutative case. Indeed,   
the Taylor, the S\l odkowski and the split joint spectra have been extended
to representations of complex solvable finite dimensional Lie algebras in 
complex Banach spaces and their main properties
have been proved; see [5], [7], [16], [20] and [21].\par 

\indent One of the most deeply studied
problems within the theory of joint spectra has been the determination of the
spectral contributions that two commuting systems of operator $S=(S_1, \ldots ,S_n)$
and $T=(T_1,\ldots ,T_m)$ defined in the Banach spaces $X_1$ and $X_2$ respectively,
make to the joint spectra of the system $(S\otimes I, I\otimes T)=(S_1\otimes I, \ldots ,
S_n\otimes I, I\otimes T_1,\ldots ,I\otimes T_m)$ defined in 
$X_1\tilde{\otimes}_{\alpha}X_2$, i.e., the completation of the algebraic tensor
product $X_1\otimes X_2$ with respect to a quasi-uniform crossnorm $\alpha$, and where the symbol $I$
stands for the identity map both in $X_1$ and $X_2$.
For example, if $X_1$ and $X_2$ are Hilbert spaces and
$X_1\overline{\otimes}X_2$ is the canonical completion of $X_1{\otimes}X_2$, then in [10]
the Taylor joint spectrum of $(S\otimes I, 
I\otimes T)$ in $X_1\overline{\otimes} X_2$ was characterized. Indeed, it was proved that
$$
\sigma_T(S\otimes I, I\otimes T)=\sigma_T(S)\times \sigma_T(T);
$$
see the related work [9]. In addition, the
results in [9] and [10] were extended in [27] and [28] 
to Banach spaces and quasi-uniform crossnorms. \par   

\indent Furthermore, in an operator ideal $J\subseteq L(X_2,X_1)$ 
between the Banach spaces $X_1$ and $X_2$, it is possible to consider tuples of left and
right multiplication: $L_S=(L_{S_1},\ldots ,L_{S_n})$ and
$R_T=(R_{T_1},\ldots ,R_{T_m})$ respectively, induced by commuting systems
of operators $S=(S_1,\ldots ,S_n)$ and $T=(T_1,\ldots ,T_m)$ 
defined in $X_1$ and $X_2$ respectively,
where $L_U(A)=UA$ and $R_V(B)=BV$, $U\in L(X_1)$, 
$V\in L(X_2)$ and $A$, $B\in J$ . However, the
tuple $(L_S,R_T)$ is
closely related to the system $(S\otimes I,
I\otimes T^{'})$; see [12], [14]. Indeed, the completion
$H\tilde{\otimes}_{\alpha}H^{'}$ of the algebraic tensor product
of a Hilbert space $H$ and its dual relative to a uniform crossnorm
$\alpha$ can be regarded as an operator ideal in $L(H)$, see [14]. As regards this
identification the operators $S_i\otimes I$ and $I\otimes T_j^{'}$
correspond to the operators $L_{S_i}$ and $R_{T_j}$ respectively,
for $i=1,\ldots ,n$ and $j=1,\ldots ,m$. In particular, the joint spectra
of $(L_S, R_T)$ are closely related to the corresponding joint spectra of
$(S\otimes I ,I\otimes T^{'})$. 
The Taylor joint spectrum and the essential joint spectrum of $(L_S,R_T)$
were studied in the works [12] and [14] in the Hilbert and Banach space
setting respectively.\par

\indent In addition, an axiomatic tensor product was introduced in [14].
This tensor product is general and rich enough to allow, on the
one hand, the description of the Taylor, the split, 
the essential Taylor and the essential split joint spectra of a
system $(S\otimes I, I\otimes T)$ defined in the
tensor product of two Banach spaces and, on the other hand,
the description of all the above-mentioned joint spectra of tuples
of left and right multiplications $(L_S, R_T)$ defined in a class of
operator ideals between Banach spaces introduced in [14]. \par

\indent Some of the main results in [9], [10], [12], [14], [27] and [28]
were extended to the non-commutative setting.    
In fact, the main result in [10] was extended in [6] to solvable Lie algebras of operators
defined in Hilbert spaces, and
in [21] the descriptions in [14] in connection with the Taylor and 
the split joint spectra
of a system $(S\otimes I, I\otimes T)$ and of a 
tuple of left and right multiplications $(L_S, R_T)$ were 
extended to the tensor product representation of the direct
sum of two solvable Lie algebras, and to the multiplication 
representation respectively; see [21, Chapter 3]. This work aims at extending
the central results in [14] and [21, Chapter 3] to other joint spectra in 
the commutative and non-commutative settings.\par  
  
\indent Indeed, one of the main objetives of this work is to 
describe, by
means of the tensor product introduced in [14], the S\l odkowski 
and the split joint spectra of the 
tensor product representation of the 
direct sum of two solvable Lie algebras, and of the multiplication 
representation in an operator ideal between
Banach in the sense of [14]; see sections 5 and 7. 
These descriptions provide an extension from the Taylor
joint spectrum and the usual split
joint spectrum to the
S\l odkowski and the split joint spectra 
of two of the main results in [21, Chapter 3] for the tensor
product introduced in [14]. 
Moreover, we consider nilpotent systems
of operators, in particular commutative, and we 
describe the S\l odkowski and the split joint spectra 
of a system $(S\otimes I , I\otimes T)$, 
and of a tuple of left and right  
multiplications $(L_S, R_T)$ in operator ideals between
Banach spaces in the sense of [14]; see section 5 and 7. 
\par  

\indent Our second main objective 
is to describe the essential S\l odkowski and the essential split joint spectra 
of the tensor product representation of the direct
sum of two solvable Lie algebras and of the multiplication 
representation in an operator 
ideal between Banach spaces in the sense of [14]; see section 6 and 7. 
These results are an extension of the
description proved in [14],
from the essential Taylor and the essential split joint spectra to the
essential S\l odkowski and the essential split joint spectra, 
and from commuting tuples
of operators to representations of solvable Lie algebras.
Furthermore, we consider nilpotent systems 
of operators and we describe the essential S\l odkowski and 
the essential split joint spectra of 
the systems  
mentioned in the last paragragh. 
\par
 
\indent However, in order to prove our second main result, we need 
to introduce the essential S\l odkowski and the essential split joint spectra of
a representation of a complex solvable finite dimensional Lie algebra in a 
complex Banach
space, and to prove the main properties of these joint
spectra; see section 3.\par  

\indent In addition, as an application, in section 8 we describe all the 
above-mentioned 
joint spectra of two particular representations of a nilpotent Lie 
algebra, one in a tensor product of Banach spaces, where the
tensor product is the one introduced in [14], and the
otherone in an operator ideal between Banach spaces
in the sense of [14].\par

\indent This work is organized as follows. In section 2 we recall
the definitions and the main properties of the 
S\l odkowski and 
the split joint spectra; we also include a little review of Lie algebras. In section 3 we introduce the essential
S\l odkowski and the essential split joint spectra, and we prove their main 
properties. In section 4 we recall the axiomatic tensor product introduced
in [14] and we prove some results needed for our main 
theorems. In section 5 we describe the S\l odkowski
and the split joint spectra of the tensor product representation
of the direct sum of two solvable Lie algebras. In section 6 we describe the essential 
S\l odkowski and the essential split joint spectra of the tensor product 
representation of the direct sum of two solvable Lie algebras. 
In section 7 we describe the S\l odkowski, the split,
the essential S\l odkowski and the essential split joint spectra 
of the multiplication representation in
an operator ideal between Banach spaces in the sense of [14].
In addition, in sections 5, 6 and 7 we consider nilpotent 
systems of operators and we obtain descriptions of the corresponding
joint spectra. Finally, in section 8, we apply our main results to 
some representations of nilpotent Lie algebras.\par  

\section{The Taylor, the S\l odkowski and the split joint spectra}

\indent In this section we review the definitions and the main  
properties of the Taylor, the S\l odkowski and the split joint spectra of a 
representation of a  
Lie algebra in a Banach space; see [24], [23], 
[13], [14], [16]
[7], [5], [20] and [21]. However, in order to develop 
a self-contained exposition to a reasonable
extent, we first review  
some well-known facts of Lie algebras used in this 
work. Since we are interested in solvable Lie algebras acting on complex
Banach spaces, we limit our review to this case; 
for a complete exposition see [8]. \par

\indent  A complex Lie algebra is a vector space over the complex 
field $\Bbb C$ provided with a bilinear bracket, named the Lie product, 
$[.,.]\colon L\times L\to L$, which complies with the Lie conditions
$$
[x,x]=0,\hskip1cm [[x,y],z]+[[y,z],x]+[[z,x],y]=0,
$$
for every $x$, $y$ and $z\in L$. The second of these equations is called 
the Jacobi identity. By $L^{op}$ we denote the opposite Lie algebra 
of $L$, i.e., the algebra that  
as a vector space coincides with $L$ and has the bracket
$[x,y]^{op}=-[x,y]=[y,x]$, for $x$ and $y\in L$.
\par

\indent An example of a Lie algebra structure is given by the algebra
of all bounded linear maps defined in a Banach
space $X$, $L(X)$, and the bracket 
$[.,.]\colon L(X)\times L(X)\to L(X)$, $[S,T]=ST-TS$, 
for $S$ and $T\in L(X)$. \par

\indent Given two Lie algebras $L_1$ and $L_2$ with Lie brackets
$[.,.]_1$ and $[.,.]_2$ respectively, a morphism of Lie algebras
$H\colon L_1\to L_2$ is a linear map such that 
$H([x,y]_1)=[H(x),H(y)]_2$, for $x$ and $y\in L_1$. In particular, when $L_2=L(X)$, $X$ a 
Banach space, we say that $H\colon L_1\to L(X)$ is
a representation of $L_1$. \par   

\indent We say that a subspace $I$ of $L$ is a subalgebra when
$[I,I]\subseteq I$, and an ideal when $[I,L]\subseteq I$, where
if $M$ and $N$ are two subsets of $L$, then
$[M,N]$ denotes the set $\{[m,n]: m\in M,\hbox{ } n\in N\}$.
In  particular, $L^2=[L,L]=\{ [x,y]: x,\hbox{ }y\in L\}$
is an ideal of $L$. In addition, we say that a linear map $f\colon L\to \Bbb C$ is a character
when $f(L^{2})=0$, i.e., when $f\colon L\to \Bbb C$ is a Lie morphism.\par

\indent For any Lie algebra $L$ we can consider the following two
series of ideals. The derived series, i.e.
$$
L=L^{(1)}\supseteq L^{(2)}=[L,L]\supseteq L^{(3)}=
[L^{(2)},L^{(2)}]\supseteq \ldots \supseteq L^{(k)}=[L^{(k-1)},
L^{(k-1)}],
$$
and the descending central series, i.e.
$$
L=L^{1}\supseteq L^{2}=L^{(2)}=[L,L]\supseteq L^{3}=
[L,L^{2}]\supseteq \ldots \supseteq L^{k}=[L,
L^{k-1}]\supseteq\ldots.
$$

\indent A Lie algebra $L$ is considered solvable or nilpotent if there is some positive
integer $k$ such that $L^{(k)}=0$ or $L^k=0$ respectively. 
Obviously all nilpotent Lie algebras 
are solvable.\par

\indent One of the most useful properties of a complex solvable finite
dimensional Lie algebra $L$ is the existence of Jordan-H\" older
sequences, i.e., a sequence of 
ideals $(L_k)_{0\le k\le n}$ such that\par
\noindent (i) $L_0=0$, $L_n=L$,\par
\noindent  (ii) $L_i\subseteq L_{i+1}$, for $i=0,\ldots , n-1$,\par
\noindent  (iii) $\dim L_i=i$, where $n=\dim L$; see [8, Chapter 5, Section 3, Corollaire 3]. \par

\indent Another important property of these
algebras is the existence of polarizations. A polarization of a 
character $f$ of $L$ is a subalgebra $P(f)$ of $L$ maximal 
with respect to the property $f([I,I])=0$, where $I$ is a subalgebra 
of $L$. In fact, if $(L_k)_{0\le k\le n}$ is a Jordan-H\"older
sequence of ideals of $L$, then $P(f; (L_k)_{0\le k\le n})=
\sum_{i=0}^n N_i(f_i)$ is a polarization of $f$, where $N_i(f_i)=
\{x\in L_i: f([x,L_i])=0\}$; see [4, Chapter IV, Section 4, Proposition 4.1.1]. \par  

\indent Next we review the definitions of the Taylor, the S\l odkowski 
and the split joint spectra. From now on $L$ denotes a complex solvable finite dimensional 
Lie algebra, $X$ a complex Banach space and $\rho\colon L\to {\rm L}
(X)$ a representation of $L$ in $X$.
We consider the Koszul complex of the representation $\rho$, i.e., 
$(X\otimes\wedge L, d(\rho))$, where $\wedge L$ denotes the 
exterior algebra of $L$, and  
$d_p (\rho)\colon X\otimes\wedge^p L\rightarrow X\otimes
\wedge^{p-1} L$ is the map defined by

\begin{align*} d_p (\rho)&( x\otimes\langle l_1\wedge\dots\wedge l_p\rangle) 
=  \sum_{k=1}^p (-1)^{k+1}\rho (l_k)(x)\otimes\langle l_1
\wedge\ldots\wedge\hat{l_k}\wedge\ldots\wedge l_p\rangle\\
                                                         + & \sum_{1\le i< j\le p} (-1)^{i+j}
x\otimes\langle [l_i, l_j]\wedge l_1\wedge\ldots\wedge\hat{l_i}\wedge
\ldots\wedge\hat{l_j}\wedge
\ldots\wedge l_p\rangle ,\\ \end{align*} 

where $\hat{ }$ means deletion. 
For $p$ such that $p\le 0$ or $p\ge {n+1}$, $n=\dim L$, we 
define $d_p (\rho) =0$.\par

\indent In addition, if $f$ is a character of $L$, then we 
consider the representarion of $L$ in $X$ $\rho-f\equiv \rho-f\cdot I$, 
where $I$ denotes the identity map of $X$. Now, if $H_*(X\otimes
\wedge L , d(\rho-f))$ denotes the homology of the Koszul complex of 
the representation $\rho-f$, then we consider the set
$$
\sigma_p(\rho)= \{ f\in L^*: f(L^2)=0, H_p(X\otimes\wedge L,
d(\rho-f))\ne 0\}.
$$
Now we state the definition of the Taylor and the S\l odkowski joint 
spectra; see [5], [7], [16], [20] and [21]. We follow the notation of 
[21, Definition 2.11.1].\par

\begin{df} Let $X$ be a complex Banach
space, $L$ a complex solvable finite dimensional Lie algebra, 
and $\rho\colon L\to {\rm L}(X)$ a representation of $L$
in $X$. Then, 
the Taylor joint spectrum of $\rho$
is the set
$$\sigma(\rho)= \bigcup_{p=0}^n\sigma_p(\rho)=\{f\in L^*: f(L^2) =0, 
H_*(X\otimes\wedge L,d(\rho-f))\neq 0\}.
$$
\indent In addition, the $k$-th $\delta$-S\l odkowski joint spectrum of $\rho$ 
is the set
$$
\sigma_{\delta ,k}(\rho)= \bigcup_{p=0}^k\sigma_p(\rho) ,
$$
and the $k$-th $\pi$-S\l odkowski joint spectrum of $\rho$ is the set 
$$
\sigma_{\pi ,k}(\rho)= \bigcup_{p=n-k}^n\sigma_p(\rho)\cup \{f
\in L^*: f(L^2)=0, R(d_{n-k}(\rho-f)) \hbox{ is not closed}\},
$$
for $0\le k\le n=\dim L$. \par
\indent We observe that   $\sigma_{\delta ,n}(\rho)= \sigma_{\pi ,n}
(\rho)= \sigma(\rho)$.

\end{df}

\indent The Taylor and the S\l odkowki joint spectra are compact non-void
subsets of $L^*$.
When $L\subseteq L(X)$ is a commutative subalgebra of operators
and the representation is the inclusion $\iota
\colon L\to L(X)$, $\iota (T)=T$, $T\in L$, $\sigma (\iota )$, $\sigma_{\delta ,k}(\iota)$ and
$\sigma_{\pi ,k}(\iota)$ are reduced to the usual Taylor and the usual S\l odkowski joint spectra respectively
in the following sense. If $l=(l_1,\ldots , l_n)$ is a basis of $L$ and 
$\sigma$ denotes either the Taylor joint spectrum or 
one of the S\l odkowski joint spectra of $\iota$, 
then $\{(f(l_1),\ldots , f(l_n)): f\in\sigma \}=
\sigma(l_1,\ldots ,l_n)$, i.e., the joint 
spectrum $\sigma$ in terms of the
basis $l=(l_1,\ldots ,l_n)$ coincides with
the spectrum of the $n$-tuple $l$. 
In addition, these joint spectra have the so-called projection property. 
However, since this property is one of the most important
results that all the joint spectra considered in this work
have in common,
we give the explicit definition. \par

\begin{df} Let $X$ be a complex Banach space and $\sigma$
a function which assigns a compact non-void subset
of the characters of $L$ to each representation 
$\rho\colon L\to {\rm L}(X)$ of a complex
solvable finite dimensional Lie algebra $L$ in
$X$ . In addition, let $I$ be an ideal or a subalgebra
of $L$, in the solvable or nilpotent case respectively, and 
consider the representation
$\rho\mid I\colon I\to {\rm L}(X)$, i.e., the restriction of $\rho$ 
to $ I$. 
Then, we say that $\sigma$
has the projection property when for each ideal 
or subalgebra, in the solvable or nilpotent case respectively,  
we have 
$$
\pi (\sigma(\rho))=\sigma (\rho\mid I),
$$
where $\pi\colon L^*\to I^*$ is the restriction map.\par
\end{df} 

\indent Next we review the definition of the split joint spectra, and we
prove their most important properties, the projection
property among them.\par

\indent A finite complex of Banach space and bounded linear operators 
$(X,d)$ is a sequence 
$$
0\to X_n\xrightarrow{d_n}X_{n-1}\to\ldots\to X_1\xrightarrow{d_1}X_0\to 0,
$$
where $n\in \Bbb N$, $X_p$ is a Banach space, and the maps 
$d_p \in{\rm L}(X_p ,X_{p-1})$ are such that $d_p
\circ d_{p-1}=0$,  for $p=0,\ldots , n$. \par
\indent For a fixed integer $p$, $0\le p\le n$, 
we say that $(X,d)$ is split in degree $p$ if there are continous 
linear operators
$$
X_{p+1}\xleftarrow{h_p}X_p\xleftarrow{h_{p-1}}X_{p-1},
$$
such that $d_{p+1}h_p+h_{p-1}d_p=I_p$, where 
$I_p$ denotes the identity operator of $X_p$.\par

\indent In addition, if $L$, $X$ and $\rho$ are as above, then for each 
$p$ we consider the set

$$
sp_p (\rho )= \{ f\in L^*: f(L^2)=0, (X\otimes\wedge L,d(\rho-f))
\hbox{ is not split in degree p} \}.
$$
Next follows the definition of the split joint spectra; see 
[13] and [21].\par

\begin{df}Let $X$ be a complex Banach
space, $L$ a complex solvable finite dimensional Lie algebra, 
and $\rho\colon L\to {\rm L}(X)$ a representation of $L$
in $X$. Then, 
the split joint spectrum of $\rho$ is the set
$$
sp(\rho)=\bigcup_{p=0}^n sp_p(\rho).
$$
In addition, the $k$-th $\delta$-split joint spectrum of $\rho$ is the set    
$$
sp_{\delta ,k}(\rho)= \bigcup_{p=0}^k sp_p(\rho) ,
$$
and the $k$-th $\pi$-split joint spectrum of $\rho$ is the set
$$
sp_{\pi ,k}(\rho)= \bigcup_{p=n-k}^n sp_p(\rho),
$$
for $0\le k\le n=\dim L$.\par
\indent  We observe that   $sp_{\delta ,n}(\rho)= sp_{\pi ,n}(\rho)= 
sp(\rho)$.
\end{df}

\indent It is clear that $\sigma_{\delta ,k}(\rho)\subseteq 
sp_{\delta ,k}(\rho)$, $\sigma_{\pi ,k}(\rho)\subseteq 
sp_{\pi ,k}(\rho)$, and that $\sigma(\rho)\subseteq sp(\rho)$. 
Moreover,
if $X$ is a Hilbert space, the above inclusions are equalities. In addition, 
when $L\subseteq L(X)$ is a commutative subalgebra
of operators and the representation is the inclusion $\iota
\colon L\to L(X)$, these joint spectra
coincide with the ones introduced by J. Eschmeier in [13] for 
commuting 
tuples of operators in the same sense explained for the 
Taylor and the S\l odkowski joint spectra. \par

\indent In the following theorem we consider
the main properties of the split joint spectra; for
a complete exposition see [21, Chapter 3].\par

\begin{thm}  Let $X$ be a complex Banach space, 
$L$ a complex 
solvable  finite dimensional Lie algebra, and $\rho\colon L\to 
{\rm L}(X)$ 
a representation of $L$ in $X$. Then
the sets $sp(\rho)$,
$sp_{\delta ,k }(\rho)$, and $sp_{\pi ,k }(\rho)$
are compact non-void subsets of $L^*$ that have the 
projection property.
\end{thm}
\begin{proof}

\indent First of all, in [21, Korollar 3.1.9] it was proved that 
$sp(\rho)$ is a compact non-void subset of $L^*$ that 
has the projection property.\par

\indent On the other hand, by [21, Satz 3.1.5], [21, Satz 3.1.7] 
and an argument similar to the one in [21, Korollar 3.1.9], 
it is easy to prove that
$sp_{\delta ,k}(\rho)$ and $sp_{\pi ,k}(\rho)$ are compact 
non-void subsets of $L^*$ that have the projection property.\par       

\end{proof}
\section{The Fredholm joint spectra}

\indent In order to prove the main results in section 6 and 7
we need to study several essential joint spectra.
We first consider the essential joint spectra 
introduced by A. S. Fainshtein, [15], and by V. M\"uller, [19],
for commuting tuples of operators 
and we extend them to representations of solvable Lie algebras
in Banach spaces.
In addition, we extend the essential 
split joint spectra introduced by J. Eschmeier in [13] to such representations .
We begin with the essential Taylor and the S\l odkowski joint 
spectra.\par

\indent Let $X$, $L$, and $\rho\colon L\to {\rm L}(X)$ be as in 
section 2, and for each $p$ consider the set
$$
\sigma_{p ,e}(\rho)=\{f\in L^*: f(L^2)=0, \dim\hbox{ }H_p(X\otimes
\wedge L,d(\rho-f))=\infty\}.
$$

\begin{df}Let $X$ be a complex Banach
space, $L$ a complex solvable finite dimensional Lie algebra,
and $\rho\colon L\to {\rm L}(X)$ a representation of $L$
in $X$. Then, the Fredholm, or essential Taylor, joint spectrum of $\rho$ is the set    
$$
\sigma_e(\rho)=\bigcup_{p=0}^n \sigma_{p ,e}(\rho).
$$
In addition the $k$-th Fredholm or essential $\delta$-S\l odkowski joint spectrum of $\rho$ 
is the set
$$
\sigma_{\delta , k ,e}(\rho)=\bigcup_{p=0}^k \sigma_{p ,e}(\rho),   
$$
and the $k$-th Fredholm or essential $\pi$-S\l odkowski joint spectrum of $\rho$ is the set
$$
\sigma_{\pi ,k ,e}(\rho)=\bigcup_{p=n-k}^n \sigma_{p ,e}(\rho)
\cup \{f\in L^*: f(L^2)=0, 
R(d_{n-k}(\rho))\hbox{ is not closed}\},
$$
for $0\le k\le n=\dim L$.\par
We observe that $\sigma_e(\rho)=\sigma_{\delta ,n ,e}(\rho)=
\sigma_{\pi ,n ,e}(\rho)$.
 
\end{df}

\indent Now we prove that these sets are really joint spectra. 
In fact, we first show 
the properties of the sets $\sigma_{\delta ,k ,e}(\rho)$ 
and then by a duality argument we obtain the 
properties of the sets $\sigma_{\pi ,k ,e}(\rho)$. Moreover, our proof
of the properties of the sets $\sigma_{\delta ,k ,e}(\rho)$ is a direct generalization of the one developed 
in [15].\par

\begin{thm}  Let $X$ be a complex Banach space, 
$L$ a complex 
solvable  finite dimensional Lie algebra, and $\rho\colon L\to 
{\rm L}(X)$ 
a representation of $L$ in $X$. Then 
the sets
$\sigma_{\delta ,k ,e}(\rho)$
are compact non-void subsets of $L^*$ that have the 
projection property.\par
\indent In particular, $\sigma_e(\rho)$ is a compact non-void subset 
of $L^*$ that has the projection property.\par 
\end{thm}
\begin{proof}

\indent It is clear that $\sigma_{\delta ,k ,e}(\rho)\subseteq \sigma_{
\delta  ,k}(\rho)$.
Moreover, by [25, Theorem  2.11] $\sigma_{\delta ,k ,e}(\rho)$ is a closed set. 
Thus,
$\sigma_{\delta ,k ,e}(\rho)$ is a compact subset of $L^*$.\par 

\indent In order to prove the projection property for ideals of a solvable
Lie algebra, by 
[8, Chapter 5, Section 3, Corollaire 3] it is enough to consider an ideal $I$ of $L$ 
of codimension 1. Then, 
if we consider the usual decomposition of the chain complex 
$(X\otimes\wedge L, d(\rho))$ associated to the ideal $I$ and the 
short exact sequence defined by this decomposition (see [7], [5] 
and [20]), an easy calculation shows that 
$$
\pi (\sigma_{\delta ,k ,e}(\rho))\subseteq\sigma_{\delta ,k ,e}(\rho
\mid I).
$$
\indent On the other hand, to prove the reverse inclusion
we may apply A. S. Fainshtein's argument in [15, Lemma  1], i.e., 
the essential version of [23, Lemma 1.6]; see
also [19]. However, we have to verify the following fact: 
if $\tilde {f}\in \sigma_{\delta , k, e}(\rho\mid I)$, then for each 
$f\in L^*$
such that $f\mid I=\tilde {f}$, $f$ is a character of $L$, i.e., 
$f(L^2)=0$.\par

\indent Indeed, since $\tilde {f}\in\sigma_{\delta , k ,e}(\rho\mid I)
\subseteq \sigma_{\delta , k}(\rho\mid I)$,
$\tilde{f}$ is a character of $I$, i.e., $\tilde{f}(I^2)=0$. However, 
since $I$ is an ideal of $L$, by the projection property of the
joint spectrum $\sigma_{\delta ,k}(\rho)$ (see 
[5, Theorem 4.5], 
[20, Theorem  3.4] and [21, Satz 2.11.5]), there is $f\in \sigma_{\delta ,k ,e}(\rho)$ 
such that $f\mid I=\tilde{f}$.\par

\indent Now, since $f$ is a character of $L$, $L$ is a polarization for 
$f$ (see [4, Chapter IV, Section 4, Proposition 4.1.1] or section 2). Moreover, as $I$ is an ideal of codimension 1 in $L$
and $\tilde{f}$ is a character of $I$, if there was $f{'}
\in L^*$ 
such that $f{'}\mid I=\tilde{f}$ and such that $f{'}$ was not a 
character of $L$, then
$I$ would be a polarization of $f{'}$ (see [4, Chapter IV, Section 4, Proposition 4.1.1]). However, if we 
considered
$f-f{'}$, by  [8, Chapter 5, Section 3, Corollaire 3] and [4, Chapter IV, Section 4, Proposition 4.1.1] we would have $I=L$, 
which is impossible according to our assumption. Thus, every extension of 
$\tilde{f}$ to $L^*$ is a 
character of $L$. So, we showed the projection property for ideals
of a solvable Lie algebra.\par

\indent We suppose that $L$ is a nilpotent Lie algebra and 
that $I$ is a subalgebra of $L$. As in [21, Satz 0.3.7] we consider
a sequence of subalgebras of $L$, $(L^k+I)_{k\in {\Bbb N}}$,
where $(L^k)_{k\in {\Bbb N} }$ is the descending central series
of $L$. In particular, we have $L^1+I=
L+I=L$. Moreover, since $L$ is a nilpotent Lie algebra, there is $k_0\in\Bbb
N$ such that $L^k=0$ for all $k\ge k_0$, which implies that for all
$k\ge k_0$, $L^k+I=I$. In addition, since for all $k\in {\Bbb N}$
$[L, L^k]=L^{k+1}$, we have $[L^k+I, 
L^{k+1}+I]\subseteq [L,L^{k+1}]+ [L^{k}, L]+[I,I]
\subseteq L^{k+1}+I$, i.e., for each $k\in {\Bbb N}$, 
$L^{k+1}+I$ is an ideal of $L^{k}+I$. Thus, considering the projection
property for ideals, we get the projection property for subalgebras
of nilpotent Lie algebras.\par  
\end{proof}

\indent We proved the main properties of the joint 
spectra $\sigma_{\delta ,k ,e}(\rho)$. For $\sigma_{\pi ,k ,e}(\rho)$ 
we proceed by a duality argument. We begin with the following proposition,
which extends a result of 
Z. S\l odkowski (see [23, Lemma 2.1]).\par   

\begin{pro}Let $X\xrightarrow{A} Y\xrightarrow{B}Z$ be a chain complex 
of Banach spaces and
bounded linear operators. Then the following conditions are equivalent:\par

\noindent {\rm (i)} $\dim Ker (B)/R(A)<\infty$ and  $R(B)$ is closed,

\noindent  {\rm (ii)} $\dim Ker(A^*)/R(B^*)<\infty$ and $R(A^*)$ is closed.
\end{pro}  

\begin{proof}

\indent First of all, if $\dim Ker(B)/R(A)<\infty$, then $R(A)$ is closed, 
and then $R(A^*)$ is 
closed. \par

\indent Now, if $N$ is a finite dimensional subspace of $Y$ such that 
$R(A)\oplus N =Ker(B)$ and if $i\colon N\to Y$ is the inclusion map, 
then we consider the chain complex
$$
X\oplus N\xrightarrow{A{'}}Y\xrightarrow{B}Z,
$$
where $A{'}=A\oplus i$, i.e., for $x\in X$ and $n\in N$,
$$
A{'}(x,n)=A(x)+i(n).
$$  
\indent Since $R(A{'})=Ker (B)$ and $R(B)$ is closed, by [23, Lemma 2.1] 
we have 
$$
R(B^*)=Ker(A{'}^*)=Ker(A^{*})\cap Ker (i^*)=Ker (A^*)\cap 
N^{\perp}\subseteq Ker(A^*).
$$
\indent Now we consider the inclusion map
$$
\iota_1\colon Ker (A^*)\to Y{'},
$$

where $Y{'}$ denotes the dual Banach space of  $Y$.\par
\indent Since $\iota_1(R(B^*))\subseteq N^{\perp}$, we may  
consider the quotient map
$$
\tilde{\iota}_1\colon Ker(A^*)/R(B^*)\to Y{'}/N^{\perp}.
$$
However, if $M$ is a closed subspace of $Y$ such that $N\oplus M=Y$, 
then
$$
Y{'}/N^{\perp}\cong M^{\perp}\cong N{'}.
$$ 
In particular, $\dim Y{'}/N^{\perp}<\infty$, and since $\iota_1^{-1}
(N^{\perp})=R(B^*)$,
we have $\tilde{\iota}_1$ is an injection, which implies that 
$\dim Ker(A^*)/R(B^*)<\infty$.\par

\indent On the other hand, if $\dim Ker(A^*)/R(B^*)<\infty$, then 
$R(B^*)$ is 
closed and then $R(B)$ is closed.\par

\indent Now, if in the canonical way we identify $Y$ and 
$Z$ with a 
subspace of $Y{''}$ and $Z{''}$ respectively, then 
$$
R(A^{**})\cap Y=R(A),\hskip2cm Ker(B^{**})\cap Y=N(B).
$$ 
Thus,
$$
\dim Ker(B)/R(A)=\dim Ker(B^{**})\cap Y/R(A^{**})\cap Y.
$$

In addition, if we consider the inclusion map
$$
\iota_2\colon Ker(B^{**})\cap Y\to Ker(B^{**}),
$$
since $\iota_2(R(A^{**})\cap Y)\subseteq R(A^{**})$ and 
$\iota_2^{-1}(R(A^{**}))=R(A^{**})\cap Y$,
the quotient map
$$
\tilde{\iota}_2\colon Ker(B^{**})\cap Y/R(A^{**})\cap Y \to 
Ker(B^{**})/R(A^{**})
$$
is an injection. In particular, $\dim Ker(B)/R(A)\le \dim N(B^{**})/
R(A^{**})$.\par

\indent However, from the first part of the proposition, that
has just been proved,
we know that $\dim N(B^{**})/R(A^{**})<\infty$.\par
\end{proof}

\indent When $\rho\colon L\to L(X)$ is a representation of the
Lie algebra $L$ in the Banach space $X$, we may consider the adjoint
representation of $\rho$, i.e., $\rho^{*}\colon L^{op}\to
L(X{'})$, $\rho^{*}(l)=(\rho (l))^{*}$, where $X{'}$ denotes
the dual space of $X$. Now we  
relate the joint spectra $\sigma_{\delta ,k,e}(\rho)$ and 
$\sigma_{\pi ,k,e}(\rho)$.\par

\begin{thm} Let $X$ be a complex Banach space, $L$ a 
complex solvable
finite dimensional Lie algebra, and $\rho\colon L \to {\rm L}(X)$ a 
representation of $L$ in $X$. If  
$\rho^{*}\colon L^{op}\to {\rm L}(X{'})$ is the adjoint 
representation of $\rho$, then there is a character of $L$, 
$h$, depending only on the Lie structure of $L$, such that  \par

\noindent {\rm (i)} $\sigma_{\delta ,k ,e}(\rho)+h=\sigma_{\pi ,k ,e}(\rho^*),$\par

\noindent  {\rm (ii)} $\sigma_{\pi ,k ,e}(\rho)+h=\sigma_{\delta ,k ,e}(\rho^*)$,\par
    
for $0\le k\le n$.\par

\end{thm}
\begin{proof}
\indent It is a consequence of Proposition 3.3, [5, Theorem 1 ] and [21, Korollar 2.4.5].\par 

\end{proof}

\indent Now we state the main properties of the joint spectra 
$\sigma_{\pi ,k,e}(\rho)$.\par

\begin{thm}  Let $X$ be a complex Banach space, $L$ a 
complex solvable
finite dimensional Lie algebra, and $\rho\colon L\to {\rm L}(X)$ a 
representation of $L$ in $X$. Then, the sets
$\sigma_{\pi ,k ,e}(\rho)$  
are compact non-void subsets of $L^*$ that have the projection 
property.
\end{thm}
\begin{proof}

\indent According to Theorems 3.2 and 3.4, $\sigma_{\pi ,k ,e}(\rho)$ are compact 
non-void subsets of $L^*$, $0\le k\le n=\dim L$. \par

\indent On the other hand, if $L$ is a solvable Lie algebra and $I$
an ideal of $L$, then by  [8, Chapter 5, Section 3, Corollaire 3] there is a Jordan-H\" older 
sequence of ideals of $L$ such that $I$ is one of its terms. Thus 
we may suppose that $\dim I=n-1$. In addition, if 
$h$ and $h_{ I}$ are the characters of $L$ and $I$ involved in formulas 
(i) and (ii) of Theorem 3.4
that correspond to the Lie algebra $L$ and the ideal $I$ respectively, 
then by [5, Theorem 1] and 
[8, Chapter 5, Section 3, Corollaire 3], or by [21, Korollar  2.4.5], $h\mid I=h_I$. In particular,
$$
\sigma_{\pi ,k ,e}(\rho\mid I)+h\mid I=\sigma_{\delta ,k ,e}(\rho^*
\mid I).
$$
Then, according to Theorem 3.4 we have 
$$
\pi (\sigma_{\pi ,k,e}(\rho))=\pi (\sigma_{\delta ,k,e}(\rho^{*})-h)=
\sigma_{\delta ,k,e}(\rho^{*})-h\mid I= \sigma_{\pi ,k ,e}(\rho\mid I).
$$

\indent So, we proved the projection property for ideals of a solvable
Lie algebra. On the other hand, to prove the projection property
for subalgebras of a nilpotent Lie algebra, it is enough to apply the 
corresponding proof of Theorem 3.2.\par 
\end{proof}

\indent Now we study the essential split joint 
spectra. These joint spectra are the 
extension to representations of solvable Lie algebras in Banach
space of the corresponding joint spectra introduced by J. 
Eschmeier in [13] for finite tuples of commuting Banach space operators.
Moreover, in order to show their main
properties, we use a characterization
proved in [13]. \par

\indent As in section 2, we now consider a finite complex of Banach 
spaces and bounded linear operators $(X,d)$, 
$$
0\to X_n\xrightarrow{d_n}X_{n-1}\to\ldots\to X_1\xrightarrow{d_1}X_0\to 0.
$$
\indent Given a fixed integer $p$, $0\le p\le n$, 
we say that $(X,d)$ is Fredholm split in degree $p$ if there 
are continous linear operators
$$
X_{p+1}\xleftarrow{h_p}X_p\xleftarrow{h_{p-1}}X_{p-1},
$$
and $k_p$ a compact operator defined in $X_p$ such that 
$d_{p+1}h_p+h_{p-1}d_p=I_p-k_p$, where $I_p$ denotes the
identity operator of $X_p$.\par

\indent Let $X$, $L$, and $\rho\colon L\to {\rm L}(X)$
be as in section 2, and for each $p$ consider the set

\begin{align*}
sp_{p,e} (\rho )= &\{ f\in L^*: f(L^2)=0, (X\otimes\wedge L,d(\rho-f))
\hbox{ is not Fredholm split}\\
                           &\hbox{in degree p} \}.\\\end{align*}

Now we state the definition of the essential split joint spectra; see [13].\par

\begin{df}Let $X$ be a complex Banach
space, $L$ a complex solvable finite dimensional Lie algebra,
and $\rho\colon L\to {\rm L}(X)$ a representation of $L$
in $X$. Then, 
the Fredholm or essential split joint spectrum of $\rho$ is the set
$$
sp_e(\rho)=\bigcup_{p=0}^n sp_{p,e}(\rho).
$$
In addition, the $k$-th Fredholm or essential $\delta$-split joint spectrum of $\rho$ 
is the set    
$$
sp_{\delta ,k,e}(\rho)= \bigcup_{p=0}^k sp_{p,e}(\rho) ,
$$
and the $k$-th Fredholm or essential $\pi$-split joint spectrum of $\rho$ is the set
$$
sp_{\pi ,k,e}(\rho)= \bigcup_{p=n-k}^n sp_{p,e}(\rho),
$$
for $0\le k\le n$.\par
\indent  We observe that $sp_{\delta ,n,e}(\rho)= 
sp_{\pi ,n ,e}(\rho)= sp_e(\rho)$.
\end{df}

\indent In order to show the main properties of these
joint spectra, we need to prove some technical results. We first
review several facts related to complexes of Banach space operators.\par

\indent Given a finite complex of Banach spaces and bounded 
linear operators $(X,d)$ and a Banach space $Y$, we denote the complex 
$$
0\to{\rm L}(Y, X_n)\xrightarrow{L_{d_n}}{\rm L}(Y,X_{n-1})\to\ldots
\to{\rm L}(Y, X_1)\xrightarrow{L_{d_1}}{\rm L}(Y,X_0)\to 0
$$
by ${\rm L}(Y,X.)$, where $L_{d_p}$ denotes the induced operator of left multiplication 
with $d_p$, i.e., for $T\in{\rm L}(Y,X_p)$, $L_{d_p}(T)= 
d_p\circ T\in{\rm L}(Y,X_{p-1})$, $0\le p\le n$; see [13].\par 

\indent In addition, if $X_1$ and $X_2$ are two complex Banach 
spaces,
and if ${\rm K}(X_1,X_2)$ denotes the ideal of
all compact operators in ${\rm L}(X_1,X_2)$, then it is clear that 
$L_{d_p}({\rm K}(Y,X_p))\subseteq {\rm K}(Y,X_{p-1})$. Thus,
we may consider the complex ${\rm C}(Y,X.)
=({\rm C}(Y,X_p),
\tilde{L}_{d_p})$, where ${\rm C}(Y,X_p)
={\rm L}(Y,X_p)/
{\rm K}(Y,X_p)$ and $\tilde{L}_{d_p}$ is the quotient operator 
associated to $L_{d_p}$; see [13].\par

\indent On the other hand, if $L$, $X$, and $\rho\colon L
\to{\rm L}(X)$ are as in section 2, then we consider the 
representation 
$$
L_{\rho}\colon L\to {\rm L}({\rm L}(X)),\hbox{}l\mapsto L_{\rho(l)},
$$
where $L_{\rho(l)}$ denotes the left multiplication
operator associated to $\rho(l)$, $l\in L$; 
see Chapter 3, Section 3.1 of [21]. \par

\indent In addition, since $L_{\rho(l)}({\rm K}(X))\subseteq 
{\rm K}(X)$, it is possible to consider the representation
$$
\tilde{L}_{\rho}\colon L\to {\rm L}({\rm C}(X)),
$$   
where ${\rm C}(X)={\rm L}(X)/{\rm K(X)}$
and $\tilde{L}_{\rho}(l)$ is the quotient operator defined in 
${\rm C}(X)$ associated to $L_{\rho(l)}$.\par 

\indent In the following proposition we relate 
the complexes $({\rm C}(X)\otimes \wedge L, 
d(\tilde{L}_{\rho}))$ and ${\rm C}
(X,(X\otimes\wedge L ,d(\rho)).)$.\par

\begin{pro}Let $X$ be a complex Banach space, 
$L$ a complex 
solvable finite dimensional Lie algebra, and $\rho\colon L\to{\rm L} 
(X)$ a representation of $L$ in $X$. Then, the complexes 
$({\rm C}(X)\otimes \wedge L, d(\tilde{L}_{\rho}))$ and 
${\rm C}(X,(X\otimes\wedge L ,
d(\rho)).)$ are naturally isomorphic.\end{pro}
\begin{proof}

\indent First of all, we consider the complexes $({\rm L}(X)\otimes
\wedge L, d(L_{\rho}))$ and ${\rm L}(X,(X\otimes\wedge L, 
d(\rho)).)$. 
In [21, Satz 3.1.4] it was proved that these
two complexes are naturally isomorphic. Indeed, if $\Phi_p\colon {\rm L}(X)\otimes
\wedge^p L\to{\rm L}(X,X\otimes\wedge^p L)$
is the map
$$
\Phi_p(T\otimes \xi)(x)=T(x)\otimes \xi,
$$
$T\in {\rm L}(X)$, $\xi\in\wedge^p L$ and $x\in X$,
then $\Phi:({\rm L}(X)\otimes\wedge L, d(L_{\rho}))\to
{\rm L}(X,(X\otimes\wedge L, d(\rho)).)$
is an isomorphism of chain complexes. In particular, the following 
diagram is commutative
$$
\CD
{\rm L}(X)\otimes\wedge^p L @>d_p(L_{\rho})>>
{\rm L}(X)\otimes\wedge^{p-1}L\\
@VV\Phi_p  V@VV\Phi_{p-1}V\\
{\rm L}(X,X\otimes\wedge^p L)@>L_{d_p}>>
{\rm L}(X,X\otimes\wedge^{p-1} L).
\endCD
$$

\indent However, since $\Phi_p$ is an isomorphism, 
an easy calculation shows that $\Phi_p({\rm K}(X)\otimes\wedge^p L)
={\rm K}(X,X\otimes\wedge^p L)$. Thus, 
we may consider the quotient map associated to $\Phi_p$,
$\tilde{\Phi}_p\colon{\rm C}(X)\otimes
\wedge^p L\to{\rm C}(X,X\otimes\wedge^p L)$, which is
an isomorphism.\par
 
\indent In addition, it is clear that $d_p(L_{\rho})({\rm K}(X)\otimes
\wedge^p L)\subseteq{\rm K}(X)\otimes\wedge^{p-1} L$. 
Furthermore, if $\pi_p\colon {\rm L}(X)\otimes\wedge^p L\to
{\rm C}(X)\otimes\wedge^p L$ denotes the projection map, it is 
easy to prove that the quotient map associated to $d_p(L_{\rho})$ 
coincides with
$d_p(\tilde{L}_{\rho})$, i.e., we have the commutative 
diagram
$$
\CD
{\rm L}(X)\otimes\wedge^p L @>d_p(L_{\rho})>>{\rm L}(X)
\otimes\wedge^{p-1}L\\
@VV\pi_p  V@VV\pi_{p-1}V\\
{\rm C}(X)\otimes\wedge^p L@>d_p(\tilde{L}_{\rho})>>
{\rm C}(X)\otimes\wedge^{p-1} L.
\endCD
$$
In particular, the family $(\pi_p)_{0\le p\le n}\colon ({\rm L}(X)\otimes\wedge L, 
d(L_{\rho}))\to({\rm C}(X)\otimes\wedge L, 
d(\tilde{L}_{\rho}))$ is a morphism of chain complexes. 

\indent Thus, we obtain the commutative diagram
$$
\CD
{\rm C}(X)\otimes\wedge^p L @>d_p(\tilde{L}_{\rho})>>
{\rm C}(X)\otimes\wedge^{p-1}L\\
@VV\tilde{\Phi}_p  V@VV\tilde{\Phi}_{p-1}V\\
{\rm C}(X,X\otimes\wedge^p L)@>\tilde{L}_{d_p}>>
{\rm C}(X,X\otimes\wedge^{p-1} L).
\endCD
$$
Finally, since for each $p$, $0\le p\le n$, the map $\tilde{\Phi}_p$ 
is an isomorphism, the family $\tilde{\Phi}=(\tilde{\Phi}_p)_{0\le p
\le n}\colon ({\rm C}(X)\otimes \wedge L, d(\tilde{L}_{\rho}))
\to {\rm C}(X,(X\otimes
\wedge L ,d(\rho)).)$ is an isomorphism of chain complexes.

\end{proof}

\indent In order to show that the sets  
introduced in Definition 3.6 are really joint spectra, 
we need to prove a similar isomorphism to the one in 
Proposition 3.7, but which is related to right multiplication instead of 
being related to left 
multiplication. We first review some results necessary 
for our objective.\par 

\indent Let $(X,d)$ be a finite complex of Banach spaces and bounded 
linear operators and $Y$ a complex Banach space. We denote the complex 
$$
0\to{\rm L}(X_0,Y)\xrightarrow{R_{d_1}}{\rm L}(X_1,Y)\to\ldots\to{
\rm L}(Y, X_{n-1})\xrightarrow{R_{d_n}}{\rm L}(X_n,Y)\to 0
$$
by ${\rm L}(X.,Y)$, where $R_{d_p}$ denotes the induced operator of right multiplication 
with $d_p$, i.e., for $T\in{\rm L}(X_{p-1}, Y)$, $R_{d_p}(T)= T\circ d_p
\in{\rm L}(X_p,Y)$, $0\le p\le n$; see [13].\par 

\indent Moreover, it is clear that $R_{d_p}({\rm K}(X_{p-1},Y))
\subseteq {\rm K}(X_p,Y)$. Thus,
we may consider the complex ${\rm C}(X.,Y)=({\rm C}(X_p,Y),
\tilde{R}_{d_p})$, where ${\rm C}(X_p,Y)={\rm L}(X_p,Y)/
{\rm K}(X_p,Y)$ and $\tilde{R}_{d_p}$ is the quotient operator 
associated to $R_{d_p}$; see [13].\par

\indent On the other hand, if $L$, $X$, and $\rho\colon L\to{\rm L}(X)$ 
are as in section 2, then we consider the 
representation 
$$
R_{\rho}\colon L^{op}\to {\rm L}({\rm L}(X)),\hbox{}l\mapsto 
R_{\rho(l)},
$$
where $R_{\rho(l)}$ denotes the right multiplication
operator associated to $\rho(l)$, $l\in L^{op}$; see Chapter 3, Section 3 of  [21]. \par

\indent Furthermore, since $R_{\rho(l)}({
\rm K}(X))\subseteq {\rm K}(X)$, it is possible to consider the 
representation
$$
\tilde{R}_{\rho}\colon L^{op}\to {\rm L}({\rm C}(X)),
$$   
where $\tilde{R}_{\rho}(l)$ is the quotient 
operator associated to $R_{\rho(l)}$.\par 

\indent Now we consider the Chevalley-Eilenberg 
cochain complex associated to the representation 
$\tilde{R}_{\rho}\colon L^{op}\to {\rm L}({\rm C}(X))$,
i.e., $ChE(\tilde{R}_{\rho})=({\rm Hom}(\wedge L, {\rm C}(X)),$
$\delta(\tilde{R}_{\rho}))$, where $\delta_p(\tilde{R}_{\rho})
\colon {\rm Hom}(\wedge^p L,{\rm C}(X))
\to {\rm Hom}(\wedge^{p+1} L,$ ${\rm C}(X))$ is the map defined by

\begin{align*}
(\delta_p(\tilde{R}_{\rho})f)&(x_1\ldots x_{p+1})=\sum_{i=1}^{p+1}
(-1)^{i-1}\tilde{R}_{\rho}(x_i)f(x_1\ldots\hat{x_i}\ldots x_{p+1})\\
                 &+\sum_{1\le i<k\le p+1}(-1)^{i+k}f([x_i,x_k].x_1\ldots
\hat{x_i}\ldots\hat{x_k}\ldots x_{p+1}),\\
\end{align*}

for $f\in {\rm Hom}(\wedge^p L, {\rm C}(X))$ and $x_i\in L^{op}$, $1\le i
\le p+1$; see [21, Satz und Definition 2.1.9].\par 

\indent In the following proposition we relate 
the complexes $ChE(\tilde{R}_{\rho})$ and ${\rm C}((X\otimes
\wedge L ,d(\rho)).,X)$.\par

\begin{pro} The complexes $ ChE(
\tilde{R}_{\rho})$ and ${\rm C}((X\otimes\wedge L ,d(\rho)).,X)$ 
are naturally isomorphic.\end{pro}
\begin{proof}

\indent First of all, we consider the representation $R_{\rho}\colon L^{op}
\to {\rm L}({\rm L}(X))$ and the Chevalley-Eilenberg cochain complex
associated to it, i.e., $ChE(R_{\rho})=({\rm Hom}(\wedge L,$
${\rm L}(X)),
\delta(R_{\rho}))$, where $\delta_p(R_{\rho})
\colon {\rm Hom}(\wedge^p L,{\rm L}(X))
\to {\rm Hom}(\wedge^{p+1} L,$ ${\rm L}(X))$ is the map defined by

\begin{align*}
(\delta_p(R_{\rho})f)&(x_1\ldots x_{p+1})=\sum_{i=1}^{p+1}
(-1)^{i-1}R_{\rho}(x_i)f(x_1\ldots\hat{x_i}\ldots x_{p+1})\\
                 &+\sum_{1\le i<k\le p+1}(-1)^{i+k}f([x_i,x_k].x_1\ldots
\hat{x_i}\ldots\hat{x_k}\ldots x_{p+1}),\\
\end{align*}

for $f\in {\rm Hom}(\wedge^p L, {\rm L}(X))$ and $x_i\in L^{op}$, $1\le i
\le p+1$; see [21, Satz und Definition  2.1.9].\par 

\indent Now, in [21, Satz 3.1.6] it was proved that 
the complexes $ChE(R_{\rho})$ and ${\rm L}((X\otimes\wedge L, 
d(\rho))., X)$ are naturally isomorphic. Indeed, if 
$\Psi_p\colon {\rm Hom}(\wedge^p L,{\rm L}$
$(X))\to{\rm L}(X\otimes\wedge^p L,X)$
is the map
$$
(\Psi_p(f))(x\otimes \xi)=f(\xi )(x),
$$
$f\in {\rm Hom}(\wedge^p L,{\rm L}(X))$, $\xi\in\wedge^p L$ and $x\in X$,
then $\Psi:ChE(R_{\rho})\to
{\rm L}((X\otimes\wedge L, d(\rho)).,X)$
is an isomorphism of chain complex. In particular, the following 
diagram is commutative:
$$
\CD
{\rm Hom}(\wedge^p L,{\rm L}(X)) @>\delta_p(R_{\rho})>>
{\rm Hom}(\wedge^{p+1} L ,{\rm L}(X))\\
@VV\Psi_p  V@VV\Psi_{p+1}V\\
{\rm L}(X\otimes\wedge^p L, X)@>R_{d_{p+1}}>>
{\rm L}(X\otimes\wedge^{p+1} L,X).
\endCD
$$

\indent Since $\Psi_p$ is an isomorphism, 
an easy calculation shows that $\Psi_p({\rm Hom}(\wedge^p L, {\rm K}(X))$
$={\rm K}(X\otimes\wedge^p L, X)$. Thus, we may 
consider the quotient map associated to $\Psi_p$, $\tilde{\Psi}_p\colon {\rm Hom}(\wedge^p L, {\rm C}(X))
\to{\rm C}(X\otimes\wedge^p L, X)$, which is an isomorphism.\par
 
\indent In addition, it is clear that $\delta_p(R_{\rho})( {\rm Hom}( \wedge^p L,{\rm K} (X))
\subseteq {\rm Hom} (\wedge^{p+1}L, {\rm K}$  $(X))$. 
Furthermore, if $\pi_p\colon {\rm Hom} (\wedge^p L,{\rm L}(X))\to
{\rm Hom} (\wedge^p L, {\rm C}(X))$ denotes the projection map, it is 
easy to prove that the quotient map associated to $\delta_p(R_{\rho})$ 
coincides with
$\delta_p(\tilde{R}_{\rho})$, i.e., we have the commutative 
diagram
$$
\CD
{\rm Hom} (\wedge^pL, {\rm L}(X)) @>\delta_p(R_{\rho})>>{\rm Hom}(\wedge^{p+1} L,{\rm L}(X))
\\
@VV\pi_p  V@VV\pi_{p+1}V\\
{\rm Hom} (\wedge^p L,{\rm C}(X))@>\delta_p(\tilde{R}_{\rho})>>
{\rm Hom} (\wedge^{p+1} L,{\rm C}(X)).
\endCD
$$
In particular, the family $(\pi_p)_{0\le p\le n} \colon ChE(R_{\rho})
\to ChE(\tilde{R}_{\rho})$ is a morphism of chain complexes. 

\indent Thus, we obtain the commutative diagram
$$
\CD
{\rm Hom}( \wedge^p L, {\rm C}(X)) @>\delta_p(\tilde{R}_{\rho})>>
{\rm Hom}( \wedge^{p+1} L, {\rm C}(X))\\
@VV\tilde{\Psi}_p  V@VV\tilde{\Psi}_{p+1}V\\
{\rm C}(X\otimes\wedge^p L,X)@>\tilde{R}_{d_p}>>
{\rm C}(X\otimes\wedge^{p+1} L,X).
\endCD
$$
Finally, since for each $p$, $0\le p\le n$, the map $\tilde{\Psi}_p$ 
is an isomorphism, the family $\tilde{\Psi}=(\tilde{\Psi}_p)_{0\le p
\le n}\colon ChE(\tilde{R}_{\rho})\to {\rm C}((X\otimes
\wedge L ,d(\rho)).,X)$ is an isomorphism of chain complexes.

\end{proof}
 
\indent Now we state the main spectral properties of the essential
split joint spectra.\par

\begin{thm}Let $X$ be a complex Banach space, $L$ a 
complex 
solvable finite dimensional Lie algebra, and $\rho\colon L\to
{\rm L} (X)$ a 
representation of $L$ in $X$. Then \par

\noindent {\rm (i)} $sp_{\delta ,k,e}(\rho)=\sigma_{\delta ,k}(\tilde{L}_{\rho})$,\par

\noindent  {\rm (ii)} $sp_{\pi ,k,e}(\rho)=\sigma_{\delta ,k}(\tilde{R}_{\rho})+h$,\par

\noindent  {\rm (iii)} $sp_e(\rho)=\sigma_e(\tilde{L}_{\rho})=\sigma_e(
\tilde{R}_{\rho})+h$,\par

where $h$ is the character of $L$ considered in Theorem 3.4 
and $0\le k\le n$.\par
\end{thm}
\begin{proof}
 
\indent Since the argument in [13, Proposition 2.4(a-iii)] 
applies in the non-commutative case, according to Proposition 3.7 we have 
$$
sp_{\delta ,k,e}(\rho)=\sigma_{\delta ,k}(\tilde{L}_{\rho}).
$$
\indent In addition, since the argument in [13, Proposition 2.4(b-iii)] applies
in the non-commutative case, if $h$ is the character of $L$ 
considered in Theorem 3.4 (see [5, Theorem 1] and [21, Korollar 2.4.5]), 
then according to Proposition 3.8 and [21, Satz 2.4.4] we have 
$$
sp_{\pi ,k,e}(\rho)=\sigma_{\delta ,k}(\tilde{R}_{\rho})+h.
$$
\indent The third statement is clear.\par

\end{proof}

\begin{thm}  Let $X$ be a complex Banach space, 
$L$ a complex 
solvable  finite dimensional Lie algebra, and $\rho\colon L\to 
{\rm L}(X)$ 
a representation of $L$ in $X$. Then, 
the sets $sp_e(\rho)$,
$sp_{\delta ,k, e}(\rho)$, and $sp_{\pi , k ,e }(\rho)$
are compact non-void subsets of $L^*$ that have the 
projection property.
\end{thm}
\begin{proof}

\indent The main properties of the essential split joint spectra
 may be deduced from the corresponding
ones of the S\l odkowski and the Taylor joint spectra, and from the
particular behavior of the character $h$ with respect to Lie ideals of $L$; see
the proof of Theorem 3.5.

\end{proof}
\indent Finally, in the following proposition we consider
two nilpotent Lie algebras and two representations of the algebras
in a complex Banach space 
related by an epimorphism, and we describe the connection 
between the joint spectra of the representations. We 
need this result for nilpotent and commutative systems of
operators. In addition, these results provide an
extension of [21, Satz 2.7.4]
and [21, Korollar 3.1.10] for representations of nilpotent Lie algebras, 
from the Taylor 
to the S\l odkowski joint spectra and from the usual split spectrum to
the split joint spectra. Moreover,
we consider the corresponding essential joint spectra
and prove similar characterizations. \par

\begin{pro} Let $X$ be a complex Banach space,
$L_1$ and $L_2$ two complex nilpotent finite dimensional Lie algebras, $\rho_1
\colon L_1\to{\rm L}(X)$ a representation of $L_1$, and $f\colon
L_2\to L_1$ a Lie algebra epimorphism. Then, if we consider
the representation $\rho_2=\rho_1\circ f\colon L_2
\to{\rm L}(X)$, we have \par

\noindent {\rm (i)}$\sigma_{\delta ,k}(\rho_2)=\sigma_{\delta ,k}(\rho_1)\circ f$,
$\sigma_{\pi ,k}(\rho_2)=\sigma_{\pi ,k}(\rho_1)\circ f$,par
 
\noindent  {\rm (ii)}$\sigma_{\delta ,k, e}(\rho_2)=\sigma_{\delta ,k ,e}(\rho_1)\circ f$,
$\sigma_{\pi ,k ,e}(\rho_2)=\sigma_{\pi ,k ,e}(\rho_1)\circ f$,\par

\noindent  {\rm (iii)}$sp_{\delta ,k}(\rho_2)=sp_{\delta ,k}(\rho_1)\circ f$,
$sp_{\pi ,k}(\rho_2)=sp_{\pi , k}(\rho_1)\circ f$,\par

\noindent  {\rm (iv)}$sp_{\delta ,k ,e}(\rho_2)=sp_{\delta ,k ,e}(\rho_1)\circ f$,
$sp_{\pi ,k ,e}(\rho_2)=\sigma_{\pi ,k ,e}(\rho_1)\circ f$,\par

where, $\sigma_*(\rho_1)\circ f=\{\alpha\circ f: \alpha\in
\sigma_*(\rho_1)\}$ and $sp_*(\rho_1)\circ f=\{\alpha\circ f: \alpha\in
sp_*(\rho_1)\}$.\par
\end{pro}
\begin{proof}

\indent A careful inspection of [16, Proposition 2.5] and [16, Proposition 2.6] shows that
it is possible to refine the arguments of these results in order to prove
that the Koszul complex of $\rho_1$ is exact for $p=0,\ldots ,k$
if and only if the Koszul complex of $\rho_2$ is exact for $p=0,
\ldots ,k$. In particular, if $\alpha \in \sigma_{\delta ,k}(\rho_1)$,
then $\rho_2-\alpha\circ f=(\rho_1-\alpha)\circ f$, which implies 
that $\sigma_{\delta ,k}(\rho_1)\circ f\subseteq
\sigma_{\delta ,k}(\rho_2)$. On the other hand, since
$\sigma_{\delta , k}(\rho_2)\subseteq\sigma(\rho_2)$,
by [21, Satz 2.7.4], if $\beta\in\sigma_{\delta ,k}(\rho_2)$,
then there is $\alpha\in \sigma(\rho_1)$ such that $\beta=\alpha
\circ f$. However, by the above observation, since $\rho_2-\beta=
(\rho_1-\alpha)\circ f$, $\alpha\in\sigma_{\delta ,k}(\rho_1)$.
Thus, $\sigma_{\delta ,k}(\rho_2)=\sigma_{\delta ,k}(\rho_1)\circ f$.\par

\indent In addition, a careful inspection of [16, Proposition 2.5] and [16, Proposition 2.6]
shows that
it is possible to extend the arguments developed in these results
for the essential $\delta$-S\l odkowski
joint spectra, i.e., it is possible to prove that the Koszul complex of 
$\rho_1$ is Fredholm for $p=0,\ldots ,k$ if and only if the 
Koszul complex of $\rho_2$ is Fredholm for $p=0,\ldots ,k$.
In particular, we may apply the same argument that we developed for the joint 
spectra $\sigma_{\delta ,k}$ to the 
joint spectra
$\sigma_{\delta ,k ,e}$. Thus, $\sigma_{\delta ,k ,e}(\rho_2)=
\sigma_{\delta ,k ,e}(\rho_1)\circ f$.\par

Now if we consider the 
representations defined in Theorem 3.4 $\rho_1^*\colon L_1^{op}\to{\rm L}(X{'})$
and $\rho_2^*\colon L_1^{op}\to{\rm L}(X{'})$,
then $\rho_2^*=\rho_1^*\circ f$. However,
by [5, Theorem 7], [21, Lemma  2.11.4] and Theorem 3.4 we have 
$\sigma_{\pi ,k}(\rho_1)\circ f=\sigma_{\pi ,k}(\rho_2)$
and $\sigma_{\pi ,k ,e}(\rho_2)=\sigma_{\pi ,k ,e}(\rho_1)\circ f$.
\par

\indent Furthermore, if we consider the representations
$L_{\rho_i}\colon L_i\to{\rm L}({\rm L}(X))$ and
$R_{\rho_i}\colon L_i^{op}$
$\to{\rm L}({\rm L}(X))$,
for $i=1$, $2$, then 
$L_{\rho_2}=L_{\rho_1}\circ f$ and 
$R_{\rho_2}=R_{\rho_1}\circ f$. Then,
by [21, Satz 3.1.5] and [21, Satz 3.1.7] we have 
$sp_{\delta ,k}(\rho_2)=sp_{\delta ,k}(\rho_1)\circ f$ and
$sp_{\pi ,k}(\rho_2)=sp_{\pi ,k}(\rho_1)\circ f$.\par

Finally, if we consider the representations
$\tilde{L}_{\rho_i}\colon L_i\to{\rm L}({\rm C}(X))$ and
$\tilde{R}_{\rho_i}\colon L_i^{op}\to{\rm L}({\rm C}(X))$,
for $i=1$, $2$, then  
$\tilde{L}_{\rho_2}=\tilde{L}_{\rho_1}\circ f$ and 
$\tilde{R}_{\rho_2}=\tilde{R}_{\rho_1}\circ f$. Then,
according to Theorem 3.9 we have 
$sp_{\delta ,k,e}(\rho_2)=sp_{\delta ,k,e}(\rho_1)\circ f$ and
$sp_{\pi ,k,e}(\rho_2)=sp_{\pi ,k,e}(\rho_1)\circ f$.\par

\end{proof}

\section{Tensor products of Banach spaces}

\indent In this section we review the definition and the main properties
of the 
tensor product of complex Banach spaces 
introduced by J. Eschmeier in [14]. In addition, we prove some
propositions necessary for our main results. 
\par

\indent A pair $\langle X, \tilde{X}\rangle$ of Banach spaces will be called a dual 
pairing, if
$$
(A)\hbox{  }\tilde{X}=X{'} \hbox{   or   }(B)\hbox{  }X=
\tilde{X}{'}.
$$

\indent In both cases, the canonical bilinear mapping is denoted by
$$
X\times\tilde{X}\to\Bbb C ,\hbox{   }(x,u)\mapsto\langle x,u\rangle.
$$

\indent If $\langle X,\tilde{X}\rangle$ is a dual pairing, we consider the 
subalgebra ${\mathcal L}(X)$ of ${\rm L}(X)$ consisting of all 
operators $T\in{\rm L}(X)$ for which there is an operator $T{'}
\in {\rm L}(\tilde{X})$ with
$$
\langle Tx,u\rangle=\langle x,T{'}u\rangle,
$$
for all $x\in X$ and $u\in\tilde{X}$. It is clear that if the dual pairing 
is $\langle X,X{'}\rangle$, then ${\mathcal L}(X)={\rm L}(X)$, and that if the 
dual pairing is $\langle X{'},X\rangle$, then ${\mathcal L}(X)=\{T^*,\hbox{ }T
\in {\rm L}(\tilde{X})\}$. In particular, each operator of the form
$$
f_{y,v}\colon X\to X,\hbox{}x\mapsto\langle x,v\rangle y,
$$
is contained in ${\mathcal L}(X)$, for $y\in X$ and $v\in\tilde{X}$.
\par 

\indent Now we recall the definition of the tensor product introduced by J. 
Eschmeier in [14].\par

\begin{df} Given two dual pairings $\langle X,\tilde{X}\rangle$ and 
$\langle Y,\tilde{Y}\rangle$, a tensor product of the Banach spaces $X$ and $Y$ 
relative to the dual pairings $\langle X,\tilde{X}\rangle$ and 
$\langle Y,\tilde{Y}\rangle$ is a Banach space $Z$ together with 
continuous bilinear mappings
$$
X\times Y\to Z,\hbox{}(x,y)\mapsto x\otimes y;\hbox{  }{\mathcal L}(X)
\times{\mathcal L}(Y)\mapsto{\rm L}(Z),\hbox{}(T,S)\mapsto T\otimes S,
$$  
which satisfy the following conditions,

(T1)$\hbox{  }\parallel x\otimes y\parallel=\parallel x\parallel
\parallel y\parallel$,\par
(T2)$\hbox{  }T\otimes S(x\otimes y)=(Tx)\otimes(Sy)$,\par
(T3)$\hbox{  }(T_1\otimes S_1)\circ (T_2\otimes S_2)=(T_1T_2)
\otimes(S_1S_2),\hbox{}I\otimes I=I$,\par
(T4)$\hbox{  }Im(f_{x,u}\otimes I)\subseteq\{x\otimes y\colon\hbox{}y
\in Y\},\hbox{}Im(f_{y,v}\otimes I)\subseteq\{x\otimes y\colon\hbox{}x
\in X\}$.\par
\end{df}
\indent As in [14], we write $X\tilde{\otimes} Y$ instead of $Z$. In addition, as in [14] we have two 
applications of Definition 4.1, namely, 
the completion $X\tilde{\otimes_{\alpha}}Y$ of the 
algebraic tensor product of the
Banach spaces $X$ and $Y$ with respect to a quasi-uniform crossnorm 
$\alpha$, see [18], and an operator ideal between Banach spaces; see 
[14] and section 7.\par

\indent In order to prove our main
results we need to study the behavior of a split and Fredholm 
split complex of Banach spaces with respect to the procedure of
tensoring it with a fixed Banach space. We begin with some preparation
and then we prove our characterization.\par

\indent Let $(X,d)$ be, as in section 2, a complex of Banach spaces and bounded linear 
operators, and let us suppose that $(X,d)$ is Fredholm 
split for $p=0,\ldots ,k$. 
Then, by [13, Theorem 2.7] and its proof, the complex $(X,d)$ is Fredholm 
for $p=0,\ldots , k$ and $Ker (d_p)$ is a complemented 
subspace of $X_p$ for $p=1,\ldots , k+1$. In addition, if for $p=1,
\ldots ,k+1$ we decompose 
$X_p=Ker( d_p)\oplus L_p$, then for $p=1,\ldots , k$ we have
$X_p=R(d_{p+1})\oplus N_p\oplus L_p$, 
where $N_p$ is a finite dimensional subspace of $X_p$ such 
that $R(d_{p+1})\oplus N_p=Ker(d_p)$. Moreover,
for $p=0$ we know that $X_0=R(d_1)\oplus N_0$, where $N_0$
is a finite dimensional subspace of $X_0$; in particular,
we may define $L_0=0$. However, thanks to
these decompositions, for $p=0,\ldots , k$ 
there are well-defined operators $h_p\colon X_p  \to X_{p+1}$, 
such that
\par (i) $h_p\mid L_p=0$, $h_p\mid N_p=0$, $h_p\circ d_{p+1}=
I_p\mid L_{p+1}$, where $I_p$ denotes the identity operator of $X_p$,\par
(ii) $d_{p+1}h_p+h_{p-1}d_p=I_p-k_p$, where $k_p$ is  
the projector of $X_p$ with range $N_p$ and null space 
$R(d_{p+1})\oplus L_p$,\par
(iii) $h_ph_{p-1}=0$ for $p=1,\ldots ,k$.\par

\indent In addition, if the complex $(X,d)$ is split for $p=0,\ldots , k$, 
then it is exact for $p=0,\ldots ,k$, and in the above decompositions
$N_p=0$ for $p=0,\ldots ,k$. In particular, $k_p=0$
for $p=0,\ldots ,k$.
\par

\indent If there is a Banach space $Z$
such that for each $p\in \Bbb Z$ there
is an $n_p\in\Bbb N_0$ with $X_p=Z^{n_p}$,
and a Banach space Y such that there is 
a tensor products $Y\tilde{\otimes} Z$ relative to 
$\langle Y, Y{'}\rangle$ and $\langle Z,Z{'}\rangle$, then 
we may consider the chain complex
$$
Y\tilde{\otimes}X_{k+1}\xrightarrow{I\otimes d_{k+1}}
Y\tilde{\otimes}X_k \xrightarrow{I\otimes d_k} 
Y\tilde{\otimes} X_{k-1}\to\ldots\to 
Y\tilde{\otimes} X_1\xrightarrow{I\otimes d_1}
Y\tilde{\otimes}X_0\to 0,
$$
where $I$ denotes the identity of $Y$. Moreover, if for $p=0,\ldots ,k$
we consider the maps 
$I\otimes h_p\colon Y\tilde{\otimes} X_p\to Y\tilde{\otimes} 
X_{p+1}$,  then \par
(i) $I\otimes d_{p+1}\circ I\otimes h_p+I\otimes h_{p-1}\circ I
\otimes d_p=I-I\otimes k_p$, \par
(ii) $I\otimes h_p\circ I\otimes h_{p-1}=0$.\par
 
\indent It is worth noticing that the properties of the tensor product 
and the fact $X_p=Z^{n_p}$ imply that the
maps 
$I\otimes d_p$, $p=0,\ldots ,k+1$, and $I\otimes h_p$,
$p=0,\ldots ,k$, are well defined
and the compositions behave as usual. \par

\indent Similarly, we consider a chain complex that is split or
Fredholm split 
for $p= k,\ldots ,n$.\par 

\indent Let $(X,d)$ be, as in section 2, a complex of Banach spaces 
and bounded linear opertors, 
and let us suppose that $(X,d)$ is Fredholm split for $p=k,\ldots ,n$. 
Then, by [13, Theorem 2.7] and its proof, the complex $(X,d)$ is Fredholm 
for $p=k,\ldots , n$ and  
$R(d_{p+1})$ is a closed 
complemented subspace of $X_p$ for $p=k-1,\ldots , n-1$. In 
addition, for $p=k,\ldots ,n-1$ we may decompose $X_p=R( d_{p+1})\oplus N_p
\oplus L_p$, 
where $N_p$ is a finite dimensional subspace of $Ker(d_p)$
such that $Ker (d_p)=R(d_{p+1})\oplus N_p$. Moreover, for
$p=n$ we know that $X_n=N_n\oplus L_n$, where $N_n=
Ker (d_n)$, and for $p=k-1$ we define $N_{k-1}=0$
and $L_{k-1}$ such that $X_{k-1}=R(d_k)\oplus L_{k-1}$. 
However, thanks to these decompositions, for
$p=k-1,\ldots ,n$ there are well-defined operators 
$h_p\colon X_p  \to X_{p+1}$ such that
\par
(i) $h_p\mid L_p=0$, $h_p\circ d_{p+1}=I_p\mid L_{p+1}$, 
$h_p\mid N_p=0$, where $I_p$ denotes the identity operator of $X_p$,\par
(ii) $d_{p+1}h_p+h_{p-1}d_p=I_p-k_p$, for $p=k,\ldots ,n$, 
where $k_p$ is the projector of $X_p$
with range $N_p$ and null space $L_p
\oplus R(d_{p+1})$,\par
(iii) $h_ph_{p-1}=0$ for $p=k,\ldots ,n$.\par

\indent In addition, if the complex $(X,d)$ is split for $p=k,\ldots , n$,
it is exact for $p=k,\ldots ,n$, and in the above decompositions
$N_p=0$ for $p=k,\ldots ,n$. In particular, $k_p=0$
for $p=k,\ldots ,n$.\par

\indent If there is a Banach space $Z$
such that for each $p\in\Bbb Z$
there is an $n_p\in\Bbb N_0$ with $X_p=Z^{n_p}$, and 
a Banach space Y such that there is 
a tensor product $Y\tilde{\otimes} Z$ relative to 
$\langle Y, Y{'}\rangle$ and $\langle Z,Z{'}\rangle$, then
we may consider the chain complex
$$
0\to Y\tilde{\otimes}X_n\xrightarrow{I\otimes d_n}
Y\tilde{\otimes}X_{n-1}\to\ldots\to 
Y\tilde{\otimes} X_k\xrightarrow{I\otimes d_k}
Y\tilde{\otimes}X_{k-1}\to ,
$$
where $I$ denotes the identity of $Y$. Then, if for 
$p=k-1,\ldots ,n-1$ we consider the maps 
$I\otimes h_p\colon Y\tilde{\otimes} X_p\to Y\tilde{\otimes} 
X_{p+1}$, for $p=k,\ldots ,n$, we have\par
(i) $I\otimes d_{p+1}\circ I\otimes h_p+I\otimes h_{p-1}\circ I
\otimes d_p=I- I\otimes k_p$,\par
(ii) $I\otimes h_p\circ I\otimes h_{p-1}=0$.\par

\indent As before, the maps $I\otimes d_p$, $p=n,\ldots k$, and
$I\otimes h_p$, $p=k-1,\ldots n-1$, are well defined and the
compositions behave as usual.\par

\begin{pro} In the above conditions, for $p=0,\ldots ,k$ we have\par
{\rm (i)} $I\otimes h_p\circ I\otimes d_{p+1}=I\otimes h_pd_{p+1}$ 
is a projector defined in $Y\tilde{\otimes }X_{p+1}$. In particular, $Y\tilde
{\otimes} X_{p+1}=Ker(I\otimes h_pd_{p+1})\oplus R(I\otimes h_p
d_{p+1})$.\par
 {\rm (ii)} $Ker(I\otimes h_pd_{p+1})=Ker (I\otimes d_{p+1})$,
$R ( I\otimes h_pd_{p+1})= R (I\otimes h_p)$,
and $Ker(I\otimes h_p)=R(I\otimes h_{p-1})\oplus R(I\otimes k_p)$.
\par

\indent Similarly, for $p=k,\ldots ,n$ we have\par
{\rm (i)} $I\otimes d_p\circ I\otimes h_{p-1}=I\otimes d_ph_{p-1}$ 
is a projector defined in $Y\tilde{\otimes }X_{p-1}$. In particular, $Y\tilde
{\otimes} X_{p-1}=Ker(I\otimes d_ph_{p-1})\oplus R(I\otimes d_p
h_{p-1})$.\par
 {\rm (ii)} $Ker(I\otimes d_ph_{p-1})=Ker (I\otimes h_{p-1})$,
$R ( I\otimes d_ph_{p-1})= R (I\otimes d_p)$, and
$Ker(I\otimes h_p)=R(I\otimes h_{p-1})\oplus R(I\otimes k_p)$.\par
\end{pro}

\begin{proof}

\indent We only prove the first part of the proposition; the proof
of the second one is similar.\par
 
\indent It is easy to prove that $h_pd_{p+1}\colon X_{p+1}\to X_{p+1} $ 
is a projector. Thus, according to the properties of the tensor product we obtain 
the first assertion.\par

\indent With regard to $R(I\otimes h_pd_{p+1})$, since 
$I\otimes h_pd_{p+1}=I\otimes h_p\circ I\otimes d_{p+1}$,
it is clear that $R(I\otimes h_pd_{p+1})\subseteq R(I\otimes h_p)$.
\par

\indent On the other hand, since
$$
I\otimes h_pd_{p+1}\circ I\otimes h_p=I\otimes h_pd_{p+1}
h_p= I\otimes h_p(I_p-k_p-h_{p-1}d_p)=I\otimes h_p,
$$
we have $R(I\otimes h_p)\subseteq R(I\otimes h_pd_{p+1})$. 
Thus, the equality is proved.\par

\indent With respect to $Ker(I\otimes h_pd_{p+1})$, since 
$I\otimes h_pd_{p+1}=I\otimes h_p\circ I\otimes d_{p+1}$,
it is clear that $Ker(I\otimes d_{p+1})\subseteq Ker(I\otimes h_p
d_{p+1})$. However,
$$
I\otimes d_{p+1}\circ I\otimes h_pd_{p+1}= I\otimes d_{p+1}h_p
d_{p+1}= I\otimes (I_p-k_p-h_{p-1}d_p)d_{p+1}=
I\otimes d_{p+1}.
$$
Thus $Ker(I\otimes h_pd_{p+1})\subseteq Ker(I\otimes d_{p+1})$, 
and we have the equality.\par

\indent In order to prove the characterization of $Ker(I\otimes h_p)$, we
first suppose that $p=1,\ldots ,k$. We observe that 
$I\otimes h_p\circ I\otimes h_{p-1}=I\otimes h_p
h_{p-1}=0$, and that $I\otimes h_p\circ I\otimes k_p=I\otimes 
h_pk_p=0$. Thus, $R(I\otimes k_p)+R(I\otimes h_{p-1})\subseteq
Ker(I\otimes h_p)$. Moreover, $R(I\otimes k_p)\cap 
R(I\otimes h_{p-1})=0$. \par

\indent In fact, since $k_p$ is a projector, $I\otimes k_p$
is a projector. In particular, we may suppose that if $z\in
R(I\otimes k_p)$, then $z=I\otimes k_p(z)$.
In addition, if $z=I\otimes h_{p-1}(w)$, then we have 
$$
z=I\otimes k_p(z)=I\otimes k_p(I\otimes h_{p-1}(w))=I\otimes 
k_ph_{p-1}(w)=0.
$$
Then, $R(I\otimes k_p)\oplus R(I\otimes h_{p-1})\subseteq 
Ker (I\otimes h_p)$. \par

\indent On the other hand, if $z\in Ker (I\otimes h_p)$, then 
we have $z=I\otimes k_p (z)+I\otimes h_{p-1}d_p(z)$.
Thus, $z\in R(I\otimes h_{p-1})\oplus R(I\otimes k_p)$,
and we have the equality.\par

\indent Now, if $p=0$, it is clear that $R(I\otimes k_0)\subseteq 
Ker(I\otimes h_0)$. On the other hand, $I-I\otimes k_0=
I\otimes d_1\circ I\otimes h_0$. 
In particular, if $z\in Ker(I\otimes h_0)$, then $z\in R(I\otimes k_0)$.
Thus, $Ker(I\otimes h_0)=R(I\otimes k_0)$.\par
  
\end{proof}

\begin{rema}\rm  In the above conditions, if there is a Banach space $Y$ and  
a tensor product $Z\tilde{\otimes} Y$ relative to 
$\langle Z,Z{'}\rangle$ and $\langle Y, Y{'}\rangle$, then 
by similar arguments it is possible to obtain similar results  
to the ones of  Proposition 4.2, but in which  
the order of the spaces and maps in the tensor products are
interchanged.\end{rema}

\indent Now we review the relation between the tensor product of 
J. Eschmeier and complexes of Banach spaces; see 
[14, Section 3]. \par

\indent Let $(\langle X_i,\tilde{X}_i\rangle )_{0\le i\le n}$ be a system of
dual pairings of Banach spaces such that $\tilde{X}_i=X_i{'}$
for all $i=0,\ldots ,n$, or $X_i=\tilde{X}_i{'}$ for all $i=0,\ldots ,n$. 
Then, if ${\mathcal X}=\bigoplus_{p=0}^nX_p$ and if 
$\tilde{\mathcal X}=\bigoplus_{p=0}^n \tilde{X}_p$, according to the observations
in [14, Section 3], $\langle {\mathcal X },\tilde{\mathcal X}\rangle$ is a dual 
pairing. Moreover, if for all $i=1,\ldots ,n$ there is an operator 
$d_i^{'}\in{\mathcal L}(X_i,X_{i-1})$ such that $d^{'}_{i-1}
\circ d^{'}_i=0$, then 
$$
0\to X_n\xrightarrow{d^{'}_n}X_{n-1}\to\ldots\to X_1\xrightarrow{d^{'}_1}X_0
\to 0
$$
is a complex of Banach spaces and bounded linear operators; 
we denote it by $(X,d^{'})$. In
addition, if $\partial^{'}=\oplus_{p=1}^n d_p^{'}$, then
$({\mathcal X},\partial^{'}) $ is the differential space associated to the
complex $(X,d^{'})$ and $\partial^{'}\in {\mathcal L}({\mathcal X})$.\par  

\indent Now we consider another system of dual pairings $(\langle Y_j,
\tilde{Y}_j\rangle)_{0\le j\le m}$, with the property stated above, i.e., 
$\tilde{Y}_j=Y_j{'}$ for all $j=0,\ldots ,m$,
or $Y_j=\tilde{Y}_j{'}$ for 
all $j=0,\ldots ,m$. As above, we suppose that for all 
$j=1,\ldots ,m$ there is an operator 
$d_j^{''}\in{\mathcal L}(Y_j,Y_{j-1})$ such that $d^{''}_{j-1}
\circ d^{''}_j=0$. Thus, we have a
differential complex
$$
0\to Y_m\xrightarrow{d^{''}_m}Y_{m-1}\to\ldots\to Y_1\xrightarrow{d^{''}_1}
Y_0\to 0;
$$
we denote it by $(Y,d^{''})$. In
addition, if $\partial^{''}=\oplus_{q=1}^m d_q^{''}$, then
$({\mathcal Y},\partial^{''}) $ is the differential space associated to the
complex $(Y,d^{''})$ and $\partial^{''}\in {\mathcal L}({\mathcal Y})$.\par  

 \indent We suppose that for each $i=0,\ldots ,n$ and for each $j=0,\ldots , m$
there is a tensor product $X_i\tilde{\otimes} Y_j$ relative to
$\langle X_i,\tilde{X}_i\rangle$ and $\langle Y_j,\tilde{Y}_j\rangle$, in such a way that
all these tensor products are compatible in the sense described at the 
end of section 1 in [14]. In particular, it is possible to consider 
the tensor
product ${\mathcal X}\tilde{\otimes}{\mathcal Y}$ relative to 
$\langle\mathcal X,\tilde{\mathcal X}\rangle$ and
$\langle\mathcal Y,\tilde{\mathcal Y}\rangle$; see [14, Section 1]. 
Moreover, if $\eta\in{\mathcal L}({\mathcal X})$ is the map defined by
$\eta\mid X_p=(-1)^pI_p$, where $I_p$ denotes the identity
of $X_p$, then the map
$\partial \colon{\mathcal X}\tilde{\otimes}{\mathcal Y}\to 
{\mathcal X}\tilde{\otimes}{\mathcal Y}$ defined by 
$$
\partial=\partial^{'}\otimes I_q+\eta \otimes 
\partial^{''},
$$
is such that $\partial\circ\partial=0$ and that $\partial\in
{\mathcal L}({\mathcal X}\tilde{\otimes}{\mathcal Y})$, where $I_q$ denotes
the identity of $Y_q$.\par

However, if we consider the double complex
$$
\CD
X_{p-1}\tilde{\otimes}Y_q @<d^{'}_p\otimes I_q <<X_p\tilde{\otimes}
Y_q\\
@V(-1)^{p-1}I_{p-1}\otimes d^{"}_q VV          @VV(-1)^p I_p\otimes d^{"}_q V\\
X_{p-1}\tilde{\otimes}Y_{q-1}@<<^d{'}_p\otimes I_{q-1}<X_p\tilde{\otimes} 
Y_{q-1},\\
\endCD
$$ 
then the differential space associated to the total complex of this 
double complex is $(\mathcal X\tilde{\otimes}\mathcal Y, \partial)$.\par

\indent Now, if $L_1$ and $L_2$ are 
two complex solvable
finite dimensional Lie algebras of dimensions $n$ and $m$
respectively, 
$X_1$ and $X_2$ two complex Banach 
spaces, and $\rho_i \colon L_i\to {\rm L}(X_i)$, $i=1$, $2$,
two representations of the Lie 
algebras, then we may consider the Koszul complexes
associated to the representations $\rho_1$ 
and $\rho_2$, i.e., $(X_1\otimes\wedge 
L_1, d(\rho_1))$
and $(X_2\otimes\wedge L_2, d(\rho_2))$ respectively.
\par

\indent It is clear that for $p=0,\ldots , n$ and for $q=0,\ldots , m $
$\langle X_1\otimes\wedge^p L_1,X_1\otimes \wedge^p L_1{'}\rangle$ 
and $\langle X_2\otimes\wedge^q L_2,X_2\otimes \wedge^q L_2{'} \rangle$
are dual pairings. Moreover, $d_p(\rho_1)\in{\mathcal L }(X_1\otimes\wedge^p L_1,
X_1\otimes \wedge^{p-1} L_1)$ and $d_q(\rho_2)\in{\mathcal L }(X_2
\otimes\wedge^q L_2,X_2\otimes \wedge^{q-1} L_2)$, for $p=0,\ldots ,n$ and
$q=0,\ldots ,m$. Thus, we may consider the differential spaces
$({\mathcal X}_1,\partial_1)$ and $({\mathcal X}_2,\partial_2)$,
where ${\mathcal X}_1=X_1\otimes\wedge L_1$,
${\mathcal X}_2=X_2\otimes\wedge L_2$, $\partial_1=\oplus^n_{p=1}
d_p(\rho_1)$ and $\partial_2=\oplus^m_{q=1}
d_q(\rho_2)$.\par

\indent We suppose that there is a tensor product of 
$X_1$ and $X_2$ with respect to $\langle X_1,X_1{'}\rangle$ and 
$\langle X_2,X_2{'}\rangle$, $X_1\tilde{\otimes} X_2$. Then, 
according to the
considerations at the end of section 1 in [14], for all $p=0,\ldots ,n$
and $q=0,\ldots , m$ there is a well-defined tensor product of
$X_1\otimes \wedge^p L_1$ and $X_2\otimes \wedge^q L_2$,   
$X_1\otimes\wedge^p L_1\tilde{\otimes}
X_2\otimes\wedge^q L_2$, relative to  
$\langle X_1\otimes\wedge^p L_1,X_1\otimes\wedge^p L_1{'}\rangle$  
and $\langle X_2\otimes\wedge^q L_2,X_2\otimes\wedge^q L_2{'}\rangle$. 
Furthermore, since
for all $p$ and $q$ such that $p=0,
\ldots, n$ and $q=0,\ldots, m$, these tensor products
are compatible in the sense described at the end of section 1 in [14];
as above, we may consider the tensor product of ${\mathcal X}_1$
and ${\mathcal X}_2$, ${\mathcal X}_1\tilde{\otimes}{\mathcal X}_2$,
which is a differential space with differential $\partial\in{\mathcal L}
({\mathcal X}_1\tilde{\otimes}{\mathcal X}_2)$, $\partial=\partial_1
\otimes I+\eta\otimes\partial_2$. However, 
$({\mathcal X}_1\tilde{\otimes}{\mathcal X}_2,\partial)$ is the 
differential space associated to the total complex of the 
double complex

$$
\CD   
 X_1\otimes\wedge^{p-1 }L_1\tilde{\otimes} X_2\otimes
\wedge^q L_2@<d_p(\rho_1)\otimes I_q<<X_1\otimes\wedge^p L_1
\tilde{\otimes} X_2\otimes\wedge^q L_2 \\
@V(-1)^{p-1}I_{p-1}\otimes d_q(\rho_2)VV@V(-1)^pI_p\otimes 
d_q(\rho_2)VV\\
X_1\otimes\wedge^{p-1 }L_1\tilde{\otimes} X_2\otimes
\wedge^{q-1} L_2@<d_p(\rho_1)\otimes I_{q-1}<<X_1\otimes\wedge^pL_1
\tilde{\otimes} X_2\otimes\wedge^{q-1} L_2. \\
\endCD
$$

\indent We recall that given the Koszul complexes $(X_1\otimes
\wedge L_1,d(\rho_1))$ and $(X_2\otimes\wedge L_2,d(\rho_2))$,
according to the properties of the tensor product introduced in [14] and 
the considerations of sections 1 and 3 in [14], it is possible to 
consider the complex of Banach spaces defined by the
tensor product of $(X_1\otimes
\wedge L_1,d(\rho_1))$ and $(X_2\otimes\wedge L_2,
d(\rho_2))$, denoted by $(X_1\otimes
\wedge L_1,d(\rho_1))\tilde{\otimes}(X_2\otimes\wedge L_2,
d(\rho_2))$. This complex is the total complex of the above
double complex, i.e., for $k$ such that $0\le k\le n+m$,
the $k$ space is $\bigoplus_{p+q=k}X_1\otimes\wedge^p L_1\tilde{\otimes}X_2
\otimes\wedge^q L_2,$
and the boundary map, $d_k$, restricted to $X_1\otimes\wedge^p L_1\tilde{\otimes}X_2
\otimes\wedge^q L_2$ is $d_k=d_p(\rho_1)\otimes I_q+
(-1)^p\otimes d_q(\rho_2)$. In particular, $({\mathcal X}_1\tilde
{\otimes}{\mathcal X}_2,\partial)$ is the differential space of 
the complex $(X_1\otimes
\wedge L_1,d(\rho_1))\tilde{\otimes}(X_2\otimes\wedge L_2,
d(\rho_2))$. \par

\indent On the other hand, we may consider the direct sum of 
the Lie algebras $L_1$ and $L_2$, $L=L_1\times L_2$, which is 
a 
complex solvable finite dimensional Lie algebra, and the tensor product 
representation of $L$ in $X_1\tilde{\otimes} X_2$, i.e.,

$$
\rho=\rho_1\times\rho_2\colon L\to {\rm L}(X_1\tilde{\otimes} X_2),
\hbox{ }\rho_1\times\rho_2(l_1,l_2)=\rho_1(l_1)\otimes I+I
\otimes \rho_2(l_2),    
$$
where $I$ denotes the identity operator of both $X_2$ and $X_1$. In 
particular, we may consider the Koszul complex of the representation 
$\rho\colon L\to {\rm L}(X_1\tilde{\otimes}X_2)$, i.e., $(X_1
\tilde{\otimes} X_2\otimes\wedge L, d(\rho))$, and the differential 
space associated to it, i.e., $(X_1\tilde{\otimes}X_2\otimes
\wedge L, \tilde{\partial})$, where $\tilde{\partial}=\oplus_{k=
1}^{n+m}d_k(\rho)$.
\par

\indent In the following proposition we relate  
the complexes $(X_1\otimes
\wedge L_1,d(\rho_1))\tilde{\otimes}(X_2\otimes\wedge L_2,
d(\rho_2))$ and $(X_1
\tilde{\otimes} X_2\otimes\wedge L, d(\rho))$.\par  
 
\begin{pro} Let $X_1$ and $X_2$ be two complex 
Banach spaces, $L_1$ and $L_2$ two complex solvable finite
dimensional Lie algebras, and $\rho_i\colon L_i\to 
{\rm L}(X_i)$, $i=1$, $2$, two representations of the algebras. 
Then, the 
complexes $(X_1\tilde{\otimes}X_2\otimes\wedge L,d
(\rho))$ 
and $(X_1\otimes
\wedge L_1,d(\rho_1))\tilde{\otimes}(X_2\otimes\wedge L_2,
d(\rho_2))$ are isomorphic. In particular, the differential
spaces $({\mathcal X}_1\tilde{\otimes}{\mathcal X}_2,
\partial)$ and $(X_1\tilde{\otimes}X_2\otimes
\wedge L, \tilde{\partial})$ are isomorphic.
 \end{pro}
\begin{proof}

\indent First of all we consider the identification
$$
\Phi\colon \wedge L_1\otimes \wedge L_2\to\wedge L, \hbox{  }
\Phi (w_1\otimes w_2)=w_1\wedge w_2,
$$
for $w_1\in L_1$, $w_2\in L_2$. Now an easy calculation shows 
that for $k=0,\ldots ,n+m$ the map
$$
\tilde{\Phi}_k \colon \bigoplus_{p+q=k} X_1\otimes\wedge^p L_1
\tilde{\otimes}X_2\otimes\wedge^q L_2\to X_1\tilde{\otimes}
X_2\otimes\wedge^k L,
$$
$$
\tilde{\Phi}(x_1\otimes w_1\tilde{\otimes} x_2\otimes w_2)=x_1
\tilde{\otimes}x_2\otimes w_1\wedge w_2,
$$
is a well-defined isomorphism. Moreover, since $L$ is the direct sum of
$L_1$ and $L_2$, it is easy to prove that 
$\tilde{\Phi}=(\tilde{\Phi}_k)_{0\le k\le n+m}$ is a 
chain map, i.e., $\tilde{\Phi}(d)=d(\rho)\tilde{\Phi}$.\par
\end{proof}

\section{Joint spectra of the tensor product representation}

\indent In this section we consider two representation of Lie algebras in 
two Banach spaces and a tensor product of the Banach spaces in the sense of 
[14], and we describe
the S\l odkowski and the split joint spectra of the
tensor product representation of the direct sum of the algebras; see section 4. 
Moreover, for Hilbert 
spaces, the joint spectra are characterized in a precise manner. In addition, we apply our results to 
nilpotent systems of operators. 
We start by recalling the objects we shall 
work with.\par    

\indent Let $L_1$ and $L_2$ be  
two complex solvable
finite dimensional Lie algebras, $X_1$ and $X_2$ two complex Banach 
spaces, and $\rho_i \colon L_i\to {\rm L}(X_i)$, $i=1$, $2$,
two representations of Lie 
algebras. We suppose that there is a tensor product of $X_1$ 
and $X_2$ relative to $\langle X_1,X_1{'}\rangle$ and 
$\langle X_2,X_2{'}\rangle$, 
$X_1\tilde{\otimes} X_2$. Thus, as in section 4, we may consider
the direct sum of 
the Lie algebras $L_1$ and $L_2$, $L=L_1\times L_2$, which is a 
complex solvable finite dimensional Lie algebra, and the tensor
product representation of $L$ in $X_1\tilde{\otimes} X_2$, i.e.
$$
\rho=\rho_1\times\rho_2\colon L\to {\rm L}(X_1\tilde{\otimes} X_2),
\hbox{ }\rho_1\times\rho_2(l_1,l_2)=\rho_1(l_1)\otimes I+I
\otimes \rho_2(l_2),    
$$
where $I$ denotes the identity of $X_2$ and $X_1$ respectively. In 
particular, we may consider the Koszul complex of the representation 
$\rho\colon L\to {\rm L}(X_1\tilde{\otimes}X_2)$, $(X_1
\tilde{\otimes} X_2\otimes\wedge L, d(\rho))$.\par

\indent Now we state the most important result of this section. 
However, we first observe that the sets of characters of $L$ may be naturally
identified with the cartesian product of the sets 
of characters of $L_1$ and $L_2$. Indeed, it is clear that $L^*
\cong L_1^*\times L_2^*$. Moreover, since  
as Lie algebra $L$ is the direct sum of $L_1$ and $L_2$, 
if $[.,.]$ 
denotes the Lie bracket 
of $L$, then the restriction of $[.,.]$ to $L_1$ or $L_2$ 
coincides with the bracket of $L_1$ or $L_2$ respectively, and for $l_1
\in L_1$ and $l_2\in L_2$, $[l_1,l_2]=0$. Then, the map 
$$
H\colon L^*\to L_1^*\times L_2^*,\hbox{ }f\mapsto (f\circ\iota_1,f
\circ\iota_2)
$$
defines an identification of the characters of $L$ and the 
cartesian product of the characters of $L_1$ and $L_2$,
where $\iota_j\colon L_j\to L$ denotes 
the inclusion map, $j=1,$ $2$.
In the following theorem we use this identification.\par
 
\begin{thm}  Let $L_1$ and $L_2$ be two complex solvable 
finite dimensional Lie algebras, $X_1$ and $X_2$ two complex Banach 
spaces, and $\rho_i \colon L_i\to {\rm L}(X_i)$, $i=1$, $2$, two 
representations of Lie algebras. We suppose 
that there is a tensor product of $X_1$ and $X_2$ relative to 
$ \langle X_1,X_1{'}\rangle$ and $\langle X_2,X_2{'}\rangle$, 
$X_1\tilde{\otimes} 
X_2$. Then, if we consider the tensor product representation of $L=
L_1\times L_2$, $\rho=\rho_1\times
\rho_2\colon L\to {\rm L}(X_1\tilde{\otimes} X_2)$,
we have 
$$
{\rm (i)} \hbox{}\bigcup_{p+q=k}\sigma_{\delta ,p}(\rho_1)\times 
\sigma_{\delta ,q}(\rho_2)\subseteq\sigma_{\delta ,k}(\rho)
\subseteq sp_{\delta ,k } (\rho)\subseteq \bigcup_{p+q=k}sp_{\delta ,p}
(\rho_1)\times sp_{\delta ,q}(\rho_2),
$$
$$
 {\rm (ii)}\hbox{}\bigcup_{p+q=k}\sigma_{\pi ,p}(\rho_1)\times 
\sigma_{\pi ,q}(\rho_2)\subseteq\sigma_{\pi ,k}(\rho)
\subseteq sp_{\pi ,k} (\rho)\subseteq \bigcup_{p+q=k}sp_{\pi ,p}
(\rho_1)\times sp_{\pi ,q}(\rho_2).
$$
In particular, if $X_1$ and $X_2$ are Hilbert spaces, the above 
inclusions are equalities.\par
\end{thm}
\begin{proof}
\indent We begin with the first statement. \par

\indent We consider $\alpha\in\sigma_{\delta ,p}(\rho_1)$, 
$\beta\in \sigma_{\delta , q}(\rho_2)$, $p+q=k$, and the Koszul complexes 
associated to the representations $\rho_1-\alpha\colon L_1\to{\rm L}
(X_1)$ and $\rho_2-\beta\colon L_2\to{\rm L}(X_2)$,
$(X_1\otimes\wedge L_1,d(\rho_1-\alpha))$ and
$(X_2\otimes\wedge L_2,d(\rho_2-\beta))$ respectively. Then, there is 
$p_1$, $0\le p_1\le p$, and $q_2$, $0\le q_2\le q$, such that 
$H_{p_1}(X_1\otimes\wedge L_1,d(\rho_1-\alpha))\neq 0$ and that 
$H_{q_2}(X_2\otimes\wedge L_2,d(\rho_2-\beta))\neq 0$.\par

\indent In addition, if we consider 
the differential spaces associated to the Koszul complexes
of $\rho_1-\alpha$ and $\rho_2-\beta$, $({\mathcal X}_1,\partial_1)$ 
and $({\mathcal X}_2,\partial_2)$ respectively, then by
[14, Theorem 2.2] we have $H_*({\mathcal X}_1\tilde{\otimes}{\mathcal X}_2)
\neq 0$. Moreover, since $({\mathcal X}_1\tilde{\otimes}{\mathcal X}_2,
\partial)$ is the differential space of $(X_1\otimes\wedge L_1,
d(\rho_1-\alpha))\tilde{\otimes} 
(X_2\otimes\wedge L_2,d(\rho_2-\beta))$, according to the structure
of the map $\varphi$ in [14, Theorem 2.2], we have $H_{p_1+q_2}
((X_1\otimes\wedge L_1,d(\rho_1-\alpha))\tilde{\otimes} 
(X_2\otimes\wedge L_2,d(\rho_2-\beta)))\neq 0$.
However, according to Proposition 4.4, since $(\rho_1-\alpha)\times
(\rho_2-\beta)=\rho-(\alpha ,\beta)$, we have $H_{p_1+q_2}(X_1\tilde{\otimes}X_2\otimes
\wedge L,d(\rho-(\alpha ,\beta )))\neq 0$. In particular, since $0\le p_1+q_2
\le p+q=k$, $(\alpha ,\beta)\in
\sigma_{\delta ,k}(\rho)$.\par

\indent The middle inclusion is clear.\par

\indent With regard to the inclusion on the right, we prove that
if $(\alpha ,\beta)$ does not belong to $\bigcup_{p+q=k}sp_{\delta ,p}
(\rho_1)\times sp_{\delta ,q}(\rho_2)$, then $(\alpha ,\beta)$ does 
not belong to $sp_{\delta ,k} (\rho)$. To this end, we shall construct a 
homotopy  operator. There are several cases to be considered. \par

\indent We first suppose that $\alpha\notin sp_{\delta ,k}
(\rho_1)$. Thus, the complex $(X_1\otimes\wedge L_1 ,d(\rho_1-
\alpha))$ is split for $p=0,\ldots ,k$, i.e., for $p=0,\ldots ,k$ there are
bounded linear operators $h_p\colon X_1\otimes\wedge^p L_1
\to X_1\otimes \wedge^{p+1}L_1$, such that
$h_{p-1}d_p(\rho_1-\alpha) +d_{p+1}(\rho_1-\alpha)h_p=
I_p$, where $I_p$ denotes the identity of $X_1\otimes\wedge^p L_1$. 
Then, if $p$ and $q$ are such that
$0\le p+q\le k$, we define
$$
H_{p,q}\colon X_1\otimes \wedge^p L_1\tilde{\otimes} X_2
\otimes
\wedge^qL_2 \to X_1\otimes \wedge^{p+1}L_1\tilde{\otimes} 
X_2\otimes
\wedge^qL_2,
\hbox{ }H_{p,q}=h_p\otimes I_q,
$$  
where $I_q$ denotes the identity map of $X_2\otimes\wedge^q L_2$. 
We observe that since ${\mathcal L} (X_1\otimes\wedge^p L_1,X_1
\otimes
\wedge^{p+1}L_1)={\rm L} (X_1\otimes\wedge^pL_1,X_1\otimes\wedge
^{p+1}L_1)$, $H_{p,q}$ is a well-defined map. Moreover, a direct 
calculation shows that the maps $H_r$, $r=0,\ldots , k$, $H_r=
\bigoplus_{p+q=r} H_{p,q}$, define a homotopy operator for the  
complex  $(X_1\otimes\wedge L_1)\tilde{\otimes}
(X_2\otimes\wedge L_2)$, for $r=0,\ldots ,k$. Thus, according to 
Proposition 4.4 the complex $(X_1\tilde{\otimes}X_2\otimes\wedge L,
d(\rho-(\alpha,\beta)))$ is split for $r=0,\ldots ,k$, i.e., $(\alpha,\beta)$
does not belongs to  
$sp_{\delta ,k}(\rho)$. \par

\indent By a similar argument, it is possible to prove that if $\beta\notin 
sp_{\delta ,k}(\rho_2)$, then $(\alpha ,\beta)$ does not belongs 
to $sp_{\delta ,k}(\rho)$.
Thus, we may suppose that $\alpha\in sp_{\delta ,k}(\rho_1)$
and $\beta\in sp_{\delta ,k}(\rho_2)$.\par

\indent Now, since $(\alpha ,\beta)$ does not belong to $\bigcup_{p+q=k}
sp_{\delta ,p}(\rho_1)\times sp_{\delta ,q}(\rho_2)$ and $\alpha\in 
sp_{\delta ,k}(\rho_1)$, we have
$\beta\notin sp_{\delta ,0}(\rho_2)$. Similarly, since $\beta\in 
sp_{\delta ,k}(\rho_2)$ we have  
$\alpha\notin sp_{\delta ,0}(\rho_1)$. Thus, there is $p_1$,  
$1\le p_1\le k$, such that $\alpha \notin sp_{\delta ,p_1-1}(\rho_1)$, 
$\alpha\in sp_{\delta ,p_1}$, and
$\beta\notin sp_{\delta ,k-p_1}(\rho_2)$. \par

\indent In order to construct a homotopy operator for the 
Koszul complex associated to $\rho-(\alpha ,\beta)$, 
$(\alpha ,\beta)$ as in the last paragraph, it is necessary to 
consider several cases. In fact, we shall define the operator according to 
the relation of $p$ and $q$ with $p_1$ and $k-p_1$ respectively, and 
for each particular case, we shall prove that it is a homotopy. At the end of the 
proof, it is clear that this map is a well-defined homotopy for 
the Koszul complex of $\rho$ at $r=0,\ldots ,k$. \par   

\indent Moreover, according to Proposition 4.4, it is enough to prove that the
complex $(X_1\otimes\wedge L_1, 
d(\rho_1-\alpha))\tilde{\otimes}(X_2\otimes\wedge L_2, 
d(\rho_2-\beta))$ is split in 
dimension 
$r=0,\ldots ,k$. Now, the $r$-space of this
complex is $\bigoplus_{p+q=k}X_1\otimes \wedge^p L_1\tilde
{\otimes} X_2\otimes \wedge^q L_2$. We  
construct the operator $H_{p,q}$ satisfying the homotopy 
identity for $p$ and  
$q$ such that $p+q=r$, and then we verify that
$(H_r)_{0\le r\le k}$ is a homotopy operator for the complex, 
where $H_r=\bigoplus_{p+q=r} H_{p,q}$. The construction of the maps 
$H_{p,q}$ is divided into five cases. \par

\indent We first suppose that $0\le p\le p_1-1$ and $0\le q\le k-p_1$. 
Then, we have well-defined maps
$$
X_1\otimes\wedge^{p-1} L_1\xrightarrow{h_{p-1}}X_1\otimes
\wedge^p L_1\xrightarrow{h_p}X_1\otimes\wedge^{p+1}L_1,
$$
such that $d_{p+1}(\rho_1-\alpha)h_p+h_{p-1}d_p(\rho_1-\alpha)=
I_p$, where $I_p$ denotes the identity of $X_1\otimes 
\wedge^p L_1$, and
 $$
X_2\otimes\wedge^{q-1} L_2\xrightarrow{g_{q-1}}X_2\otimes\wedge^q 
L_2\xrightarrow{g_q}X_2\otimes\wedge^{q+1}L_2,
$$
such that $d_{q+1}(\rho_2-\beta)g_q+g_{q-1}d_q(\rho_2-\beta)=
I_q$, where $I_q$ denotes the identity of $X_2\otimes \wedge^q L_2$.\par
Thus, we may define the map 
$$
H_{p,q}\colon X_1\otimes\wedge^pL_1\tilde{\otimes}X_2\otimes
\wedge^q L_2\to
X_1\otimes\wedge^{p+1}L_1\tilde{\otimes}X_2\otimes\wedge^q L_2
\oplus X_1\otimes\wedge^p L_1\tilde{\otimes}X_2\otimes
\wedge^{q+1}L_2,
$$
$$
H_{p,q}=1/2(h_p\otimes I_q\oplus (-1)^p I_p\otimes g_q).
$$
We observe that according to the properties of the tensor product,  
$H_{p,q}$ is a well-defined map. \par

\indent In addition, since $p-1<p\le p_1-1$ and $q-1<q\le k-p_1$, we 
may define the maps $H_{p-1,q}$ and $H_{p,q-1}$. However, a 
direct calculation shows that in $X_1\otimes\wedge^pL_1\tilde
{\otimes}X_2\otimes\wedge^q L_2$, we have
$$
d_{r+1}H_{p,q}+(H_{p-1,q}\oplus H_{p,q-1})d_r=I,
$$
where $d$ and $I$ denote the boundary and the identity of the complex $(X_1\otimes\wedge L_1, 
d(\rho_1-\alpha))\tilde{\otimes}(X_2\otimes\wedge L_2, 
d(\rho_2-\beta))$ respectively.\par

\indent In the second case we suppose that $p$ and $q$ are
such that $p\le p_1-1$ and 
$q=k-p_1+1$.
Then, we know that for $q=0,\ldots ,k-p_1$
there are bounded maps 
$$
X_2\otimes\wedge^{q-1} L_2\xrightarrow{g_{q-1}}X_2\otimes
\wedge^q L_2\xrightarrow{g_q}X_2\otimes\wedge^{q+1}L_2,
$$
such that $d_{q+1}(\rho_2-\beta)g_q+g_{q-1}d_q(\rho_2-\beta)
=I_q$. \par

\indent In addition, we may suppose that the maps $g_q$
satisfy the preliminary 
facts recalled before Proposition 4.2, for $q= 0,\ldots , k-p_1$. Moreover, 
according to Proposition 4.2, 
we have $X_1\otimes\wedge^p L_1
\tilde{\otimes }X_2\otimes\wedge^q L_2= Ker(I_p\otimes 
d_{q}(\rho_2-\beta))
\oplus R(I_p\otimes g_{q-1})$, for
$p=0,\ldots , n$ and $q=0,\ldots , k-p_1+1$.\par

\indent On the other hand, it is easy to prove that \par
(i) $d_p(\rho_1-\alpha )\otimes I_q(R(I_p\otimes g_{q-1}))\subseteq
R(I_{p-1}\otimes g_{q-1})$ and $d_p(\rho_1-\alpha)\otimes I_q(Ker(I_p
\otimes d_q(\rho_2-\beta )))\subseteq
Ker(I_{p-1}\otimes d_q(\rho_2-\beta ))$,\par
(ii) $I_p\otimes d_q(\rho_2-\beta)(Ker (I_p\otimes d_q(\rho_2-\beta) ))=0$ and 
$I_p\otimes d_q(\rho_2-\beta)(R (I_p\otimes g_{q-1}))=R(I_p\otimes 
d_q(\rho_2-\beta))=Ker(I_p\otimes d_{q-1}(\rho_2-\beta))$.\par

\indent Furthermore, as in the first case, 
we have well-defined maps, 
$(h_p)_{0\le p\le p_1-1}$, such that $h_p\colon X_1\otimes
\wedge^p 
L_1\to X_1\otimes\wedge^{p+1} L_1$,
and that $d_{p+1}(\rho_1-\alpha)h_p+h_pd_{p-1}(\rho_1-\alpha)
=I_p$, for $p=0,\ldots ,p_1-1 $. A straightforward 
calculation shows that\par
(iii) $h_p\otimes I_q (R(I_p\otimes g_{q-1}))\subseteq R(I_{p+1}
\otimes g_{q-1})$, and that $h_p\otimes I_q (Ker(I_p\otimes 
d_q(\rho_2-\beta)))\subseteq Ker(I_{p+1}\otimes d_q(\rho_2-\beta))$. \par

\indent Now, for $p=0, \ldots , p_1-1$ and  
$q=k-p_1+1$ we define $H_{p,q}$ as follows:
$$
H_{p,q}\mid R(I_p\otimes g_{q-1})=
1/2(h_p\otimes I_q),\hbox{} H_{p,q}\mid Ker(I_p\otimes d_q
(\rho_2-\beta))=h_p\otimes I_q.
$$
According to the properties of the tensor product,
the map $H_{p,q}$ is well defined.\par

\indent In addition, for $p=0,\ldots , p_1-1$ and
$q-1=k-p_1$, according to the first case, we have the well-defined 
map $H_{p,q-1}$. On the other hand, for $p-1$, 
$p=0,\ldots , p_1-1$, and $q=k-p_1+1$, we may define $H_{p-1,q}$ 
in a similar way as we did with $H_{p,q}$.\par

\indent Now, using (i)-(iii) it is easy to prove that
$$
d_{r+1}H_{p,q} +(H_{p,q-1}\oplus H_{p-1,q}) d_r=I,
$$
where $d$ and $I$ are as above.\par

\indent In the third case $p=0,\ldots, p_1-1$ and $q>k-p_1+1$.
There are two subcases: $q-1>k-p_1+1$ 
and $q-1=k-p_1+1$. We begin with the first
subcase.\par

\indent For $ p=0,\ldots ,p_1-1$ and $q>q-1>k-p_1+1$ we define 
$$
H_{p,q}\mid X_1\otimes\wedge^p L_1\tilde{\otimes}X_2\otimes
\wedge^q L_2\to X_1\otimes\wedge^{p+1} L_1\tilde{\otimes}X_2
\otimes\wedge^q L_2, \hbox{} H_{p,q}=h_p\otimes I_q.
$$
According to the properties of the tensor product, $H_{p,q}$ is a well-defined 
map.\par
\indent Moreover, since $q-1>q>k-p_1+1$, we may define 
$H_{p-1,q}$ and $H_{p,q-1}$ in a similar way. Then, an easy calculation shows that
$$
d_{r+1}H_{p,q} +(H_{p,q-1}\oplus H_{p-1,q}) d_r=I.
$$
\indent On the other hand, for $p=0,\ldots , p_1-1$ and $q=k-p_1+1$, 
we define $H_{p,q}=h_p\otimes I_q$. Furthermore,
for $p-1$ and $q=k-p_1+1$, we may define $H_{p-1,q}=
h_{p-1}\otimes I_q$, but for $p=0,\ldots , p_1-1$ and $q-1= k-p_1$, 
$H_{p,q-1}$ was defined in the second case. 
However, a direct calculation shows that
$$
d_{r+1}H_{p,q} +(H_{p,q-1}\oplus H_{p-1,q}) d_r=I.
$$  
 
\indent In the fourth case $p=p_1$ and $q\le k-p_1$. This case is similar 
to the second one.\par

\indent We consider the complex associated to the representation 
$\rho_1-\alpha$, i.e., $(X_1\otimes \wedge L_1, d(\rho_1-\alpha ))$. We know that 
for $p=0,\ldots ,p_1-1$ 
there are bounded maps
$$
X_1\otimes\wedge^{p-1} L_1\xrightarrow{h_{p-1}}X_1\otimes\wedge^p L_1
\xrightarrow{h_p}X_1\otimes\wedge^{p+1}L_1,
$$
such that $d_{p+1}(\rho_1-\alpha)h_p+h_{p-1}d_p(\rho_1-\alpha)
=I_p$. \par
   
\indent Moreover, as in the second case,  
we may suppose that the maps $h_p$ satisfy the preliminary facts recalled before Proposition
4.2, for $p=0,\ldots ,p_1-1$. Furthermore, according to Proposition 4.2 and Remark 4.3,
we have 
$X_1\otimes\wedge^p L_1
\tilde{\otimes }X_2\otimes\wedge^q L_2= Ker(d_p(\rho_1-\alpha )\otimes I_q)
\oplus R(h_{p-1}\otimes I_q)$, for $p=0,\ldots , p_1$ and 
$q=0,\ldots , m$.\par

\indent As in the second case, it is easy to prove that \par
(i) $I_p\otimes d_q(\rho_2-\beta ) (R(h_{p-1}\otimes I_q))\subseteq
R(h_{p-1}\otimes I_{q-1})$ and $I_p\otimes d_q(\rho_2-\beta )
(Ker(d_p(\rho_1-\alpha )\otimes I_q))\subseteq
Ker(d_p(\rho_1-\alpha )\otimes I_{q-1})$,\par
(ii) $d_p(\rho_1-\alpha )\otimes I_q(Ker (d_p(\rho_1-\alpha)\otimes I_q))=0$ 
and $d_p(\rho_1-\alpha )\otimes I_q(R (h_{p-1}\otimes I_q))=
R(d_p(\rho_1-\alpha)\otimes I_q)=Ker(d_{p-1}(\rho_1-\alpha)\otimes I_q)$.\par

\indent In addition, for $q=0,\ldots ,k-p_1 $, we have well-defined maps, 
$(g_q)_{0\le q\le k-p_1}$, such that $g_q\colon X_2\otimes\wedge^q 
L_2\to X_2\otimes\wedge^{q+1} L_2$,
and that $d_{q+1}(\rho_2-\beta )g_p+g_{q-1}d_{q-1}(\rho_2-\beta)=I_q$. A 
straightforward calculation shows \par
(iii) $I_p\otimes g_q (R(h_{p-1}\otimes I_q ))\subseteq R(h_{p-1}
\otimes I_{q+1})$ and $I_p\otimes g_q (Ker(d_p(\rho_1-\alpha)\otimes 
I_q))\subseteq Ker(d_p(\rho_1-\alpha)\otimes I_{q+1})$. \par

\indent Now for $p= p_1$ and $0\le q\le k-p_1$,
we define $H_{p_1,q}$ as follows:
$$
H_{p_1,q}\mid R(h_{p_1-1}\otimes I_q)=(-1)^p1/2(I_p\otimes g_q),
$$
$$
H_{p_1,q}\mid Ker(d_{p_1}(\rho_1-\alpha )\otimes I_q )=(-1)^pI_p
\otimes g_q.
$$
According to the properties of the tensor product, $H_{p_1,q}$
is a well-defined map.\par

\indent In addition, according to the first case, we have 
the well-defined map $H_{p_1-1,q}$, $p=p_1-1$ and
$q=0,\ldots ,k-p_1$. On the other hand, we may define $H_{p_1,q-1}$ 
like $H_{p_1,q}$, $p=p_1$ and
$q-1=0,\ldots , k-p_1$.\par

\indent Now, as in the second case, using (i)-(iii) 
it is easy to prove that
$$
d_{r+1}H_{p_1,q} +(H_{p_1-1,q}\oplus H_{p_1,q}) d_r=
I.
$$

\indent In the last case, we have  
$p\ge p_1+1$ and $q=0,\ldots ,k-p_1$.  
Moreover, as in the third case, there are two subcases: 
$p-1\ge p_1+1$, and $p-1=p_1$.
We begin with the first subcase.\par

\indent For $ p>p-1\ge p_1+1$ and $q=0,\ldots ,k-p_1$, we define 
$$
H_{p,q}\mid X_1\otimes\wedge^p L_1\tilde{\otimes}X_2\otimes
\wedge^q L_2\to X_1\otimes\wedge^p L_1\tilde{\otimes}X_2
\otimes\wedge^{q+1} L_2, \hbox{} H_{p,q}=I_p\otimes g_q.
$$
According to the properties of the tensor product, the map $H_{p,q}$ is
well defined.\par

\indent Since $p-1>p_1+1$, we may define $H_{p-1,q}$ 
and $H_{p,q-1}$. Then, an easy calculation shows that
$$
d_{r+1}H_{p,q} +(H_{p,q-1}\oplus H_{p-1,q}) d_r=I.
$$
\indent On the other hand, for $p-1= p_1$ and $q=0,\ldots ,k-p_1$, 
we define $H_{p,q}=I_p\otimes g_q$. Moreover,
for $p-1=p_1$ and $q$, $H_{p-1,q}$ was defined in the fourth case,
and for $p$ and $q-1$, we may define $H_{p,q-1}=I_p\otimes 
g_{q-1}$. 
However, a direct calculation shows 
$$
d_{r+1}H_{p,q} +(H_{p-1,q}\oplus H_{p,q-1}) d_r=I.
$$  

\indent Since we considered all the possible cases for 
$p$ and $q$, $0\le p+q\le k$, if for $r=0,\ldots , k$ we consider
the map $H_r=\bigoplus_{p+q=r}H_{p,q}$, then the above computations
show that $(H_r)_{0\le r\le k}$ is a homotopy for the 
complex $(X_1\otimes\wedge L_1,d(\rho_1-\alpha))\tilde
{\otimes}(X_2\otimes\wedge L_2, d(\rho_2-\beta))$. Thus, 
according to Proposition 4.4,
$(\alpha ,\beta)$ does not belong to $sp_{\delta ,k}(\rho)$.\par

\indent The second part of the theorem may be proved 
by a similar argument, using the second half of Proposition 4.2
for the inclusion on the right. 

\end{proof}   

\indent We recall that in Chapter 3, Section 3 of [21], the axiomatic tensor
in [14] was generalized. However, as it was explained
in Chapter 3, Section 3 of [21], the objective was to simplify 
the form of the axioms rather than to generalize the definition
in [14]; in addition, the known applications of both tensor products
coincide. As in the way we prove the main results in this work,
the definition in [14] is more useful than the one in [21], we proved 
Theorem 5.1 and shall prove the other results 
for the tensor product introduced in [14]. In 
particular, Theorem 5.1
may be seen as an extension of [21, Korollar 3.6.8]
for the tensor product in [14].
However, we believe that with the axiomatic tensor product
introduced in Chapter 3, Section 3 of [21], it would be possible to obtain results,
which would be similar to the main ones in this 
work.\par  

\indent Now we consider nilpotent systems of operators and we
prove a variant of Theorem 5.1 for this case. This result 
extends [21, Satz 3.7.2] for the tensor product in [14].
Moreover, the following theorem is an 
extension of well-known results for
commuting tuples of operators; see [9], [10], [28] and [14].
First we give a definition.\par

\indent Let $X$ be a complex Banach space and $T=
(T_1,\ldots ,T_n)$ an $n$-tuple of operators defined in $X$,
such that the linear subspace of ${\rm L}(X)$ generated
by them, $\langle T_i\rangle_{1\le i\le n}=L$, is a nilpotent Lie subalgebra 
of ${\rm L}(X)$. We consider the representation defined
by the inclusion $\iota_L\colon L\to{\rm L}(X)$. Then, if 
$\sigma $ denotes a subset of a joint spectrum defined for representations
of complex solvable finite dimensional Lie algebras, we denote
the set $
\{(\alpha (T_1),\ldots ,\alpha (T_n)): 
\alpha\in \sigma(\iota_L)\}$ by $\sigma (T)$.\par

\begin{thm}Let $X_1$ and $X_2$ be two complex Banach 
spaces. We suppose 
that there is a tensor product of $X_1$ and $X_2$ with respect 
to $\langle X_1,X_1{'}\rangle$ and $\langle X_2,X_2{'}\rangle$, 
$X_1\tilde{\otimes} 
X_2$. Let $a=(a_1,\ldots ,a_n)$ and $b=(b_1,\ldots ,b_m)$ be two 
tuples of operators, $a_i\in {\rm L}(X_1)$, $1\le i\le n$, and 
$b_j\in {\rm L}(X_2)$, $1\le j\le m$, such that the vector 
subspaces generated by them, $\langle a_i\rangle_{1\le i\le n}$ and
$\langle b_j\rangle_{1\le j\le m}$, are nilpotent Lie subalgebras
of ${\rm L}(X_1)$ and ${\rm L}(X_2)$ respectively.
We consider the $(n+m)$-tuple of  
operators defined in $X_1\tilde{\otimes}X_2$,
$c=(a_1\otimes I ,\ldots ,a_n\otimes I ,I\otimes b_1,\ldots ,
I\otimes b_m)$,
where $I$ denotes the identity of $X_2$ and $X_1$ respectively. Then
we have 
$$
{\rm (i)} \hbox{}\bigcup_{p+q=k}\sigma_{\delta ,p}(a)\times 
\sigma_{\delta ,q}(b)\subseteq\sigma_{\delta ,k}(c)
\subseteq sp_{\delta ,k} (c)\subseteq \bigcup_{p+q=k}sp_{\delta ,p}
(a)\times sp_{\delta ,q}(b),
$$
$$
 {\rm (ii)}\hbox{}\bigcup_{p+q=k}\sigma_{\pi ,p}(a)\times 
\sigma_{\pi ,q}(b)\subseteq\sigma_{\pi ,k}(c)
\subseteq sp_{\pi ,k} (c)\subseteq \bigcup_{p+q=k}sp_{\pi ,p}
(a)\times sp_{\pi ,q}(b).
$$
\indent In particular, if $X_1$ and $X_2$ are Hilbert spaces, 
the above inclusions are equalities.\par
\end{thm}
 
\begin{proof}
 
\indent We consider the nilpotent Lie algebras $L_1=\langle a_i
\rangle_{1\le i\le n}$
and $L_2=\langle b_j\rangle_{1\le j\le m}$, and the  
representations of the above algebras defined by the inclusion, i.e.,
$$
\iota_1\colon L_1\to {\rm L}(X_1),\hskip2cm\iota_2\colon L_2
\to {\rm L}(X_2).
$$
Then, if we consider the representation $\iota=\iota_1\times
\iota_2\colon L_1\times L_2\to {\rm L}(X_1
\tilde{\otimes}X_2)$, according to Theorem 5.1 we have 
$$
 \hbox{}\bigcup_{p+q=k}\sigma_{\delta ,p}(\iota_1)\times 
\sigma_{\delta ,q}(\iota_2)\subseteq\sigma_{\delta ,k}(\iota)
\subseteq sp_{\delta ,k} (\iota)\subseteq \bigcup_{p+q=k}sp_{\delta ,p}
(\iota_1)\times sp_{\delta ,q}(\iota_2).
$$
\indent Now, if we consider the identification of the characters
of $L_1\times L_2$ with the cartesian product of the characters of
$L_1$ and $L_2$, it is clear that $\sigma_{\delta ,p}(a)\times 
\sigma_{\delta ,q}(b)$ coincides with the set
$$
\{(\alpha (a_1 ),\ldots ,\alpha (a_n),\beta (b_1),\dots , 
\beta (b_m)): (\alpha ,\beta )
\in \sigma_{\delta ,p}(\iota_1)\times \sigma_{\delta ,q}(\iota_2)\}.
$$
\indent Similarly, $sp_{\delta ,p}(a)\times sp_{\delta ,q}(b)$
coincides with
$$
\{(\alpha (a_1),\ldots ,
\alpha (a_n),\beta (b_1),\ldots ,\beta (b_m)):
(\alpha ,\beta )\in sp_{\delta ,p}(\iota_1)\times sp_{\delta ,q}
(\iota_2)\}.
$$

\indent On the other hand, we consider the nilpotent Lie subalgebra of
${\rm L} (X_1\tilde{\otimes }X_2)$ generated by the
elements of the $(n+m)$-tuple $c$; we denote it by $L$.
Then, if $\iota\colon L\to{\rm L} (X_1\tilde{\otimes }X_2)$ 
is the representation defined by the inclusion, we have 
$\iota_1\times \iota_2=\iota\circ h$, where $h\colon L_1\times
L_2\to L$ is the epimorphism of Lie algebras that satifies 
$h(a_i)=a_i\otimes I$ and $h(b_j)=I\otimes b_j$,
for
$i=1,\ldots ,n$ and $j=1,\ldots ,m$. In particular, according to Proposition 3.11 we have 
$$
\sigma_{\delta ,k}(\iota_1\times \iota_2)=\sigma_{\delta ,k}
(\iota)\circ h, \hbox{  } sp_{\delta ,k}(\iota_1\times \iota_2)=
sp_{\delta ,k}(\iota)\circ h.
$$
\indent However, 
$$\sigma_{\delta ,k}(c)=\{(\gamma\circ h(a_1),
\ldots ,\gamma\circ h (a_n),\gamma\circ h (b_1),\ldots ,
\gamma\circ h (b_m)):\gamma\in\sigma_{\delta ,k}
(\iota )\}.
$$
Moreover, according to Proposition 3.11 $\sigma_{\delta ,k}(c)$ coincides with 
$$
\{(\alpha (a_1),\ldots ,\alpha (a_n),\beta (b_1),\ldots ,
\beta (b_m)): (\alpha ,\beta)\in \sigma_{\delta ,k }
(\iota_1\times \iota_2)\}.
$$

Similarly,
$$
sp_{\delta ,k}(c)=\{(\alpha (a_1),\ldots ,\alpha (a_n),\beta (b_1),
\ldots ,
\beta (b_m)): (\alpha ,\beta)\in sp_{\delta ,k }
(\iota_1\times \iota_2)\}.
$$
\indent Thus, the above equalities prove the
first part of the theorem. \par

\indent The second statement may be proved by a similar
argument.\par
\end{proof}

\section{Fredholm joint spectra of the tensor product representation}

\indent In this section we consider two representation of Lie algebras 
in two Banach spaces and
a tensor product of the Banach spaces in the sense of [14],
and we describe the essential S\l odkowski and 
the essential split joint spectra of the tensor product representation of the 
direct sum of the algebras; see section 4. In addition, we apply our
results to nilpotent systems of operators.
We first prove a result needed for the main
theorem in this section.\par
 
\begin{pro} Let $X_1$ and 
$X_2$ be two Banach spaces. We suppose that there is a 
tensor product  of $X_1$
and $X_2$ relative to $\langle X_1, X_1{'}\rangle$ and $\langle 
X_2,X_2{'}\rangle$,
$X_1\tilde{\otimes}X_2$. We consider in $X_1$ and $X_2$ two
projectors with finite dimensional range , $k_1$ and $k_2$ respectively. 
Then $k_1\otimes k_2\in {\rm L}(X_1\tilde{\otimes}X_2)$ is a 
projector with 
finite dimensional range. In fact, $R(k_1\otimes k_2)=R(k_1)\otimes R(k_2)$.
\end{pro}

\begin{proof}

\indent According to the properties of the tensor product, it is clear that 
$k_1\otimes k_2$ is a projector and that
$R(k_1\otimes k_2)\supseteq R (k_1)\otimes R(k_2)$.\par

\indent In order to prove the other inclusion, we consider 
a base of $R(k_1)$, $(v_i)_{1\le i\le n}$, i.e.,
$R(k_1)=\langle v_i\rangle_{1\le i\le n}$.
Then we have $X_1=Ker( k_1)
\oplus_{i=1}^n \langle v_i\rangle$. Moreover, if for each $s=1,\ldots , n$ we
consider the map $l_s\colon X_1\to \Bbb C$, $l_s\mid Ker (k_1)
\equiv 0$,
$l_s\mid \langle v_i\rangle\equiv 0$, $i=1,\ldots , n$, $i\neq s$, and $l_s(v_s)
=1$, then we may define the maps $f_{v_i l_i}\colon X_1\to X_1$, 
$f_{v_i l_i}(x_1)=l_i(x_1)v_i$, for $x_1\in X_1$. Now, an easy
calculation shows that $k_1=\sum_{i=1}^n f_{v_il_i}$.\par

\indent In addition, we may consider a base of $R(k_2)$,
$(v_j{'})_{1\le j\le m}$, 
and then we have $X_2=Ker ( k_2)\oplus_{j=1}^m \langle v_j^{'}\rangle$. 
Moreover, if for $j=1,\ldots ,m$ we consider the maps $h_j\colon X_2\to \Bbb C$, 
$h_j\mid Ker (k_2)\equiv 0$, $h_j(v_t^{'})=0$, $t=1,
\ldots , m$,
$t\neq j$, and $h_j(v_j^{'})=1$, then we may define the maps
$f_{v_j^{'}h_j}\colon X_2\to X_2$, $f_{v_j^{'}h_j}(x_2)=
h_j(x_2)v_j{'}$, for $x_2\in X_2$. As above, an easy 
calculation shows that $k_2=
\sum_{j=1}^m f_{v_j^{'} h_j}$.\par

\indent Now, according to the properties of the tensor product, we have 
$$
k_1\otimes k_2= \sum_{i,j} f_{v_il_i}\otimes f_{v_j^{'}h_j}=
\sum_{i,j}  f_{v_il_i}\otimes I\circ I\otimes f_{v_j^{'}h_j}.
$$ 

\indent Moreover, by [14, Lemma 1.1], for each $l_i$, $i=1,\ldots ,n$, there
is a map $f_{l_i} \colon X_1\tilde{\otimes}X_2\to X_2$ such that
$f_{x_1l_i}\otimes I(z)=x_1\otimes f_{l_i}(z)$, 
for $x_1\in X_1$
and $z\in X_1\tilde{\otimes}X_2$, where $f_{x_1l_i}\colon X_1\to
X_1$ is the map $f_{x_1l_i}(x)=l_i(x)x_1$. In addition, for each
$h_j$, $j=1,\ldots ,m$, there is a map $g_{h_j}\colon X_1
\tilde{\otimes}X_2\to X_1$
such that $I\otimes f_{x_2h_j}(z)=g_{h_j}(z)\otimes x_2$, for
$x_2\in X_2$ and $z\in X_1\tilde{\otimes}X_2$, where 
$f_{x_2h_j}$ has a definition similar to the one of $f_{x_1l_i}$. 
In particular, for
$z\in X_1\tilde{\otimes}X_2$ we have 

\begin{align*}
k_1\otimes k_2(z)&=\sum_{i,j}  f_{v_il_i}\otimes I\circ I\otimes 
f_{v_j^{'}h_j}(z)=\sum_{i,j}  f_{v_il_i}\otimes I( g_{h_j}(z)
\otimes v_j^{'})\\
                                  &=\sum_{i,j} v_i\otimes f_{l_i}(g_{h_j}(z)
\otimes v_j^{'}).\\
\end{align*}

\indent Thus, $R (k_1\otimes k_2)\subseteq R( K_1)\otimes X_2$.\par

\indent Moreover, since $k_2$ is a projection, if for $z\in X_1
\tilde
{\otimes}X_2$ we denote $z_{ij}=f_{l_i}(g_{h_j}(z)\otimes 
v_j^{'})$, then
we have $z_{ij}=k_2(z_{ij} )+(I-k_2)(z_{ij})$. In particular
$$
k_1\otimes k_2(z)=\sum_{i,j} v_i\otimes z_{ij}=\sum_{i,j} v_i
\otimes k_2(z_{ij})+\sum_{i,j} v_i\otimes (I-k_2)(z_{ij}).
$$ 

\indent However, since 
$k_1\otimes k_2$ is a projector in $X_1\tilde
{\otimes}X_2$, we have 
$$
k_1\otimes k_2(z)=(k_1\otimes k_2)^2(z)=\sum_{i,j} v_i\otimes 
k_2(z_{ij}).
$$
In particular, $R( k_1\otimes k_2)\subseteq R( k_1)\otimes R(K_2)$.\par
\end{proof}
 
\indent Now we state the main result of this section. The
following theorem is an extension of [14, Theorem 3.2].\par 

\begin{thm}  Let $X_1$ and $X_2$ be two complex Banach 
spaces, $L_1$ and $L_2$ two complex solvable 
finite dimensional Lie algebras, and $\rho_i \colon L_i\to 
{\rm L}(X_i)$, $i=1$, $2$,
two representations of Lie algebras. 
We suppose
that there is a tensor product of $X_1$ and $X_2$ 
relative 
to $\langle X_1,X_1{'}\rangle$ and $\langle X_2,X_2{'}\rangle$, $X_1\tilde{\otimes} 
X_2$. Then, if we consider the tensor product representation 
of the direct sum of $L_1$ and $L_2$, $\rho=\rho_1\times
\rho_2\colon L_1\times L_2\to {\rm L}(X_1\tilde{\otimes} X_2)$,
we have 
\begin{align*}
{\rm (i)} \hbox{}&\bigcup_{p+q=k}\sigma_{\delta ,p,e}(\rho_1)\times 
\sigma_{\delta ,q}(\rho_2 )\bigcup\bigcup_{p+q=k}\sigma_{\delta ,p}(\rho_1)
\times\sigma_{\delta ,q ,e}(\rho_2)
\subseteq\sigma_{\delta ,k ,e}(\rho)
\subseteq \\
&sp_{\delta ,k,e} (\rho)\subseteq \bigcup_{p+q=k}sp_{\delta ,p,e}
(\rho_1)\times sp_{\delta ,q}(\rho_2)\bigcup\bigcup_{p+q=k}
sp_{\delta ,p}(\rho_1)
\times sp_{\delta ,q,e}(\rho_2),\\\end{align*}

\begin{align*}
 {\rm (ii)}\hbox{}&\bigcup_{p+q=k}\sigma_{\pi ,p,e}(\rho_1)\times 
\sigma_{\pi ,q}(\rho_2 )\bigcup\bigcup_{p+q=k}\sigma_{\pi ,p}(\rho_1)
\times\sigma_{\pi ,q ,e}(\rho_2)
\subseteq\sigma_{\pi ,k ,e}(\rho)\subseteq \\
&sp_{\pi ,k,e} (\rho)\subseteq \bigcup_{p+q=k}sp_{\pi ,p,e}
(\rho_1)\times sp_{\pi ,q}(\rho_2)\bigcup\bigcup_{p+q=k}sp_{\pi ,p}
(\rho_1)
\times sp_{\pi ,q,e}(\rho_2).\\\end{align*}

In particular, if $X_1$ and $X_2$ are Hilbert spaces, the above 
inclusions are equalities.\par
\end{thm}

\begin{proof}

\indent First of all, in the proof of this theorem we use the 
notations and identifications of
Theorem 5.1. In particular, if $\alpha $ is a character of $L_1$
and $\beta$ is a character of $L_2$ we work with the complex
$(X_1\otimes
\wedge L_1,d(\rho_1-\alpha ))\tilde{\otimes}(X_2\otimes\wedge L_2,
d(\rho_2-\beta ))$ instead of the Koszul complex associated to
the representation $\rho-(\alpha ,\beta )\colon L_1\times L_2\to
{\rm L}(X_1\tilde{\otimes }X_2)$. We begin with the first statement.\par

\indent In order to prove the inclusion on the left, the same
argument 
used in Theorem 5.1 for the $\sigma_{\delta ,k}$ joint
spectra may be
applied to the essential $\delta$-S\l odkowski joint spectra. In
fact,
the argument still works when we consider two homology spaces, one of
which is non null and the other is infinite dimensional,
instead of considering two non
null homology
spaces.\par 

\indent As in Theorem 5.1, the middle inclusion is clear.\par

\indent With regard to the inclusion on the right, we shall
prove it by an induction argument.\par

\indent First of all we study the case $k=0$.\par

\indent  We consider a pair $(\alpha ,\beta)\in sp_{\delta ,0,e}
(\rho)\setminus (sp_{\delta ,0,e}(\rho_1)
\times sp_{\delta ,0}(\rho_2)\cup sp_{\delta ,0}(\rho_1)
\times sp_{\delta ,0,e}(\rho_2))$.
Now, since according to Theorem 5.1 $(\alpha ,\beta)\in sp_{\delta ,0}
(\rho_1)
\times sp_{\delta ,0}(\rho_2)$, we have $\alpha\in
sp_{\delta ,0}(\rho_1)
\setminus sp_{\delta ,0 ,e}(\rho_2)$ and $\beta\in
sp_{\delta ,0}(\rho_2)
\setminus sp_{\delta ,0,e}(\rho_2)$. In particular, there are bounded
linear maps
$$
h_0\colon X_1\to X_1\otimes\wedge^1 L_1,\hskip1cm g_0
\colon X_2\to X_2\otimes \wedge^1 L_2,
$$
and finite range projectors
$$
k_0\colon X_1\to X_1,\hskip1cm k^{'}_0\colon X_2\to X_2,
$$
such that
$$
d_1(\rho_1-\alpha)h_0=I_0-k_0,\hskip1cm d_1(\rho_2-\beta)g_0=I_0-k^{'}_0.
$$
\indent Now, if we consider the map
$$
H_0\colon X_1\tilde{\otimes}X_2\to X_1\tilde{\otimes}
X_2\otimes\wedge^1 L_2\oplus X_1\otimes\wedge^1 L_1
\tilde{\otimes} X_2,\hbox{     } H_0=(I_0\otimes g_0
,h_0\otimes I_0),
$$
then it is easy to prove that
$$
d_1H_0=I-k_0\otimes k^{'}_0,
$$
where $d$ and $I$ denote the boundary and the identity of the complex
$(X_1\otimes
\wedge L_1,d(\rho_1-\alpha ))\tilde{\otimes}(X_2\otimes\wedge L_2,
d(\rho_2-\beta ))$ respectively.\par
However, according to Proposition 6.1, the map $k_0\otimes k^{'}_0$ 
is a projector with
finite dimensional range. In particular, according to Proposition 4.4 
$(\alpha ,\beta)$
does not belong to $sp_{\delta ,k,e}(\rho)$, which is impossible
according to our assumption.\par

\indent Now we suppose that the statement on the right is
true for $0$ and for all natural numbers lower than $k$, and for $k$ we
prove the inclusion on the right. We proceed
as in the case $k=0$.\par

\indent We consider a pair 
$(\alpha ,\beta)\in sp_{\delta
 ,k,e}(\rho)$ such that it does not belong to 
$$
(\bigcup_{p+q=k}
 sp_{\delta ,p,e}
(\rho_1)\times sp_{\delta ,q}(\rho_2)\bigcup\bigcup_{p+q=k}
sp_{\delta ,p}(\rho_1)
\times sp_{\delta ,q,e}(\rho_2)).$$ 
In particular,
$$(\alpha ,\beta)\notin (\bigcup_{p+q=k-1}
 sp_{\delta ,p,e}
(\rho_1)\times sp_{\delta ,q}(\rho_2)\bigcup\bigcup_{p+q=k-1}
sp_{\delta ,p}(\rho_1)
\times sp_{\delta ,q,e}(\rho_2)).$$ Thus, by the inductive hypothesis
$(\alpha ,\beta)\notin sp_{\delta ,k-1,e}(\rho)$.\par

\indent In addition, since according to Theorem 5.1,
$$sp_{\delta ,k, e}(\rho)\subseteq sp_{\delta ,k}(\rho)\subseteq \
\bigcup_{p+q=k}sp_{\delta ,p}(\rho_1)\times sp_{\delta ,q}(\rho_2)$$
there are $p_0$ and $q_0$, $p_0+q_0=k$, such that $\alpha
\in sp_{\delta ,p_0}(\rho_1)$ and $\beta\in sp_{\delta ,q_0}(\rho_2)$.
Moreover, we may suppose that $p_0=min\{p,\hbox{} 0\le p\le k
:\alpha\in sp_{\delta ,p}(\rho_1)\}$. In particular,
it is easy to prove that the following assertions are
true:\par

(i) $\alpha\in sp_{\delta ,p_0}(\rho_1)$, $\alpha\notin
sp_{\delta ,p}(\rho_1)$, $p=0,\ldots ,p_0-1$, and
$\beta\notin sp_{\delta ,q_0,e}(\rho_2)$,\par

(ii) $\beta\in sp_{\delta ,q_0}(\rho_2)$, and either $\alpha\notin
sp_{\delta ,k,e}(\rho_1)$ and $\beta\in sp_{\delta ,0}(\rho_2)$,
or there is $p_1$, $p_0\le p_1\le k-1$,
such that $\alpha \notin sp_{\delta ,p_1,e}(\rho_1)$,
$\alpha \in sp_{\delta ,p_1+1,e}(\rho_1)$, and $\beta\notin sp_{\delta ,k-p_1-1}
(\rho_2)$.\par

\indent By means of assertions (i) and (ii), we prove that $\dim 
Ker(d_k)/R(d_{k+1})$ is finite, and that $Ker
(d_{k+1})$ is a complemented subspace. Since $(\alpha ,\beta )
\notin sp_{\delta ,k-1 ,e}(\rho)$, by [13, Theorem 2.7], we have  
$(\alpha ,\beta)\notin sp_{\delta ,k ,e}(\rho )$, which is impossible 
according to our assumption.\par

\indent On the other hand, we work with assertion (i) and 
the second part of
assertion (ii). The other case is similar and in fact easier.
\par

\indent By (i) and (ii) there are
bounded linear operators
$h_p\colon X_1\otimes\wedge^p L_1\to X_1\otimes\wedge^{p+1} L_1$,
$p=0,\ldots ,p_1$, and there are projectors with finite dimensional
range, $k_p\colon X_1\otimes\wedge^p L_1\to X_1\otimes\wedge^p L_1$,
$p=p_0,\ldots , p_1$, such that for $p=0,\ldots , p_0-1$,
$$
h_{p-1}d_p(\rho_1-\alpha)+d_{p+1}(\rho_1-\alpha )h_p=I_p,
$$
and for $p=p_0,\ldots , p_1$,
$$
h_{p-1}d_p(\rho_1-\alpha)+d_{p+1}(\rho-\alpha)h_p=I_p-k_p.
$$
\indent In addition, by (i) and (ii) there are bounded
linear maps $g_q\colon X_2\otimes\wedge^q L_2\to X_2\otimes
\wedge^{q+1} L_2$, $q=0,\ldots , q_0=
k-p_0$, 
and there are projectors with finite 
dimensional range, $k_q^{'}\colon X_2\otimes\wedge^q L_2\to 
X_2\otimes\wedge^q L_2$, $q=k-p_1,\ldots , q_0$, 
such that for $q=0,\ldots ,k-p_1-1$,
$$
g_{q-1}d_q(\rho_2-\beta )+d_{q+1}(\rho_2-\beta)g_q=I_q,
$$
and for $q=k-p_1,\ldots , q_0$,
$$
g_{q-1}d_q(\rho_2-\beta)+d_{q+1}(\rho_2-\beta)g_q=I_q-k^{'}_q.
$$
 
\indent In order to prove that $Ker(d_{k+1})$ is a complemented 
subspace of
$\bigoplus_{p+q=k+1}$\par \noindent $X_1\otimes\wedge^pL_1\tilde{\otimes }X_2\otimes
\wedge^q L_2$, we first characterize it and then show a  
complement.\par

\indent It is easy to prove that $Ker (d_{k+1})$ is the set of all 
$(x_{p,q})$, $p+q=k+1$, $x_{p,q}\in X_1\otimes\wedge^p L_1
\tilde{\otimes }X_2\otimes\wedge^q L_2$ such that in 
$X_1\otimes\wedge^{p-1}L_1\tilde{\otimes }X_2\otimes
\wedge^q L_2$ 
$$
d_p(\rho_1-\alpha)\otimes I_{q}(x_{p,q})+(-1)^{p-1}I_{p-1}\otimes 
d_{q+1}(\rho_2-\beta)(x_{p-1,q+1})=0.
$$
\indent According to Proposition 4.2, we know that for $q=0,\ldots ,k-p_1$
and $p+q=k+1$,
$$
X_1\otimes\wedge^p L_1\tilde{\otimes }X_2\otimes
\wedge^q L_2=R(I_{p}\otimes g_{q-1})\oplus Ker(I_p\otimes 
d_{q}(\rho_2-\beta )),
$$
$$
X_1\otimes\wedge^{p-1}L_1\tilde{\otimes }X_2\otimes
\wedge^q L_2=R(I_{p-1}\otimes g_{q-1})\oplus Ker(I_{p-1}\otimes 
d_{q}(\rho_2-\beta )),
$$
and for $q=0,\ldots,  k-p_1-1$,
$$
X_1\otimes\wedge^{p-1}L_1\tilde{\otimes }X_2\otimes
\wedge^{q+1} L_2=R(I_{p-1}\otimes g_q)\oplus Ker(I_{p-1}\otimes 
d_{q+1}(\rho_2 -\beta)).
$$
\indent In particular, we may present each $x_{p,q}\in X_1\otimes
\wedge^p L_1\tilde{\otimes }X_2\otimes
\wedge^q L_2$, $p+q=k+1$, $q=0,\ldots ,k-p_1$, as $x_{p,q}=(a_{p,q},b_{p,q})$,
where $a_{p,q}\in Ker(I_p\otimes d_{q}(\rho_2-\beta ))$ and $b_{p,q}\in
R(I_{p}\otimes g_{q-1})$. \par

\indent On the other hand, according to Proposition 4.2
$$
I_{p-1}\otimes g_{q}\colon Ker(I_{p-1}\otimes d_{q}(\rho_2-\beta
))\to R(I_{p-1}\otimes g_q)$$
is a topological isomorphism for $q=0,\ldots , k-p_1-1$. Then, an easy calculation shows that
$x_{k+1,0}=a_{k+1,0}$ and that $b_{p,q}=(-1)^{p+1}d_{p+1}(\rho_1-\alpha )\otimes g_{q-1}
(a_{p+1,q-1})$, for $q=1,\ldots , k-p_1$.\par

\indent Thus, $x_{p,q}$ is described for $q=0,\ldots ,k-p_1-1$
and $p$ such that $p+q=k+1$. However, we may continue this procedure
till $q=k-p_0$.\par

\indent In fact, according to Proposition 4.2 the above 
decompositions 
of the spaces $X_1\otimes\wedge^pL_1\tilde{\otimes }X_2\otimes
\wedge^q L_2$, $X_1\otimes\wedge^{p-1}L_1\tilde{\otimes }X_2\otimes
\wedge^q L_2$ and $X_1\otimes\wedge^{p-1}L_1\tilde{\otimes }X_2\otimes
\wedge^{q+1} L_2$ remain true for $q=k-p_1,\ldots , q_0+1=k-p_0+1$.
Moreover, according to Proposition 4.2, it is easy to prove that $Ker(I_{p-1}\otimes 
d_{q}(\rho_2-\beta))=
R(I_{p-1}\otimes k^{'}_q)\oplus R(I_{p-1}\otimes 
d_{q+1}(\rho_2-\beta ))$, $q=k-p_1,\ldots , k-p_0$,
and that
$$
I_{p-1}\otimes g_q\colon R(I_{p-1}\otimes d_{q+1}(\rho_2-\beta 
))\to R(I_{p-1}\otimes g_q)
$$
is a topological isomorphism. Then, if for $q=k-p_1,\ldots , k-p_0+1$
we decompose $x_{p,q}=((a_{p,q}^1,a_{p,q}^2,), b_{p,q})$,
where $a_{p,q}^1\in R(I_p\otimes d_{q+1}(\rho_2-\beta))$, 
$a^2_{p,q}\in
R(I_{p}\otimes k^{'}_q)$ and $b_{p,q}\in
R(I_{p}\otimes g_{q-1})$, an easy calculation shows that
$a_{p,q}^2\in R(I_{p}\otimes k^{'}_q)
\cap Ker(d_{p}(\rho_1-\alpha )\otimes I_q)$, $q=k-p_1,\ldots ,k-p_0$,
and $b_{p,q}=(-1)^{p+1}d_{p+1}(\rho_1-\alpha )\otimes g_{q-1}
(a_{p+1,q-1}^1)$, $q=k-p_1,\ldots , k-p_0+1$.\par

\indent On the other hand, by a similar argument, it is possible to
prove the following fact. 
If we consider for $p=0,\ldots , p_0$ the decomposition 
$$ 
X_1\otimes\wedge^pL_1\tilde{\otimes }X_2\otimes
\wedge^q L_2=R(h_{p-1}\otimes I_q)\oplus Ker(d_{p}
(\rho_1-\alpha )\otimes 
I_q),
$$
and we present $x_{p,q}$ as $x_{p,q}=(c_{p,q},d_{p,q})$,
where $c_{p,q}\in Ker(d_{p}(\rho_1-\alpha )\otimes I_q)$ and $d_{p,q}\in
R(h_{p-1}\otimes I_q)$, then $x_{0,k+1}=c_{0,k+1}$
and $d_{p,q}=(-1)^ph_{p-1}
\otimes d_{q+1}(\rho_2-\beta)(c_{p-1,q+1})$, for $p=1,\ldots , p_0$. \par

\indent Thus, if $(x_{p,q})$, $p+q=k+1$, belongs to $Ker(d_{k+1})$,
$x_{p,q}$ is described for $p=0,\ldots , p_0-1$ and $q=0,\dots ,
k-p_0$. In order to characterize $Ker(d_{k+1})$ in a complete way,  
we have to consider $X_1\otimes\wedge^{p_0} L_1\tilde{\otimes}
X_2\otimes \wedge^{k+1-p_0}L_2$.\par

\indent In $X_1\otimes\wedge^{p_0} L_1\tilde{\otimes}
X_2\otimes \wedge^{k+1-p_0}L_2$, we have two well-defined projectors,
$$
S=I_{p_0}\otimes g_{k-p_0}d_{k-p_0+1}(\rho_2-\beta),\hskip1cm T=
h_{p_0-1}d_{p_0}(\rho_1-\alpha)\otimes I_{k-p_0+1}.
$$
Moreover, since $S$ commutes with $T$, $X_1\otimes\wedge^{p_0} L_1\tilde{\otimes}
X_2\otimes \wedge^{k+1-p_0}L_2$ may be decomposed as the direct
sum of the ranges of the operators $ST$, $S(I-T)$, $(I-S)T$ and
$(I-S)(I-T)$, and each $x$ that belongs to this space may be
decomposed as $x=(x_{ST},x_{S(I-T)}, x_{(I-S)T}, x_{(I-S)(I-T)})$.
\par

\indent Now, if $(x_{p,q})$, $p+q=k+1$, belongs to $Ker (d_{k+1})$, 
in order to determine $x_{p_0,k-p_0+1}$ it is enough to consider the
equations in which it takes part, i.e.,
$$    
d_{p_0+1}(\rho_1-\alpha)\otimes I_{k-p_0}(x_{p_0+1,k-p_0})+(-1)^{p_0}
I_{p_0}\otimes d_{k+1-p_0}(\rho_2-\beta ))(x_{p_0,k+1-p_0})=0,
$$
$$
d_{p_0}(\rho_1-\alpha)\otimes I_{k+1-p_0}(x_{p_0,k+1-p_0})+(-1)^{p_0-1}
I_{p_0-1}\otimes d_{k+2-p_0}(\rho_2-\beta)(x_{p_0-1,k+2-p_0})=0.
$$

\indent In addition, an easy calculation shows that if we present 
$x_{p_0,k-p_0+1}=x$ in the above decomposition, $x_{ST}=0$,
$x_{(I-S)T}=d_{p_0,k+1-p_0}$, $x_{S(I-T)}=b_{p_0,k+1-p_0}$, 
and $x_{(I-S)(I-T)}$ is an arbitrary element in the range of 
$(I-S)(I-T)$.\par

\indent Thus, $Ker (d_{k+1})$ may be presented as the direct sum of
the following spaces. \par
(i) For $q=0,\ldots ,k-p_0$, the graph of the map $(-1)^pd_{p}(\rho_1-\alpha)
\otimes g_q\colon R(I_p\otimes d_{q+1}(\rho_2-\beta ))\to R(I_{p-1}
\otimes g_q)$, $p+q=k+1$.\par
(ii) For $q=k-p_1,\ldots , k-p_0$, $R(I_p\otimes k_q^{'})\cap Ker 
(d_{p}(\rho_1-\alpha)\otimes I_q)$, $p+q=k+1$.\par
(iii) For $p=0,\ldots , p_0-1$, the graph of the map $(-1)^ph_p
\otimes d_{q}(\rho_2-\beta )\colon R(d_{p+1}(\rho_1-\alpha)\otimes I_q)\to R(h_p
\otimes I_{q-1})$, $p+q=k+1$.\par
(iv) The range of the projector $(I-S)(I-T)$.\par
 
\indent In order to construct a direct complement of $Ker 
(d_{k+1})$ we need the following observations.\par

\indent First, if $X$ and $Y$ are Banach spaces and $T\in {\rm L}
(X,Y)$, then $X\oplus Y=Graph(T)\oplus Y$.\par

\indent Second, an easy calculation shows that $R(I_p\otimes k_q^{'})\cap Ker( d_{p}(\rho_1-\alpha)
\otimes I_q)\oplus R(h_{p-1}d_{p}(\rho_1-\alpha)\otimes k^{'}_q)=R(I_p\otimes
k_q^{'})$, for $q=k-p_1,\ldots 
,k-p_0$.\par

\indent Now, depending on $p$ and $q$, $p+q=k+1$, 
the space $X_1\otimes\wedge^p 
L_1\tilde{\otimes} X_2\otimes\wedge^q L_2$ is equal to the direct 
sum of the following spaces.\par
(i) For $p=0,\ldots ,p_0-1$, $R(d_{p+1}(\rho_1-\alpha)\otimes I_q)$ and 
$R(h_{p-1}\otimes I_q)$.\par 
(ii) For $q=0,\ldots ,k-p_0$, $R(I_p\otimes d_{q+1}(\rho_2-\beta))$, 
$R(I_p\otimes g_{q-1})$ and $R(I_p\otimes k_q^{'})$; 
when  $q=0,\ldots , k-p_1-1$, we have 
$k_q^{'}=0$.\par 
(iii) For $p=p_0$ and $q=k-p_0+1$, the ranges of the operators 
$ST$, $S(I-T)$, $(I-S)T$ and $(I-S)(I-T)$.\par

\indent Then, if we consider $V$, the space defined by the direct sum
of the sets, $R(h_{p-1}\otimes I_q)$, $p=0,\ldots ,p_0$,
$R(I_p\otimes g_{q-1})$, $q=0,\ldots ,k-p_0+1$,
$R(h_{p-1}d_{p}(\rho_1-\alpha)\otimes k^{'}_q)$, $q=k-p_1,\ldots 
,k-p_0$, and $R(ST)$ for $p=p_0$ and $q=k-p_0+1$,
we have $\bigoplus_{p+q=k+1}X_1\otimes\wedge^p 
L_1\tilde{\otimes} X_2\otimes\wedge^q L_2=Ker( d_{k+1})\oplus V$.
\par

\indent We now prove that $\dim Ker d_{k}/R(d_{k+1})$ is finite.\par

\indent As with 
$Ker (d_{k+1})$,
we may present $Ker (d_k)$ as the direct sum of the following spaces.\par
(i) For $q=0,\ldots , k-p_0-1$, the graph of $(-1)^pd_{p}(\rho_1-\alpha)
\otimes g_q\colon R(I_p\otimes d_{q+1}(\rho_2-\beta))\to R(I_{p-1}
\otimes g_q)$, $p+q=k$.\par
(ii) For $q=k-p_1,\ldots ,k-p_0-1$, $R(I_p\otimes k_q^{'})\cap 
Ker( d_{p}(\rho_1-\alpha)
\otimes I_q)$, $p+q=k$.\par
(iii) For $p=0,\ldots , p_0-1$, the graph of $(-1)^ph_p
\otimes d_{q}(\rho_2-\beta )\colon R(d_{p+1}(\rho_1-\alpha)\otimes I_q)\to R(h_p
\otimes I_{q-1})$, $p+q=k$.\par
(iv) For $p=p_0$ and $q=k-p_0$, the range of the projector $(I-S)(I-T)$,
where 
$$
S=I_{p_0}\otimes g_{k-p_0-1}d_{k-p_0}(\rho_2-\beta),\hskip1cm T=
h_{p_0-1}d_{p_0}(\rho_1-\alpha)\otimes I_{k-p_0}.
$$
\indent Now we consider $p$ and $q$ such that $p+q=k$ and
$q=0,\ldots , k-p_0-1$. Then, if we consider $(-1)^pI_{p}\otimes g_{q}
(a)$, $a\in R(I_p\otimes d_{q+1}(\rho_2-\beta))$, it is easy to 
prove that
$(a, (-1)^pd_{p}(\rho_1-\alpha)\otimes g_q(a))\in R(d_{k+1})$. Thus, the graph
of $(-1)^pd_{p}(\rho_1-\alpha)
\otimes g_q\colon R(I_p\otimes d_{q+1}(\rho_2-\beta ))\to R(I_{p-1}
\otimes g_q)$ is contained in $R(d_{k+1})$.\par

\indent In a similar way, we may prove that the graph of $(-1)^ph_p
\otimes d_{q}(\rho_2-\beta)\colon R(d_{p+1}(\rho_1$
$-\alpha)\otimes I_q)\to R(h_p
\otimes I_{q-1})$, $p+q=k$, $p=0,\ldots ,p_0-1$, is contained
in $R(d_{k+1})$.\par
 
\indent We denote the following spaces by $S_{p,q}$, $p+q=k$.\par
(i) For $q=k-p_1,\ldots , k-p_0-1$, $S_{p,q}=R(I_p\otimes k_q^{'})
\cap Ker (d_{p}(\rho_1-\alpha)\otimes I_q)$.\par
(ii)  For $p=p_0$ and $q=k-p_0$, $S_{p,q}=R(I-S)(I-T)$.\par

\indent Since $k-p_1\le q\le k-p_0$ and $p_0\le p\le p_1$,
we may consider the well-defined map
$$
H_{p,q}\colon X_1\otimes \wedge^p L_1\tilde{\otimes}
X_2\otimes \wedge^qL_2\to X_1\otimes \wedge^{p+1}L_1\tilde{\otimes}
X_2\otimes \wedge^q L_2\oplus X_1\otimes \wedge^p L_1
\tilde{\otimes}X_2\otimes \wedge^{q+1} L_2,
$$
$$
H_{p,q}=h_p\otimes I_q+k_p\otimes g_q.
$$
\indent Moreover, if we define $k_{p_0-1}=0$ and $k^{'}_{k-p_1-1}
=0$, then we may define the corresponding maps $H_{p-1,q}$ and
$H_{p,q-1}$, and an easy calculation shows that
$$
(H_{p-1,q}\oplus H_{p,q-1})d_k +d_{k+1}H_{p,q}=I-k_p\otimes 
k^{'}_q.
$$
However, since $S_{p,q}$ is contained in $Ker (d_k)$, 
$$
d_{k+1}( H_{p,q}(S_{p,q}) )+ k_p\otimes k{'}_q(S_{p,q})=
S_{p,q}.
$$
Thus, according to Proposition 6.1, the codimension of 
$R(d_{k+1})$ in $Ker( d_k)$  is finite.\par 

\indent The second statement of the theorem may be proved by a 
similar argument, using the second part of
Proposition 4.2.\par 
                              
\end{proof}

\indent As in the last section, we consider two nilpotent systems
of operators and prove a variant of Theorem 6.2 for this case.
In particular, in the commuting case we obtain
an extension of [14, Theorem 3.2].
 \par

\begin{thm}Let $X_1$ and $X_2$ be two complex Banach 
spaces. We suppose 
that there is a tensor product of $X_1$ and $X_2$ with respect 
to $\langle X_1,X_1{'}\rangle$ and $\langle X_2,X_2{'}\rangle$, $X_1\tilde{\otimes} 
X_2$. Let $a=(a_1,\ldots ,a_n)$ and $b=(b_1,\ldots ,b_m)$ be two 
tuples of operators, $a_i\in {\rm L}(X_1)$, $1\le i\le n$, and 
$b_j\in {\rm L}(X_2)$, $1\le j\le m$, such that the vector 
subspaces generated by them, $\langle a_i\rangle_{1\le i\le n}$ and
$\langle b_j\rangle_{1\le j\le m}$, are nilpotent Lie subalgebras
of ${\rm L}(X_1)$ and ${\rm L}(X_2)$ respectively.
We consider the $(n+m)$-tuple of  
operators defined in $X_1\tilde{\otimes}X_2$,
$c=(a_1\otimes I ,\ldots ,a_n\otimes I ,I\otimes b_1,\ldots ,
I\otimes b_m)$,
where $I$ denotes the identity of $X_2$ and $X_1$ respectively. Then
we have 

\begin{align*}
{\rm (i)}& \hbox{}\bigcup_{p+q=k}\sigma_{\delta ,p,e}(a)\times 
\sigma_{\delta ,q}(b)\bigcup\bigcup_{p+q=k} \sigma_{\delta ,p}(a)\times
\sigma_{\delta ,q,e}(b)\subseteq\sigma_{\delta ,k,e}(c)
\subseteq\\
&sp_{\delta ,k,e} (c)\subseteq \bigcup_{p+q=k}sp_{\delta ,p,e}
(a)\times sp_{\delta ,q}(b)\bigcup\bigcup_{p+q=k}sp_{\delta ,p}(a)
\times sp_{\delta ,q,e}(b),\\\end{align*}

\begin{align*}
 {\rm (ii)}& \hbox{}\bigcup_{p+q=k}\sigma_{\pi ,p,e}(a)\times 
\sigma_{\pi ,q}(b)\bigcup\bigcup_{p+q=k} \sigma_{\pi ,p}(a)\times
\sigma_{\pi ,q,e}(b)\subseteq\sigma_{\pi ,k,e}(c)
\subseteq \\
&sp_{\pi ,k,e} (c)\subseteq \bigcup_{p+q=k}sp_{\pi ,p,e}
(a)\times sp_{\pi ,q}(b)\bigcup\bigcup_{p+q=k}sp_{\pi ,p}(a)
\times sp_{\pi ,q,e}(b).\\\end{align*}
\end{thm}
 
\begin{proof}
\indent Adapt the argument in Theorem 5.2.\par 

\end{proof}
 
\section{Joint spectra of the multiplication representation}

\indent In this section we deal with an operator ideal in the sense of
J. Eschmaier, see [14] or below. These operator ideals are 
naturally a tensor product of two Banach spaces, and since
the multiplication representation may be seen as a
tensor product representation, we shall extend the 
results in sections 5 and 6 to the multiplication representation. 
We begin with the definition
of an operator ideal in the sense of J. Eschmeier.\par
  
\begin{df} An operator ideal $J$ between Banach spaces $X_2$ 
and $X_1$ is a linear subspace of ${\rm L}(X_2,X_1)$ equipped 
with a 
space norm $\alpha$ such that\par
{\rm (i)}  $x_1\otimes x_2{'}\in J$ and $\alpha (x_1\otimes x_2{'})=
\parallel x_1
\parallel \parallel x_2{'}\parallel$,\par
 {\rm (ii)} $SAT\in J$ and $\alpha (SAT)\le \parallel S\parallel \alpha (A) 
\parallel T\parallel$, \par

for $x_1\in X_1$, $x_2{'}\in X_2{'}$, $A\in J$, 
$S\in {\rm L}
(X_1)$,
$T\in{\rm L}(X_2)$, and $x_1\otimes x_2{'}$ is the usual rank one
operator $X_2\to X_1$, $x_2\mapsto <x_2,x_2{'}>x_1$.\par
\end{df}

\indent Examples of this kind of ideals are given in [14, Section 1].\par

\indent We recall that such an operator ideal $J$ is naturally a tensor
product relative to $\langle X_1,X_1{'}\rangle$ and $\langle X_2{'},
X_2\rangle$, with the 
bilinear mappings
$$
X_1\times X_2{'}\to J,\hbox{  } (x_1,x_2{'})\mapsto x_1\otimes x_2{'},
$$
$$
{\mathcal L}(X_1)\times{\mathcal L}(X_2{'})\to {\rm L}(J),
\hbox{  } (S,T{'})\mapsto S\otimes T{'},
$$
where $S\otimes T{'} (A)=SAT$.\par 

\indent Now, let $L_1$  and $L_2$ be two complex solvable
finite dimensional Lie algebras, $X_1$ and $X_2$ two complex
Banach spaces, and $\rho_i\colon 
L_i\to {\rm L}(X_i)$, $i=1$, $2$, 
two representations of Lie algebras. 
We consider the Lie algebra 
$L_2^{op}$ and the adjoint representation  
$\rho_2^*\colon L_2^{op}\to {\rm L}(X_2{'})$. Now, if $L$ is the direct sum of $L_1$ and $L_2^{op}$,
$L=L_1\times L_2^{op}$, then the multiplication representation of $L$ 
in $J$ considered in Chpater 3, Section 3.6 of [21] is
$$
\tilde{\rho}\colon L\to {\rm L}(J),\hbox{  }\tilde{\rho}(l_1,l_2)(T)=\rho_1 (l_1)T+
T\rho_2 (l_2).  
$$
According to [21, Korollar 3.6.10] $\tilde{\rho}$ is a representation of $L$ in 
${\rm L} (J)$, and when $J$ is viewed as a tensor product of 
$X_1$ and 
$X_2{'}$ relative to $\langle X_1,X_1{'}\rangle$ and $\langle X_2{'},X_2\rangle$, 
$\tilde{\rho}$
coincides with the representation
$$
\rho_1\times \rho_2^*\colon L\to {\rm L}(X_1\tilde{\otimes}X_2{'}),
\hbox{}  \rho_1\times \rho_2^*(l_1,l_2)=\rho_1(l_1)\otimes I
+I\otimes \rho_2^*(l_2).
$$
\indent Moreover, by a similar argument to the one in Proposition 4.4, 
it is easy to prove that the complex 
$(X_1\otimes\wedge L_1,d(\rho_1))\tilde{\otimes}
(X_2\otimes\wedge L_2^{op},d(\rho_2^*))$ is well defined, and 
that it is isomorphic
to the complex $((X_1\tilde{\otimes}X_2{'})\otimes\wedge L,d(\rho_1\times
\rho_2^*))$, which may be
identified 
with the complex $(J\otimes\wedge L, d(\tilde{\rho}))$, when $J$ is viewed 
as a tensor product of $X_1$ and   
$X_2{'}$ relative to $\langle X_1,X_1{'}\rangle$ and 
$\langle X_2{'},X_2\rangle$.
\par

\indent In the following theorems we describe
the joint spectra of the representation $\tilde{\rho}$.\par

\begin{thm}  Let $L_1$ and $L_2$ be two complex solvable 
finite dimensional Lie algebras, $X_1$ and $X_2$ two complex Banach 
spaces, and $\rho_i \colon L_i\to {\rm L}(X_i)$, $i=1$, $2$, 
two representations of Lie algebras. 
We suppose that there is an operator ideal $J$ between $X_2$
and $X_1$ in the sense of 
Definition 7.1, and we present it as the 
tensor product of $X_1$ and $X_2{'}$ relative  
to $\langle X_1,X_1{'}\rangle$ and $\langle X_2{'},X_2\rangle$. Then, if we consider the 
multiplication representation $\tilde{\rho}\colon L_1\times L_2^{op}\to {\rm L}(J)$,
we have 

\begin{align*}
{\rm (i)} &\hbox{}\bigcup_{p+q=k}\sigma_{\delta , p}(\rho_1)\times 
(\sigma_{\pi , m-q}(\rho_2)-h_2)\subseteq\sigma_{\delta  ,k}
(\tilde{\rho})
\subseteq\\
&sp_{\delta , k} (\tilde{\rho})\subseteq
\bigcup_{p+q=k}sp_{\delta , p}
(\rho_1)\times (sp_{\pi , m-q}(\rho_2)-h_2),\\\end{align*}

\begin{align*}
 {\rm (ii)}&\hbox{}\bigcup_{p+q=k}\sigma_{\pi ,p}(\rho_1)\times 
(\sigma_{\delta , m-q}(\rho_2)-h_2)\subseteq
\sigma_{\pi ,k}(\tilde{\rho})
\subseteq\\ 
& sp_{\pi , k} (\tilde{\rho})\subseteq
\bigcup_{p+q=k}sp_{\pi , p}
(\rho_1)\times (sp_{\delta ,m-q}(\rho_2)-h_2),\\\end{align*}

where $h_2$ is the character of $L_2$ considered in Theorem 3.4.\par
\indent In particular, if $X_1$ and $X_2$ are Hilbert spaces, the above 
inclusions are equalities.\par
\end{thm}

\begin{proof}

\indent We begin with the first statement.\par

\indent We consider the complexes $(X_1\otimes \wedge L_1,
d(\rho_1))$ and $(X_2{'}\otimes \wedge L_2^{op},
d(\rho_2^*))$. Since the complex 
$(J\otimes \wedge L,d(\tilde{\rho}))$ is isomorphic to 
$(X_1\otimes \wedge L_1,
d(\rho_1))\tilde{\otimes}(X_2{'}\otimes \wedge L_2^{op},
d(\rho_2^*))$, we work with the latter.\par

\indent In addition, if we consider the differentiable spaces associated
to the Koszul complexes 
defined by the representations $\rho_1$ and $\rho_2^*$,
$({\mathcal X}_1,\partial_1)$ and $({\mathcal X}_2{'},\partial^*_2)$
respectively, since $\partial_1\in{\mathcal L}({\mathcal X}_1)$
and $\partial_2^*\in{\mathcal L}({\mathcal X }_2{'})$, we may
consider the tensor product of $({\mathcal X}_1,\partial_1)$ 
and $({\mathcal X}_2{'},\partial^*_2)$ relative to $\langle{\mathcal X}_1,
{\mathcal X}_1{'}\rangle$ and $\langle{\mathcal X}_2{'},{\mathcal X}_2\rangle$, 
$({\mathcal X}_1,\partial_1)\tilde{\otimes}({\mathcal X}{'}_2,
\partial_2^*)$, which has 
the boundary $\tilde{\partial}=
\partial_1\otimes I+\eta\otimes\partial_2^*$; see [14] or
section 4. However,  
$({\mathcal X}_1,\partial_1)\tilde{\otimes}({\mathcal X}{'}_2,
\partial_2^*)$ is the differentiable space associated to
the complex $(X_1\otimes \wedge L_1,
d(\rho_1))\tilde{\otimes}(X_2{'}\otimes \wedge L_2^{op},
d(\rho_2^*))$; see section 4 or [14].\par
   
\indent Now we consider $\alpha\in \sigma_{\delta , p}
(\rho_1)$ and $\beta\in \sigma_{\pi , m-q}(\rho_2)-h_2$,
$p+q=k$. Then, by the duality property of the S\l odkowski
joint spectra, [5, Theorem 7] and [21, Lemma 2.11.4],
$\beta\in\sigma_{\delta ,q}(\rho_2^*)$.
Now, if we consider the Koszul complexes associated to the
representations $\rho_1-\alpha\colon L_1\to {\rm L}(X_1)$
and $\rho_2^*-\beta\colon L_2^{op}\to {\rm L}(X_2{'})$,
the differentiable spaces associated to them, $({\mathcal X}_1,
\partial_1)$ and $({\mathcal X}_2,
\partial_2^*)$ respectively, and the tensor product $({\mathcal X}_1,\partial_1)\tilde{\otimes}({\mathcal X}{'}_2,
\partial_2^*)$, then we may apply [14, Theorem 2.2], and a similar argument 
to the one in Theorem 5.1 shows the inclusion on the left.\par  

 \indent The middle inclusion is clear.\par

\indent With regard to the inclusion on the right, we  
adapt the corresponding argument in Theorem 5.1 to the present situation.\par 

\indent We consider the complex $(X_2\otimes \wedge 
L_2, d(\rho_2))$. By the duality property of the Koszul complex 
associated to $\rho_2$ (see
[5, Theorem  1] and [21, Korollar  2.4.5]), if $\beta\notin 
(sp_{\pi , m-q}(\rho_2)-h_2)$, then $\beta\notin sp_{\delta ,q}
(\rho_2^*)$. In particular, if $(\alpha ,\beta)\notin\bigcup_{p+q=k}
sp_{\delta ,p}(\rho_1)\times (sp_{\pi , m-q}(\rho_2)-h_2)$,
then $(\alpha ,\beta)\notin \bigcup_{p+q=k}sp_{\delta ,p}
(\rho_1)\times sp_{\delta ,q}(\rho_2^*)$.\par

\indent In addition, by the duality property of the Koszul complex
of the representation $\rho_2$ and by elementary properties
of the adjoint of an operator, it is easy to prove that  if $\beta\notin
(sp_{\pi , m-t}(\rho_2)-h_2)$, then there is a homotopy 
for the complex $(X_2{'}\otimes L_2^{op},d(\rho_2^*-\beta))$,
$(g_{s})_{0\le s\le t}$, which satisfies the
preliminaries facts recalled before Proposition 4.2. Besides, if
for each $s=0,\ldots , t$ we think each map $g_s$ as a matrix of operators,
then each component of this matrix is an adjoint operator. 
\par

\indent Now, according to the properties of the axiomatic tensor
product introduced in [14], if there 
is a tensor product of a Banach space $Y$ and $X{'}$
relative to $\langle Y,Y{'}\rangle$ 
and $\langle X{'}, X\rangle$, $Y\tilde{\otimes} X{'}$,
then it is possible to prove similar results
to the ones in Proposition 4.2.
In particular, it is possible to adapt the proof in Theorem 5.1 to
the present case in order to prove the inclusion on the right.\par
 
\indent The second statement may be proved by a
similar argument.\par

\end{proof}

\indent Now we describe the essential S\l odkowski and the essential
split joint spectra of the multiplication representation $\tilde{\rho}$.\par

\begin{thm}  Let $L_1$ and $L_2$ be two complex solvable 
finite dimensional Lie algebras, $X_1$ and $X_2$ two complex Banach 
spaces, and $\rho_i \colon L_i\to {\rm L}(X)$, $i=1$, $2$, 
two representations of Lie algebras. 
We suppose that there is an operator ideal $J$ between $X_2$
and $X_1$ in the sense of 
Definition 7.1, and we present it as the 
tensor product of $X_1$ and $X_2{'}$ relative to 
to $\langle X_1,X_1{'}\rangle$ and $\langle X_2{'},X_2\rangle$. 
Then, if we consider the multiplication
representation $\tilde{\rho}\colon L_1\times L_2^{op}\to {\rm L}(J)$,
we have 

\begin{align*}
{\rm (i)}&\hbox{}\bigcup_{p+q=k}\sigma_{\delta , p,e}(\rho_1)\times 
(\sigma_{\pi ,m-q}(\rho_2)-h_2)\bigcup\bigcup_{p+q=k}\sigma_{\delta , p}
(\rho_1)\times (\sigma_{\pi , m-q ,e}(\rho_2)-h_2)\\
&\subseteq\sigma_{\delta ,k ,e}(\tilde{\rho})\subseteq  sp_{\delta , k,e} (\tilde{\rho})\subseteq \\
&\bigcup_{p+q=k}sp_{\delta ,p ,e}
(\rho_1)\times (sp_{\pi , m-q}(\rho_2)-h_2)\bigcup\bigcup_{p+q=k}
sp_{\delta , p}(\rho_1)\times (sp_{\pi , m-q ,e}(\rho_2)-h_2),\\\end{align*}

\begin{align*}
 {\rm (ii)} &\hbox{}\bigcup_{p+q=k}\sigma_{\pi , p, e}(\rho_1)\times 
(\sigma_{\delta ,m-q}(\rho_2)-h_2)\bigcup\bigcup_{p+q=k}
\sigma_{\pi ,p}(\rho_1)\times (\sigma_{\delta ,m-q ,e}(\rho_2)-h_2)\\
&\subseteq\sigma_{\pi , k ,e}(\tilde{\rho})
\subseteq sp_{\delta , k ,e} (\tilde{\rho})\subseteq\\
& \bigcup_{p+q=k}sp_{\pi ,p ,e}
(\rho_1)\times (sp_{\delta , m-q}(\rho_2)-h_2)\bigcup
\bigcup_{p+q=k}sp_{\pi , p}(\rho_1)\times (sp_{\delta ,m-q ,e}(\rho_2)
-h_2),\\\end{align*}

where $h_2$ is the character of $L_2$ considered in Theorem 3.4.\par
\indent In particular, if $X_1$ and $X_2$ are Hilbert spaces, the above 
inclusions are equalities.\par
\end{thm}

\begin{proof}
 
\indent Adapt the 
proof of Theorem 6.2.\par
 
\end{proof}

\indent As in section 5 and 6 we consider nilpotent systems of 
operators, and we obtain variants of Theorems 7.2 and 7.3
for this case. \par

\begin{thm} Let $X_1$ and $X_2$ be two complex Banach
spaces, and 
$a=(a_1,\ldots ,a_n)$ and $b=(b_1,\ldots ,b_m)$ two 
tuples of operators, $a_i\in {\rm L}(X_1)$, $1\le i\le n$, and 
$b_j\in {\rm L}(X_2)$, $1\le j\le m$, such that the vector 
subspace generated by them, $\langle a_i\rangle_{1\le i\le n}$ and
$\langle b_j\rangle_{1\le j\le m}$, are nilpotent Lie subalgebras
of ${\rm L}(X_1)$ and ${\rm L}(X_2)$ respectively.
We consider $J\subseteq {\rm L}(X_2,X_1)$ an
operator ideal between $X_2$ and $X_1$ in the sense of Definition 7.1, 
and the $(n+m)$-tuple of operators
defined in ${\rm L}(J)$, $c=(L_{a_1},
\ldots , L_{a_n},R_{b_1},\ldots ,R_{b_m}) $,
where if $S\in {\rm L }(X_1)$ and if $T\in{\rm L}(X_2)$,
the maps $L_S$, $R_T \colon J\to J$ are defined by
$$
L_S(U)=SU,\hskip1cm R_T (U)= UT.
$$
Then, we have
$$
{\rm (i)} \hbox{}\bigcup_{p+q=k}\sigma_{\delta , p}(a)\times 
\sigma_{\pi , m-q}(b)\subseteq\sigma_{\delta ,k}(c)
\subseteq sp_{\delta , k} (c)\subseteq \bigcup_{p+q=k}sp_{\delta ,p}
(a)\times sp_{\pi ,m-q}(b),
$$
$$
 {\rm (ii)}\hbox{}\bigcup_{p+q=k}\sigma_{\pi ,p}(a)\times 
\sigma_{\delta , m-q}(b)\subseteq
\sigma_{\pi , k}(c)
\subseteq sp_{\pi , k} (c)\subseteq \bigcup_{p+q=k}sp_{\pi , p}
(a)\times sp_{\delta , m-q}(b).
$$
\end{thm}

\begin{proof}

\indent As in Theorem 5.2, we consider the Lie algebras
$L_1=\langle a_i\rangle_{1\le i\le n}$ and $L_2=\langle b_j\rangle_{1\le j\le m}$,
the representations of the above algebras defined by the 
inclusion, i.e., $\iota_1\colon L_1\to {\rm L}(X_1)$
and $\iota_2\colon L_2\to {\rm L}(X_2)$, and
the representation $\iota =\iota_1\times\iota_2^*
\colon L_1\times L_2^{op}\to {\rm L}(X_1\tilde{\otimes}
X_2{'})$. Then, if $J$ is viewed as a tensor product of 
$X_1$ and $X_2$ relative to $\langle X_1,X_1{'}\rangle$ and 
$\langle X_2{'},X_2\rangle$,
$\iota$ coincides with the representation $\rho\colon L_1\times 
L_2^{op}\to {\rm L}(J)$, $\rho (A,B)(T)=AT+TB$.\par

\indent Now, the argument in Theorem 5.2 may be adapted to the 
present situation using Proposition 3.11 and Theorem
7.2 instead of Theorem 5.1.\par

\end{proof} 
\begin{thm} In the conditions of Theorem 7.4

\begin{align*}
{\rm (i)} &\hbox{}\bigcup_{p+q=k}\sigma_{\delta , p, e}(a)\times 
\sigma_{\pi ,m-q}(b)\bigcup\bigcup_{p+q=k}\sigma_{\delta ,p}(a)
\times\sigma_{\pi ,m-q , e}(b)
\subseteq\sigma_{\delta ,k ,e}(c)
\subseteq \\
& sp_{\delta , k, e} (c)\subseteq \bigcup_{p+q=k}sp_{\delta ,p ,e}
(a)\times sp_{\pi , m-q}(b)\bigcup\bigcup_{p+q=k}sp_{\delta , p}(a)
\times sp_{\pi ,m-q ,e}(b),\\\end{align*}

\begin{align*}
 {\rm (ii)}&\hbox{}\bigcup_{p+q=k}\sigma_{\pi ,p , e}(a)\times 
\sigma_{\delta ,m-q}(b)\bigcup \bigcup_{p+q=k}
\sigma_{\pi ,p}(a)\times \sigma_{\delta ,m-q ,e}(b)\subseteq
\sigma_{\pi ,k ,e}(c)\subseteq \\
& sp_{\pi ,k ,e} (c)\subseteq \bigcup_{p+q=k}sp_{\pi ,p ,e}
(a)\times sp_{\delta ,m-q}(b)\bigcup\bigcup_{p+q=k}sp_{\pi ,p}(a)
\times sp_{\delta ,m-q ,e}(b).\\\end{align*}

\end{thm}

\begin{proof}

\indent Adapt the argument 
in Theorem 6.3, using Proposition 3.11 and Theorem 7.3 instead
of Theorem 6.2.\par

\end{proof}

\indent We observe that similar remarks to the ones in section 5 and
6 may be made for the theorems in this section. In particular,
Theorem 7.2 and 7.4 are extensions of [21, Korollar  3.6.10] and [21, Satz 3.7.4] 
respectively, for the tensor product
introduced in [14].
In addition, Theorem 7.3 and 7.5 extend [14, Theorem 3.1] and [14, Theorem 3.2]
respectively for the essential joint spectra.\par

\section{Applications}

\indent In this section we apply the results that we obtained in sections 5, 6, and 7 to
particular representations of nilpotent Lie algebras.  
\par

\indent We consider two complex Banach spaces $X_1$ and $X_2$,
a complex nilpotent finite dimensional Lie algebra $L$,
and two representations of $L$, $\rho_1\colon L\to {\rm L}(X_1)$
and $\rho_2\colon L\to{\rm L}(X_2)$. We suppose that 
there is a tensor product of $X_1$ and $X_2$ relative
to $\langle X_1, X_1{'}\rangle$ and $\langle X_2,X_2{'}\rangle$, 
$X_1\tilde{\otimes}X_2$. Thus, we may
consider the tensor product representation 
$$
\rho=\rho_1\times\rho_2\colon L\times L\to {\rm L}(X_1\tilde{\otimes}X_2),\hskip1cm
\rho=\rho_1\otimes I+I\otimes \rho_2.
$$

\indent Now we consider the diagonal map
$$
\Delta \colon L\to L\times L,\hskip2cm \Delta (l)=(l,l),
$$  
and we identify $L$ with $\Delta (L)$. In addition, 
we may consider the representation
$$
\theta=\rho\circ\Delta\colon L\to{\rm L}(X_1\tilde{\otimes}X_2),\hskip1cm \theta (l)=\rho_1 (l)\otimes I
+I\otimes \rho_2 (l). 
$$

\indent In the following theorem we describe the S\l odkowski, the split, 
the essential S\l odkowski and the essential split joint spectra of the representation $\theta$.\par

\begin{thm}  Let $L$ be a complex nilpotent 
finite dimensional Lie algebras, $X_1$ and $X_2$ two complex Banach 
spaces, and $\rho_i \colon L\to {\rm L}(X_i)$, $i=1$, $2$,
two representations of the Lie algebra $L$. 
We suppose that there is a tensor product of $X_1$ and $X_2$ relative to 
$\langle X_1,X_1{'}\rangle$ and $\langle X_2,X_2{'}\rangle$, $X_1\tilde{\otimes} 
X_2$. Then, if we consider the representation $\theta\colon L
\to {\rm L}(X_1\tilde{\otimes} X_2)$,
we have 
$$
{\rm (i)} \hbox{}\bigcup_{p+q=k}(\sigma_{\delta , p}(\rho_1)+ 
\sigma_{\delta ,q}(\rho_2))\subseteq\sigma_{\delta ,k}(\theta)
\subseteq sp_{\delta ,k} (\theta)\subseteq \bigcup_{p+q=k}(sp_{\delta ,p}
(\rho_1)+ sp_{\delta ,q}(\rho_2)),
$$
$$
 {\rm (ii)}\hbox{}\bigcup_{p+q=k}(\sigma_{\pi, p}(\rho_1)+ 
\sigma_{\pi ,q}(\rho_2))\subseteq\sigma_{\pi , k}(\theta)
\subseteq sp_{\pi ,k} (\theta)\subseteq \bigcup_{p+q=k}(sp_{\pi , p}
(\rho_1)+ sp_{\pi ,q}(\rho_2)),
$$

\begin{align*}
 {\rm (iii)} \hbox{}&\bigcup_{p+q=k}(\sigma_{\delta , p ,e}(\rho_1)+ 
\sigma_{\delta ,q}(\rho_2 ))\bigcup\bigcup_{p+q=k}(\sigma_{\delta ,p}(\rho_1)+
\sigma_{\delta , q ,e}(\rho_2))
\subseteq\sigma_{\delta ,k ,e}(\theta)
\subseteq \\
&sp_{\delta ,k ,e} (\theta)\subseteq \bigcup_{p+q=k}(sp_{\delta ,p, e}
(\rho_1)+ sp_{\delta ,q}(\rho_2))\bigcup\bigcup_{p+q=k}
(sp_{\delta ,p}(\rho_1)
+ sp_{\delta ,q , e}(\rho_2)),\\\end{align*}

\begin{align*}
 {\rm (iv)} \hbox{}&\bigcup_{p+q=k}(\sigma_{\pi ,p ,e}(\rho_1)+ 
\sigma_{\pi ,q}(\rho_2 ))\bigcup\bigcup_{p+q=k}(\sigma_{\pi ,p}(\rho_1)
+\sigma_{\pi ,q ,e}(\rho_2))
\subseteq\sigma_{\pi ,k ,e}(\theta)\subseteq \\
&sp_{\pi ,k , e} (\theta)\subseteq \bigcup_{p+q=k}(sp_{\pi ,p ,e}
(\rho_1) + sp_{\pi ,q}(\rho_2))\bigcup\bigcup_{p+q=k}(sp_{\pi ,p}
(\rho_1)+ sp_{\pi ,q ,e}(\rho_2)).\\\end{align*}

In particular, if $X_1$ and $X_2$ are Hilbert spaces, the above 
inclusions are equalities.\par
\end{thm}
\begin{proof}

\indent In order to prove the first statement we recall that according to Theorem 5.1 we have
$$
\bigcup_{p+q=k}\sigma_{\delta ,p}(\rho_1)\times 
\sigma_{\delta ,q}(\rho_2 )
\subseteq\sigma_{\delta ,k }(\rho)
\subseteq 
sp_{\delta ,k} (\rho)\subseteq \bigcup_{p+q=k}sp_{\delta , p}
(\rho_1)\times sp_{\delta ,q}(\rho_2).
$$
\indent Now, the map $\Delta\colon L\to L\times L$ is an identification
between $L$ and $\Delta (L)$, which is a subalgebra of the nilpotent Lie
algebra $L\times L$. Then, if we consider the representation
$\rho\mid\Delta (L)\colon \Delta (L)\to {\rm L}(X_1\tilde
{\otimes}X_2)$, since $\theta=\rho\mid\Delta (L)\circ 
\Delta$, according to Proposition 3.11 we have  
$$
\sigma_{\delta , k}(\theta)=\sigma_{\delta , k}(\rho\mid \Delta(L))
\circ \Delta=\{\alpha\circ \Delta:\hbox{ }\alpha\in \sigma_{\delta , k}
(\rho\mid \Delta(L))\},
$$
and 
$$
sp_{\delta , k}(\theta)=sp_{\delta , k}(\rho\mid \Delta(L))
\circ \Delta=\{\alpha\circ \Delta:\hbox{ }\alpha\in sp_{\delta , k}
(\rho\mid \Delta(L))\}.
$$
\indent In addition, since $\Delta (L)$
is a subalgebra of the nilpotent Lie algebra $L\times L$, by
the projection property for the S\l odkowski and the split joint
spectra, [21, Satz 2.11.5], [21, Satz 3.1.5] and Theorem 2.4, we have 
$$
\pi (\sigma_{\delta ,k}(\rho))=\sigma_{\delta ,k}(\rho\mid \Delta (L)),
\hskip1cm \pi (sp_{\delta ,k}(\rho))=sp_{\delta ,k}(\rho\mid \Delta (L)),
$$
where $\pi\colon (L\times L)^*\to \Delta (L)^*$ denotes the
restrictiton map.\par 

\indent However, it is easy to prove that 
$$
\pi(\bigcup_{p+q=k}\sigma_{\delta ,p}(\rho_1)\times 
\sigma_{\delta ,q}(\rho_2 ))\circ \Delta=
\bigcup_{p+q=k}(\sigma_{\delta ,p}(\rho_1)+ \sigma_{\delta , q}(\rho_2 )),
$$
and that
$$
\pi (\bigcup_{p+q=k}sp_{\delta , p}(\rho_1)\times sp_{\delta ,q}(\rho_2))\circ \Delta=
\bigcup_{p+q=k}(sp_{\delta ,p}(\rho_1)+ sp_{\delta , q}(\rho_2)).
$$

\indent Thus, we proved the first part of the theorem.\par

\indent The other statements may be proved by similar 
arguments, using for (ii) Theorem 5.1 and the projection property for the
S\l odkowki and the split joint spectra, and for (iii) and (iv) Theorem 6.2
and the projection property for the essential S\l odkowski and the
essential split joint spectra, Theorems
3.2, 3.5 and 3.10.\par

\end{proof}

\indent Now we consider two complex Banach spaces $X_1$ and $X_2$, an
operator ideal between $X_2$ and $X_1$ in the sense of [14],
a complex nilpotent Lie algebra $L$, two representations of $L$, $\rho_1\colon L\to{\rm L}(X_1)$
and  $\rho_2\colon L\to{\rm L}(X_2)$, and the 
representation of $L^{op}$,
$\nu =-\rho_2\colon L^{op}\to {\rm L}(X_2)$. 
As in section 7, we may consider the 
multiplication representation
$$
\tilde{\rho}\colon L\times L\to{\rm L}(J),\hskip 1cm 
\tilde{\rho} (l_1,l_2)(T)= \rho_1(l_1) T-T\rho_2(l_2).
$$ 
\indent As above, we may consider the representation 
$$
\tilde{\theta}=\tilde{\rho}\circ \Delta\colon L\to{\rm L}(J).
$$

\indent In the following theorem we describe the S\l odkowski, 
the split, the essential S\l odkowski
and the essential split joint spectra of the representation 
$\tilde{\theta}\colon L\to{\rm L}(J)$.\par

\begin{thm}  Let $L$ be a complex nilpotent 
finite dimensional Lie algebra, $X_1$ and $X_2$ two complex Banach 
spaces, and $\rho_i \colon L\to {\rm L}(X_i)$, $i=1$, $2$,
two representations of the Lie algebra $L$. 
We suppose that there is an operator ideal $J$ between $X_2$
and $X_1$ in the sense of 
Definition 7.1. Then, if we consider the 
representation $\tilde{\theta}\colon L\to {\rm L}(J)$,
we have 

\begin{align*}
{\rm (i)} &\hbox{}\bigcup_{p+q=k}(\sigma_{\delta , p}(\rho_1)- 
\sigma_{\pi ,m-q}(\rho_2)+h_2)\subseteq\sigma_{\delta ,k}
(\tilde{\theta})
\subseteq sp_{\delta ,k} (\tilde{\theta})\subseteq\\
&\bigcup_{p+q=k}(sp_{\delta , p}
(\rho_1)- sp_{\pi ,m-q}(\rho_2)+h_2),\\\end{align*}

\begin{align*}
 {\rm (ii)}&\hbox{}\bigcup_{p+q=k}(\sigma_{\pi ,p}(\rho_1)- 
\sigma_{\delta ,m-q}(\rho_2)+h_2)\subseteq
\sigma_{\pi ,k}(\tilde{\theta})
\subseteq sp_{\pi ,k} (\tilde{\theta})\subseteq\\
& \bigcup_{p+q=k}(sp_{\pi ,p}
(\rho_1)- sp_{\delta ,m-q}(\rho_2)+h_2),\\\end{align*}

\begin{align*}
 {\rm (iii)} &\hbox{}\bigcup_{p+q=k}(\sigma_{\delta, p ,e}(\rho_1)- 
\sigma_{\pi ,m-q}(\rho_2)+h_2)\bigcup\bigcup_{p+q=k}(\sigma_{\delta ,p}
(\rho_1)-\sigma_{\pi ,m-q ,e}(\rho_2)+h_2)\\
&\subseteq\sigma_{\delta ,k ,e}(\tilde{\theta})\subseteq  
sp_{\delta ,k ,e} (\tilde{\theta})\subseteq \\
&\bigcup_{p+q=k}(sp_{\delta ,p ,e}
(\rho_1)-sp_{\pi ,m-q}(\rho_2)+h_2)\bigcup\bigcup_{p+q=k}
(sp_{\delta , p}(\rho_1)- sp_{\pi , m-q, e}(\rho_2)+h_2),\\\end{align*}

\begin{align*}
 {\rm (iv)} &\hbox{}\bigcup_{p+q=k}(\sigma_{\pi , p,e}(\rho_1)- 
\sigma_{\delta , m-q}(\rho_2)+h_2)\bigcup\bigcup_{p+q=k}
(\sigma_{\pi ,p}(\rho_1)- \sigma_{\delta ,m-q ,e}(\rho_2)+h_2)\\
&\subseteq\sigma_{\pi ,k ,e}(\tilde{\theta})
\subseteq sp_{\delta ,k ,e} (\tilde{\theta})\subseteq\\
&\bigcup_{p+q=k}(sp_{\pi ,p ,e}
(\rho_1)- sp_{\delta ,m-q}(\rho_2)+h_2)\bigcup
\bigcup_{p+q=k}(sp_{\pi ,p}(\rho_1)- sp_{\delta ,m-q ,e}(\rho_2)
+h_2),\\\end{align*}

where $h_2$ is the character of $L_2$ considered in Theorem 3.4.\par
\indent In particular, if $X_1$ and $X_2$ are Hilbert spaces, the above 
inclusions are equalities.\par
\end{thm}

\begin{proof}
 
\indent The theorem may be proved by a similar argument 
to the one in Theorem 8.1, using Theorems 7.2 and 7.3
instead of Theorems 5.1 and 6.2.\par

\end{proof}

\indent Finally, Theorems 8.1 and 8.2 provide an extension of two 
of the main results in Chapter 3, Section 3.8 of [21] for the 
tensor product introduced in [14].\par

\bibliographystyle{amsplain}

\vskip.5cm

\noindent Enrico Boasso\par
\noindent E-mail address: enrico\_odisseo@yahoo.it

\end{document}